\documentclass[12pt]{article}

\usepackage{enumerate}
\usepackage{amsmath}
\usepackage{amssymb,latexsym}
\usepackage{amsthm}

\usepackage{anysize}
\usepackage{graphicx}
\usepackage{url}
\usepackage{color}
\usepackage[english]{babel}
\usepackage{latexsym}
\usepackage{amssymb,amsthm,amsmath}
\usepackage{graphicx,color}
\usepackage{verbatim,setspace}
\newtheorem{theorem}{Theorem}[section]
\newtheorem{proposition}[theorem]{Proposition}
\newtheorem{lemma}[theorem]{Lemma}
\newtheorem{corollary}[theorem]{Corollary}
\newtheorem{definition}[theorem]{Definition}
\newtheorem{remark}[theorem]{Remark}

\newcommand{\Laut}{{\rm Laut}}
\newcommand{\Paut}{{\rm Paut}}
\newcommand{\Gsec}{{\rm Gsec}}

\newcommand{\diag}{{\rm diag}}
\newcommand{\ord}{{\rm ord}}
\newcommand{\modd}{{\rm mod\ }}
\newcommand{\Hess}{{\rm Hess}}
\newcommand{\Form}{{\rm Form}}
\newcommand{\Ima}{{\rm Im}}

\newcommand{\C}{{\bf C}}
\newcommand{\N}{{\bf N}}
\newcommand{\QQ}{{\bf Q}}

\newcommand{\F}{{\bf F}}
\newcommand{\Z}{{\bf Z}}

\newcommand{\Ker}[1]{\mbox{${\rm Ker\ }{#1}$}}

\newcommand{\AAA}[1]{\mbox{${\bf A}^{#1}$}}
\newcommand{\PP}[1]{\mbox{${\bf P}^{#1}$}}

\newcommand{\A}[1]{\mbox{${\bf A}_{#1}$}}

\newcommand{\SSS}[1]{\mbox{${\bf S}_{#1}$}}
\newcommand{\ZZZ}[1]{\mbox{${\bf Z}_{#1}$}}

\begin{document}

\title{Projective Automorphism Groups of Nonsingular Quartic Surfaces
%\thanks{This research was  carried out with
%the support of the Italian MIUR (progetto "Strutture
%Geometriche, Algebriche e Combinatoria"), and of
%GNSAGA.}
}
%\titlerunning{Characterization of the Fermat curve}

\author{
%Giorgio Faina,
Stefano Marcugini and Fernanda Pambianco
\thanks{This research was supported in part by Ministry for Education, University and Research of Italy (MIUR) and by the Italian National Group for Algebraic and Geometric Structures and their Applications (GNSAGA - INdAM).
%The principale results have been obtained during Kaneta's visit in 2016.
}  \\
Dipartimento di Matematica e Informatica\\
Universit\`a degli Studi di Perugia\\
Perugia (Italy)\\
\{stefano.marcugini,fernanda.pambianco\}@unipg.it\\
Hitoshi Kaneta \\
Kyo-machi 77, Tsuyama, Okayama, Japan\\
nbidai@mx1.tvt.ne.jp\\
%hkaneta@marble.ocn.ne.jp\\
%Palace Mozu 301, Mozu-Ume-machi 3-34-8, \\ Kita-ku,  Sakai, 591-8032, Japan \\
%hkaneta@river.sannet.ne.jp \\ \\
}
\date{}

\maketitle

%\institute{ F. Pambianco ($\boxtimes$)\at Department of
%Mathematics and Informatics, Perugia University, Perugia,
%06123, Italy\\\email{fernanda@dmi.unipg.it}\\phone
%+39(075)5855006 \\fax~~~~~ +39(075)5855024\\
%\\
%H. Kaneta ($\boxtimes$)\at
%Palace Mozu 301, Mozu-Ume-machi 3-34-8, Kita-ku, Sakai, 591-8032, Japan\\
%\email{hkaneta@river.sannet.ne.jp}}
%\date{Received:  / Accepted: }

\begin{abstract}
For every $p\geq 5$, we determine all $\mathbf{Z}_p$-invariant nonsingular quartic surfaces in the three dimensional projective space over an algebraically closed field of characteristic zero. In some cases, we also determine
their full projective automorphism groups.
\end{abstract}


%%\setcounter{section}{-1}
%%% section 0

\section{Introduction}
Let $k$ be an algebraically closed field of characteristic zero.
In this paper we will discuss the projective automorphism groups  $\Paut(S_4)$ of
nonsingular algebraic quartic surfaces $S_4$ in $\PP{3}(k)$.
The  projective automorphism group $\Paut(S_d)$ for a nonsingular surface of degree $d$ is a finite subgroup of $PGL_n(k)$ \cite{mat} if $d\geq 3$.
For a nonsingular cubic surface $S_3$,  $|\Paut(S_3)|\leq |\Paut(V(x^3+y^3+z^3+t^3))|=648$  \cite[p.493]{dol2}. For a nonsingular quartic surface $S_4$, W. Burnside  conjectured that
$|\Paut(S_4)|\leq |\Paut(V(x^4+y^4+z^4+t^4+ 12 xyzt))|=1920$, \cite[\S 272]{bur}.
In other words $V(x^4+y^4+z^4+t^4+ 12 xyzt)$ was conjectured to be maximally symmetric among nonsingular quartic surfaces.
Decomposing the order $|\Paut(S_4)|$ into prime factors as $\Pi p^{\nu_{S_4}(p)}$ for any nonsingular quartic surface $S_4$ we have
$\nu_{S_4}(2)\leq 10$,  $\nu_{S_4}(3)\leq 8$, $\nu_{S_4}(5)\leq 1$, $\nu_{S_4}(7)\leq 1$, and $\nu_{S_4}(p)=0$ ($p\geq 11$) \cite{gor}.
Since $\nu_{S_4}(5)\leq 1$, $\nu_{S_4}(7)\leq 1$ and $V(x^4+y^4+z^4+t^4+ 12 xyzt)$  is $\Z_5$-invariant, we have faced the problem to determine all $\Z_p$-invariant
nonsingular quartic surfaces in $\PP{3}(k)$ for $p = 5,7$ up to projective equivalence.
Denote by $\Form_{n,d}$ the set of all homogeneous polynomials of degree $d$ in $k[x_1,...,x_n]$.
A non-zero (resp. nonsingular) $f\in \Form_{n,d}$ is $G$-invariant if and only if $f$ belongs to an eigenspace (resp. nonsingular eigenspace) of $G$ in $\Form_{n,d}$.

 Let $G_5$ be a subgroup of $PGL_4(k)$ isomorphic to $\Z_5$, $A_5=\diag[\varepsilon^4,\varepsilon^3,\varepsilon^2,\varepsilon]$ with $\ord(\varepsilon)=5$.
Then the nonsingular eigenspace of $\langle(A_5)\rangle$ in $\Form_{4,4}\subset k[x,y,z,t]$ is
$$\Form_{4,4}(A_5;1)=\langle x^3y,y^3z,z^3t,t^3x,x^2z^2,y^2t^2,xyzt \rangle,$$which
contains a four-dimensional nonsingular subspace $\langle x^3y+y^3z+z^3t+t^3x,x^2z^2,y^2t^2,xyzt \rangle$.
In general $G_5$ has a nonsingular eigenspace in $\Form_{4,4}$ if and only if $G_5$ is conjugate to $\langle (A_5) \rangle$.
 A nonsingular $f\in \Form_{4,4}$ is $\Z_5$-invariant if and only if $f$ is projectively equivalent to
\[
f^{\lambda,\mu,\xi}=x^3y+y^3z+z^3t+t^3x+\lambda x^2z^2+\mu y^2t^2+\xi xyzt,
\]
where $[\lambda,\mu,\xi]$ does not belong to the affine algebraic set $V_a(R)$ (Theorem 3.1,\ Theorem 3.7). Here $R\in k[u,v,w]$ takes the form
\begin{eqnarray*}
&&\{(27u^2+27v^2+6uv-w^2)-36(u+v)uvw+(u+v)w^3+\\
&&(-16u^3v^3+8u^2v^2w^2-uvw^4)\}^2-\{16-18(u+v)w+48u^2v^2+20uvw^2\}^2.
\end{eqnarray*}

 Let $G_7$ be a subgroup of $PGL_4(k)$ isomorphic to $\Z_7$, $A_7=\diag[\varepsilon^4,\varepsilon^2,\varepsilon,1]$ with $\ord(\varepsilon)=7$. Then the
nonsingular eigenspace of $\langle(A_7)\rangle$ in $\Form_{4,4}\subset k[x,y,z,t]$ is $\Form_{4,4}(A_7;1)=\langle x^3y,y^3z,z^3x,t^4,xyzt \rangle$,
which contains a two-dimensional nonsingular subspace $\langle x^3y+y^3z+z^3x+t^4,xyzt \rangle$.
In general $G_7$ has a nonsingular eigenspace if and only if $G_7$ is conjugate to $\langle (A_7) \rangle$.
A nonsingular $f\in \Form_{4,4}$ is $\Z_7$-invariant if and only if $f$ is projectively equivalent to
\[
f^\lambda = x^3y+y^3z+z^3x+t^4+\lambda xyzt,
\]
where $\lambda^4\not\not=4^4$ (Theorem 4.2).
Having determined the equations of all the  nonsingular quartic surfaces in $\PP{3}(k)$ which are $\Z_5$ or $\Z_7$-invariant, we may find Gorinov's results $\nu_{S_4}(5)\leq 1$ and $\nu_{S_4}(7)\leq 1$, but with a much simpler method (Proposition 3.2, Proposition 4.3).
Moreover, with a direct method which utilizes classification of finite subgroups of $PGL_4(k)$, we find $\nu_{S_4}(p) = 0$  for $p \geq 11$ (Theorem 5.5)

$\Z_5$ is a Sylow 5-subgroup of $\A{5}$ and $\SSS{5}$, and each group has  two kinds of faithful representations $\psi_i'$ and $\varphi_i'$ which have
a nonsingular eigenspace in $\Form_{4,4}$ \cite{dol} ($i\in [1,2]$). In \S 3 we will
describe the faithful representations $\varphi_i'$ of $\SSS{5}$ such that $\varphi_i'((12345))=(\diag[\varepsilon^4,\varepsilon^3,\varepsilon^2,\varepsilon])$
with $\ord(\varepsilon)=5$. The representation $\psi_i'$ is just the restriction $\varphi_{i}'|\A{5}$.
The nonsingular eigenspace of $\psi_1'(\A{5})$, which is two-dimensional \cite{mar}, admits bases of the form $f^{\lambda,\mu,\xi}$, while the nonsingular eigenspace of
$\psi_2'(\A{5})$, which is two-dimensional \cite{mar}, does not (Theorem 3.10,\ 3.13).
$\Z_5$ is also a Sylow 5-subgroup of $\Paut(V(x^4+y^4+z^4+t^4+ 12 xyzt))$, and the nonsingular eigenspace of $\Paut(V(x^4+y^4+z^4+t^4+ 12 xyzt))$ is one-dimensional. Hence there exists an
$S\in PGL_4(k)$ such that ${V(x^4+y^4+z^4+t^4+ 12 xyzt)}_{S^{-1}}$ takes the form $f^{\lambda,\mu,\xi}$ (Theorem 7.6).
%%%%%% march 29, 2018
Nonsingular quartic surfaces admitting a group of projective automorphisms isomorphic to $\A{5}$ or $\SSS{5}$ have been independently studied by
I.V. Dolgachev \cite{dol}.
%%%%%%

$\Z_7$ is a Sylow 7-subgroup of $PSL_2(\F_7)$, which has a nonsingular eigenspace in $\Form_{4,4}$ \cite{edg}. In \S 4 we will describe all
faithful representations  $\varphi_{0,1}$ and $\varphi_{\sqrt{2},1}$ of  $PSL_2(\F_7)$ in $PGL_4(k)$, and the bases of the form $f^\lambda$ of nonsingular
eigenspaces of $PSL_2(\F_7)$ in $\Form_{4,4}$ (Propositions 4.5,\ 4.8,\ 4.9).

In connection with Burnside's conjecture $\Paut(f)$ is specified for $f=f^{0}$ and $f=f^{0,0,0}$ (Proposition 4.12,\ Theorem 6.1).

%%% section 2
\section{Preliminaries}

We start fixing some notations.
Let $k$ be an algebraically closed field of characteristic zero, and $d,n\geq 2$  integers. Denote by $\Form_{n,d}$ the set of all homogeneous
polynomials of degree $d$ in $k[x_1,...,x_n]$, on which $GL_n(k)$ acts as $Af=f_{A}$, where $A\in GL_n(k)$
with $A^{-1}=[\alpha_{ij}]$, $f\in k[x_1,...,x_n]$, and $f_A(x_1,...,x_n)=f(\sum_{j_1=1}^n\alpha_{1j_1}x_{j_1},...,\sum_{j_n=1}^n\alpha_{nj_n}x_{j_n})$.
Two linear subspaces $U,\ V$ of $k[x_1,...,x_n]$ are conjugate if there exists an $A\in GL_n(k)$ such that $U=V_{A}$.
Let $(A)$ be the projective transformation defined by $A\in GL_n(k)$. For an $f\in \Form_{n,d}$ we call $\Paut(f)=\{(A)\in PGL_n(k):\ f_{A^{-1}}\sim f \}$
the projective automorphism group of the form $f$, and $\Laut(f)=\{A\in GL_n(k):f_{A^{-1}}=f\}$ the linear automorphism group of the form $f$.
A form $f\in \Form_{n,d}$ defines a projective algebraic set $V_p(f)=\{(x)\in \PP{n-1}(k):f(x)=0\}$ in the projective space $\PP{n-1}(k)$ over $k$.
The form $f\not=0$ is called nonsingular if the projective algebraic set $V_p(f)$  is nonsingular. Defining the projective automorphism group
$\Paut(V_p(f))$ of $V_p(f)$ to be $\{(A)\in PGL_n(k):(A)V_p(f)=V_p(f)\}$, we have $\Paut(V_p(f))=\Paut(f)$ if $f$ is irreducible, hence if $f$ is nonsingular.
A linear subspace $U$ of $\Form_{n,d}$ is called nonsingular if $U$ contains a nonsingular element.
We call a nonsingular $g\in \Form_{n,d}$ maximally symmetric if $|\Paut(f)|\leq |\Paut(g)|$ for any nonsingular $f\in \Form_{n,d}$.

We now define the concept of eigenspace in $\Form_{n,d}$. It is a tool for describing  the projective automorphism group of a form in $\Form_{n,d}$.
Let $G$ be a group. An $f\in \Form_{n,d}$ is said to be $G$-invariant if
there is an injective group homomorphism $\varphi:G\rightarrow PGL_n(k)$ such that $\Paut(f)$ contains $\varphi(G)$, provided $G$ is not a
subgroup of either $GL_n(k)$ or $PGL_n(k)$. When $G$ is a subgroup of $GL_n(k)$ (resp. subgroup of  $PGL_n(k)$), an $f\in \Form_{n,d}$ is $G$-invariant if
$\pi_n(G)\subset \Paut(f)$ (resp. $G\subset \Paut(f)$). Here $\pi_n:GL_n(k)\rightarrow PGL_n(k)$ stands for the natural projection.
Let $G$ be a finite subgroup of $PGL_n(k)$. We will denote the integer set $\{1,\dots,m\}$ by $[1,m]$. A generating section $\Gsec(G)$ of $G$ is a finite set $\{A_i\in GL_n(k):\ i\in [1,m]\}$ such that
$G=\langle (A_1),...,(A_m)\rangle$ and $d_i=\ord((A_i))=\ord(A_i)<\infty$. Denoting the cyclic group $\{a\in k^*:\ a^d=1\}$ of order $d$ by $C_d$, let
$E(\Gsec(G))=\Pi_{i=1}^m C_{d_i}$, whose element $[\rho_1,...,\rho_m]$ will be identified with a map $\rho:\Gsec(G)\rightarrow k^*$ such that
$\rho(A_i)=\rho_i$. Let
$\Form_{n,d}(\Gsec(G),\rho)=\{f\in \Form_{n,d}:\ f_{g^{-1}}=\rho(g)f\ (g\in \Gsec(G))\}$. According as $\Form_{n,d}(\Gsec(G),\rho)\not=\{0\}$ or
$\Form_{n,d}(\Gsec(G),\rho)$ contains a nonsingular $f$, it will be called an eigenspace or nonsingular eigenspace of $G$ in   $\Form_{n,d}$ with
respect to $\Gsec(G)$. It is clear that a non-zero (resp. nonsingular) $f\in \Form_{n,d}$ is $G$-invariant if and only if $f$ belongs to an eigenspace
(resp. nonsingular eigenspace) of $G$ in $\Form_{n,d}$.

We now give some lemmas that describe the subgroups of $PGL_4(k)$ isomorphic to $\ZZZ{q}$.
%%%%%%%%%%%%%%%%%%%%%%%%% ADDING MORE INFORMATION??? WHAT IS THE AIM OF 2.8 THAT IS NEW?
The symmetric group $\SSS{n}$ acts on the additive monoid $\N^n$ as $[i_1,...,i_n]\tau=[i_{\tau(1)},...,i_{\tau(n)}]$. Let $R$ be a commutative $k$-algebra.
Identifying
$[i_1,...,i_n]$ with the monomial $x_1^{i_1}\cdots x_n^{i_n}$, we see that $\SSS{n}$ acts on the $k$-algebra $R[x_1,...,x_n]$ as
\begin{eqnarray*}
&& (\sum c_{i_1...i_n}\Pi_{\ell}x_{\ell}^{i_\ell})\tau=\sum c_{i_1...i_n}\Pi_{\ell}{x_\ell}^{i_{\tau}(\ell)}
=\sum c_{i_1...i_n}\Pi_{j}{x_{\tau^{-1}(j)}}^{i_j}.
\end{eqnarray*}
Consequently  $\SSS{n}$ acts on $R[x_1,...,x_n]$ as
$(\tau f)(x_1,...,x_n)=f(x_{\tau(1)},...,x_{\tau(n)})$ by left transformations, for $(\tau f)=f \tau^{-1}$ ($f\in R[x_1,...,x_n]$).
Let $F(X_1,...,X_m)$ be a polynomial in $k[X_1,...,X_m]$ such that $\sigma F=F$ for a $\sigma\in \SSS{m}$. Then
$F(p_1,...,p_m)=F(p_{\sigma(1)},...,p_{\sigma(m)})$ for any $m$-tuple of polynomials $p_1,...,p_m$ in $k[x_1,...,x_n]$-variables,
for two polynomials coincide if and only if they coincide as $k$-valued functions.
Let $E_n=[e_1,...,e_n]$, the unit matrix in $M_n(k)$. $\SSS{n}$ acts on the set $\{e_1,...,e_n\}$, hence on $k^n$ as
$\tau(\sum_{i=1}^n a_ie_i)=\sum_{i=1}^n a_ie_{\tau(i)}$. Denote by $\hat{\tau}$ the matrix $[e_{\tau(1)},...,e_{\tau(n)}]$, associated with the linear isomorphism $\tau$ of the column vector space $k^n$ with respect to the basis $e_i$ ($i\in [1,n]$).
Thus $\hat{\sigma\tau}=\hat{\sigma}\hat{\tau}$. Since the $i$-th column of the matrix $\hat{\sigma}\diag[a_1,...,a_n]\hat{\tau}$ is equal to
$a_{\tau(i)}e_{\sigma\tau(i)}$, the equality $\hat{\tau}^{-1}\diag[a_1,...,a_n]\hat{\tau}=\diag[a_{\tau(1)},...,a_{\tau(n)}]$ holds.
%%% lemma 2.1     <Dec 4,2017>
\begin{lemma}
Let $a_i\in k$ $(i\in [1,n])$ be distinct, and $[b_1,...,b_n]\in k^n$. Then there exists an $A\in GL_n(k)$ such that
$\diag[a_1,...,a_n]\sim A^{-1}\diag[b_1,...,b_n]A$, if and only if there exists a $\tau\in \SSS{n}$ such that
$[a_1,...,a_n]\sim [b_{\tau(1)},...,b_{\tau(n)}]$.
\end{lemma}
\begin{proof}
Assume the existence of $\tau$. Then $\diag[a_1,...,a_n]\sim \hat{\tau}^{-1}\diag[b_1,...,b_n]\hat{\tau}$. \\Conversely,
assume $\diag[a_1,...,a_n]=A^{-1}BA$, where $B=\lambda\diag[b_1,...,b_n]$. The $n$-point set $\{\lambda b_1,...,\lambda b_n\}$ is nothing but the set
of eigenvalues of $A^{-1}BA$, so that it coincides with $\{a_1,...,a_n\}$, hence $a_i=\lambda b_{\tau(i)}$ for some $\tau\in \SSS{n}$.
\end{proof}

For integers $n\geq 3$ and $d\geq 2$ $\Form_{n,d}$ stands for the set of all forms of degree $d$ in the $k$-algebra $k[x_1,...,x_n]$. The usual group
homomorphism of $GL_n(k)$ into the automorphism group of $k[x_1,...,x_n]$ will be denoted by $\sigma_.$, i.e.,
$\sigma_T(f)(x_1,...,x_n)=f(\sum_{i=1}^n\tau_{1i}x_i,...,\sum_{i=1}^n \tau_{ni}x_{i})$ for $T\in GL_n(k)$ with $T^{-1}=[\tau_{ij}]$ and
$f\in k[x_1,...,x_n]$. Writing $\sigma_{A}(f)$ as $f_A$, we have $f_{AB}=(f_B)_A$.  An $f\in \Form_{n,d}\backslash\{0\}$ is nonsingular
if the projective algebraic set $V_p(f)$ in the projective space $\PP{n-1}(k)$ is nonsingular. $\Form_{n,d,nons}$ stands for the set of all nonsingular forms in $\Form_{n,d}$.
Introducing a polynomial $F(\xi,x)=\sum_{i_1+...+i_n=d}\xi_{i_1\dots i_n}x_1^{i_1}\cdots x_n^{i_n}$, any element
$f=\sum_{i_1+\cdots+i_n=d}a_{i_1\dots i_n}x_1^{i_1}\cdots x_n^{i_n}\in \Form_{n,d}$ can be written as $F(a,x)$, where $a=[a_{i_1\dots i_n}]$.
As is known, there exists a polynomial $disc_{n,d}[\xi]$, called the discriminant, such that an $f\in \Form_{n,d}\backslash\{0\}$ is nonsingular
if and only if $disc_{n,d}(a)\not=0$ \cite [Definition 5.20]{muk}. We write also $disc_{n,d}(f)$ for $disc_{n,d}(a)$.
Consequently, any linear subspace $L$ of $\Form_{n,d}$ such that $L\cap \Form_{n,d,nons}\not=\emptyset$ admits a basis consisting of elements in
$\Form_{n,d,nons}$. In fact  let $f_i$ ($i\in [1,\ell]$) form a basis of $L$, and define a polynomial $d_L(\lambda_1,...,\lambda_\ell)$
in indeterminates $\lambda_i$ by $d_L(\lambda_1,...,\lambda_\ell)=disc_{n,d}(\sum_{i=1}^\ell \lambda_if_i)$, which is a non-zero polynomial,
for $\sum_{i=1}^\ell \alpha_if_i\in \Form_{n,d,nons}$ for some $[\alpha_1,...,\alpha_\ell]\in k^\ell$.
So we can find $\ell$ linearly independent elements $\alpha_i=[\alpha_{i1},...,\alpha_{i\ell}]\in k^\ell$
outside of the affine algebraic set $V_a(d_L)$, and obtain a basis $g_i=\sum_{j=1}^\ell \alpha_{ij}f_j$ of $L$.

A character $\chi$ of $H$ is a group homomorphism of $H$ into $k^*$, and $\breve{H}$ stands for the set of characters of $H$.

A subset $S\not=\{1\}$ of a group $H$ is said to be doubly finitely generated if there is a finite subset $G\subset S$ such that
1) $S\subset \langle G\rangle$ and 2) $\ord(g)<\infty$ ($g\in G$). In this case $\langle G\rangle=\langle S\rangle$. Obviously a finite subgroup $S$ of $H$
is doubly finitely generated.

Let $S$ be a subset of $GL_n(k)$. An $f\in \Form_{n,d}$ is $S$-invariant if $f_{A}\sim f$, i.e., $ f_{A}=\lambda_A f$ for some $\lambda_A\in k^*$, for
any $A\in S$. Note that $f_A\sim f$ is equivalent to $f_{A^{-1}}\sim f$. Let $\Form_{n,d}^S=\{f\in \Form_{n,d}: f\ {\rm is\ }S\makebox{\rm -}{\rm invariant}\}$, and
$\Form_{n,d,nons}^S=\Form_{n,d}^S\cap \Form_{n,d,nons}$. We assume that $S$ is doubly finitely generated, i.e., there exists a finite subset $G$ of $S$
such that 1) $S\subset \langle G\rangle$ and 2) $\ord(g)<\infty$ for any $g\in G$. Clearly $\Form_{n,d}^{\langle G\rangle}=\Form_{n,d}^S=\Form_{n,d}^G$, hence
$\Form_{n,d,nons}^{\langle G\rangle}=\Form_{n,d,nons}^S=\Form_{n,d,nons}^G$ as well. We shall define an eigenspace and nonsingular eigenspace of
$\Form_{n,d}^G$, hence of $\Form_{n,d}^S$, as well.
For a positive integer $m$, $C_m$ stands for the cyclic group  $\{z\in k^*:\ z^m=1\}$ of order $m$. Let $E(G)=\Pi_{g\in G}C_{\ord(g)}$,
\begin{eqnarray*}
&& \Form_{n,d}(G,\rho)=\{f\in \Form_{n,d}:\ f_{g^{-1}}=\rho(g)f\ \ (g\in G)\},\\
&& \Form_{n,d,nons}(G,\rho)=\Form_{n,d}(G,\rho)\cap \Form_{n,d,nons},\\
&& E(G)^*=\{\rho\in E(G):\ \dim_k\Form_{n,d}(G,\rho)>0\},\\
&& E(G)^*_{nons}=\{\rho\in E(G)^*:\ \Form_{n,d}(G,\rho)\ {\rm contains\ a\ nonsingular\ element}\}.
\end{eqnarray*}
Similarly, for a subgroup $H$ of $GL_n(k)$ and $\chi\in \check{H}$ let
\begin{eqnarray*}
&& \Form{n,d}(H,\chi)\{f\in \Form_{n,d}:\ f_{h^{-1}}=\chi(h)f\ \ (h\in H)\},\\
&& \Form_{n,d,nons}(H,\chi)=\Form_{n,d}(H,\chi)\cap \Form_{n,d,nons},\\
&& {\check{H}}^*=\{chi\in {\check{H}}:\ \dim_k\Form_{n,d}(H,\chi)>0\},\\
&& {\check{H}}^*_{nons}=\{\chi\in {\check{H}}^*:\ \Form_{n,d}(H,\chi)\ {\rm contains\ a\ nonsingular\ element}\}.
\end{eqnarray*}
%%% lemma 2.2
\begin{lemma} Let $H$ be as above.\\
$(1)$ $\Form_{n,d}^H=\cup_{\chi\in {\check{H}}^*}\Form_{n,d}(H,\chi)$, and $\Form_{n,d,nons}^H=\cup_{\chi\in {\check{H}}^*_{nons}}\Form_{n,d,nons}(H,\chi)$. \\
$(2)$ $\langle \Form_{n,d}^H\rangle=\bigoplus_{\chi\in {\check{H}}^*}\Form_{n,d}(H,\chi)$.
$\langle \Form_{n,d,nons}^H\rangle=\bigoplus_{\chi\in {\check{H}}^*_{nons}}\Form_{n,d}(H,\chi)$.\\
In particular $\Form_{n,d}(H,\chi)\cap \Form_{n,d}(H,\chi')=\{0\}$ for distinct $\chi,\chi'\in {\check{H}}^*$. \\
\end{lemma}
\begin{proof}
(1) Let $f\in \Form_{n,d}^H\backslash\{0\}$. Then $f_{h^{-1}}=\lambda_h f$ for some $\lambda_h\in k^*$ for
each $h\in HG$ such that $\lambda_h^{\ord(h)}=1$ , for $h^{\ord(h)}=E_n$. Moreover the map $\lambda:H\rightarrow k^*$ is a character of $H$.
Thus $\Form_{n,d}^H=\cup_{\chi\in {\check{H}}^*}\Form_{n,d}(H,\chi)$.
(2) Suppose that
 $\chi_j\in {\check{H}}^*$ ($j\in [1,\ell]$) are distinct. Assuming $f=\sum_{j=1}^\ell\alpha_jf_j=0$ for some $\alpha_j\in k$ and
$f_j\in \Form_{n,d}(H,\chi_j)\backslash\{0\}$,
we will show that $\alpha_j=0$. We proceed by induction on $\ell$. If $\ell=1$, clearly $\alpha_1=0$.  Suppose that our assertion holds for $\ell-1\geq 1$ characters.
There exists $g_\ell\in H$ such that $\chi_1(g_\ell)\not=\chi_\ell(g_\ell)$, hence
the conditions $f=f_{g_\ell^{-1}}=0$ implies that $\sum_{j=1}^{\ell-1}\alpha_j(\rho_j(g_\ell)-\rho_\ell(g_\ell))f_j=0$. Repeating this argument we arrive at
$\alpha_1'f_1=0$, where $\alpha_1'=\alpha_1(\chi_1(g_\ell)-\chi_\ell(g_\ell))\cdots (\chi_1(g_2)-\chi_2(g_2))$ with $\chi_1(g_j)\not=\chi_j(g_j)$
($j\in [2,\ell]$). Thus $\alpha_1=0$. By the induction hypothesis,  $\alpha_2=...=\alpha_\ell=0$, as desired. Consequently ${\check{H}^*}$ is a finite set,
and the equality $\langle \Form_{n,d}^H\rangle=\bigoplus_{\chi\in {\check{H}}^*}\Form_{n,d}(H,\chi)$ holds.
As remarked after the introduction of $disc_{n,d}$, $\Form_{n,d}(H,\chi)$ ($\chi\in {\check{H}}^*_{nons}$) has a basis consisting of nonsingular
elements. Since $\Form_{n,d,nons}^H=\cup_{\chi\in {\check{H}}^*_{nons}}\Form_{n,d,nons}(H,\chi)$, the second equality follows from (1).
\end{proof}

%%% lemma 2.3<<lemma 2.2
\begin{lemma} Let $S\subset GL_n(k)$ be doubly finitely generated by $G$. \\
$(1)$ $\Form_{n,d}^G=\cup_{\rho\in E(G)^*}\Form_{n,d}(G,\rho)$, and $\Form_{n,d,nons}^G=\cup_{\rho\in E(G)^*_{nons}}\Form_{n,d,nons}(G,\rho)$. \\
$(2)$ $\langle \Form_{n,d}^G\rangle=\bigoplus_{\rho\in E(G)^*}\Form_{n,d}(G,\rho)$.
$\langle \Form_{n,d,nons}^G\rangle=\bigoplus_{\rho\in E(G)^*_{nons}}\Form_{n,d}(G,\rho)$.\\
In particular $\Form_{n,d}(G,\rho)\cap \Form_{n,d}(G,\rho')=\{0\}$ for distinct $\rho,\rho'\in E(G)^*$. \\
\end{lemma}
\begin{proof}
Let $H=\langle S\rangle$, and $G={g_i:i\in [1,\ell]}$. Since $H=\langle G\rangle$, the restriction map sending $\chi\in{\check{H}}$ to $\chi|_G\in E(G)$
is injective, and the restriction map sending  $\chi\in{\check{H}}^*$ to $\chi|_G\in E(G)^*$ is bijective.
 So Lemma 2.3 follows from Lemma 2.2.
\end{proof}

%%% lemma 2.4<<lemma 2.3
\begin{lemma} Suppose that both $G$ and $G'$  generate $S$ doubly finitely. Then $E(G)^*=\emptyset$ if and only if $E(G')^*=\emptyset$.
$E(G)^*_{nons}=\emptyset$ if and only if $E(G')^*_{nons}=\emptyset$. There exists a bijection $\varphi:E(G)^*\rightarrow E(G')^*$ such that
$\Form_{n,d}(G,\rho)=\Form_{n,d}(G',\varphi(\rho))$, and that $\varphi(E(G)^*_{nons})=E(G')^*_{nons}$.
\end{lemma}
\begin{proof}
Let $H=\langle G\rangle$, which is equal to $\langle G'\rangle$. It is clear that $E(G)^*=\emptyset$ if and only if ${\check{H}}^*=\emptyset$. If
${\check{H}}^*\not=\emptyset$, then we have bijections $r_G:{\check{H}}^*\rightarrow E(G)^*$ and $r_{G'}:{\check{H}}^*\rightarrow E(G')^*$, and
$\varphi={r_{G'}}^{-1}\circ r_G$ is a desired map.
\end{proof}

Suppose $S'$ is conjugate to $S$  for some $T\in GL_n(k)$, i.e. $S'=T^{-1}ST$, and $S$ is doubly finitely generated by $G$. Let $G'=T^{-1}GT$. Then the map
$\theta_T:E(G)\rightarrow E(G')$ defined by $(\theta_T(\rho))(g')=\rho(Tg'T^{-1})$ is a bijection, and $\theta_T^{-1}=\theta_{T^{-1}}$.

%%% lemma 2.5<<lemma 2.4 \ \\
\begin{lemma} Let the notations be as above. Then
$\sigma_{T^{-1}}(\Form_{n,d}(G,\rho))=\Form_{n,d}(G',\theta_T(\rho))$.
\end{lemma}

%%% def 2.6<< def 2.5
\begin{definition}
We call $\Form_{n,d}(G,\rho)$ an eigenspace or a nonsingular eigenspace of $S\subset GL_n(k)$ in $\Form_{n,d}$ with respect to the generator $G$ and
$\rho\in E(G)$ according as $\rho\in E(G)^*$ or $\rho\in E(G)^*_{nons}$.
If $G=\{A_i:i\in [1,q]\}$ with $\ord(A_i)=r_i$, then we write
$\Form_{n,d}(A_1,...,A_q;\delta_1^{j_1},...,\delta_q^{j_r})$ for $\Form_{n,d}(G,\rho)$ such that $\rho(A_i)=\delta_i^{j_i}$, where $\delta_i\in k^*$ with
$\ord(\delta_i)=r_i$.
\end{definition}

Next assume that $S$ is a subset of $PGL_n(k)$. Let $\pi_n:GL_n(k)\rightarrow PGL_n(k)=GL_n(k)/k^*E_n$ be the canonical projection, and write $(A)$ for
$\pi_n(A)$ ($A\in GL_n(k)$). An element $f\in \Form_{n,d}$ is $S$-invariant if $f_{A^{-1}}\sim f$ for every $(A)\in S$, and the set of all $S$-invariant forms
in $\Form_{n,d}$ will be denoted by $\Form_{n,d}^S$. Consequently, if $\tilde{S}\subset GL_n(k)$, then $\Form_{n,d}^{\tilde{S}}=\Form_{n,d}^{\pi_n(\tilde{S})}$.
 Let $\Form_{n,d,nons}^S=\Form_{n,d}^S\cap \Form_{n,d,nons}$. Let $\Paut(f)=\{(A)\in PGL_n(k):f_{A^{-1}}\sim f\}$ for an $f\in \Form_{n,d}$.
If $f\in \Form_{n,d,nons}$ ($d\geq 3$), $\Paut(f)$ is a finite group \cite{mat}.
 $\Paut(f)$ will be called the projective automorphism group of the form $f$. We further assume
that $S$ is doubly finitely  generated by $G$, i.e., $G$ is a finite subset of $S$ such that $\langle G\rangle\supset S$ and $\ord(g)<\infty$ ($g\in G$).
Obviously $\Form_{n,d}^S=\Form_{n,d}^{\langle G\rangle}=\Form_{n,d}^G$. Let $\eta$ be a finite section on $G$ of $\pi_n$, namely, $\eta$ is a map of $G$ into $GL_n(k)$
such that $\pi_n\circ \eta=id_G$ with $\ord(\eta(g))=\ord(g)$ for every $g\in G$. Denote by $S(G)$ the set of all order-preserving  sections on $G$ of $\pi_n$.
If $\eta\in S(G)$, then $S(G)=\{\rho\eta:\rho\in E(G)\}$, where an element $\rho\in E(G)=\Pi_{g\in G}C_{\ord(g)}$ is a map of $G$ into $k^*$ such that
$\rho(g)\in C_{\ord(g)}$. Let $\Form_{n,d}(G,\eta,\rho)=\{f\in \Form_{n,d}:f_{\eta(g)^{-1}}=\rho(g)f \}$ for $\eta\in S(G)$ and $\rho\in E(G)$.
It is clear that $\Form_{n,d}^S=\Form_{n,d}^G=\Form_{n,d}^{\eta(G)}$, and that $\Form_{n,d}(G,\eta,\rho)=\Form_{n,d}(\eta(G),\rho\circ \pi_n)$. Moreover,
if $\eta'=\lambda \eta$ ($\lambda\in E(G)$), then $\Form_{n,d}(G,\eta,\rho)=\Form_{n,d}(G,\eta',\lambda^{d}\rho)$. Define
\begin{eqnarray*}
E(G,\eta)^*&=&\{\rho\in E(G):\Form_{n,d}(G,\eta,\rho)\not=\{0\}\},\\
E(G,\eta)_{nons}^*&=&\{\rho\in E(G):\Form_{n,d}(G,\eta,\rho)\cap \Form_{n,d,nons}\not=\emptyset\}.
\end{eqnarray*}
We introduce the subgroup $H=\langle \eta(G),\ k^*E_n\rangle$ of $GL_n(k)$, which  depends neither $\eta\in S(G)$ nor the generator $G$.
Obviously $\Form_{n,d}^S=\Form_{n,d}^H$, and $E(G,\eta)^*\not=\emptyset$ if and only if ${\check{H}}^*\not=\emptyset$. Moreover, since
$f_{\alpha^{-1}E_n}=\alpha^d f$ for any $f\in \Form_{n,d}$, the restriction map $r_{\eta(G)}:{\check{H}}^*\rightarrow E(G,\eta)^*$ is a bijection.
%%% def 2.7<< def 2.6
\begin{definition}
We call $\Form_{n,d}(G,\eta,\rho)$ for $\rho\in E(G,\eta)^*$ $(resp.\ \rho\in E(G,\eta)_{nons}^*)$ eigenspace (resp. nonsingular  eigenspace )
of $S$ in $\Form_{n,d}$ with respect to the generator $G$, the section $\eta\in S(G)$ and $\rho\in E(G,\eta)$.
\end{definition}
$\langle \Form_{n,d}^S\rangle$ (resp. $\langle \Form_{n,d,nons}^S\rangle$ ) is a direct sum of eigenspaces (resp. nonsingular eigenspaces) of
$S$ in $\Form_{n,d}$ by Lemma 2.2.

Assume that $S'=(T^{-1})S(T)$ and that $S\subset PGL_n(n)$ is doubly finitely generated by a subset $G$. Then $S'$ is doubly finitely generated by
$(T^{-1})G(T)$, and we have bijections $s_T:S(G)\rightarrow S(G')$ and $\theta_T:E(G)\rightarrow E(G')$ such that
\begin{eqnarray*}
(s_T(\eta))(g')&=&T^{-1}\eta((T)g'(T^{-1}))T\ \ (\eta\in S(G),\ g'\in G'),\\
(\theta_T(\rho))(g')&=&\rho((T)g'(T^{-1}))\ \ \ \ \ (\rho\in E(G),\ g'\in G').
\end{eqnarray*}
It is immediate that $s_{T}^{-1}=s_{T^{-1}}$ and $\theta_T^{-1}=\theta_{T^{-1}}$. If $f\in \Form_{n,d}(G,\eta,\rho)$, $\eta'=s_T(\eta)$, and
$g'=(T)^{-1}g(T)$ ($g\in G$), then
\[
 (f_{T^{-1}})_{\eta'(g')^{-1}}=(f_{T^{-1}})_{T^{-1}\eta(g)^{-1}T}=f_{T^{-1}\eta(g)^{-1}}=\rho(g)f_{T^{-1}}=(\theta_T(\rho))(g')f_{T^{-1}}.
\]
So we have
%%% lemma 2.8<<lemma 2.7
\begin{lemma} $\sigma_{T^{-1}}(\Form_{n,d}(G,\eta,\rho))=\Form_{n,d}(G',{s_T(\eta)},\theta_T(\rho))$.
\end{lemma}

For a representation $\psi$ of a group $H$ in $GL_n(k)$ or $PGL_n(k)$, $\Form_{n,d}^\psi$ stands for $\Form_{n,d}^{\psi(H)}$. The next proposition
is concerned with the tangent cone \cite[chapter 2,\ \S 1.5]{sha}.
%% When $\psi$ is injective, we  use $\Form_{n,d}^H$ for $\Form_{n,d}^{\psi(H)}$. \\

%%%theorem 2.9 <<proposition 2.8
\begin{theorem} Let $d,\ e$ be positive integers such that $d>e$, and $A=[a_{ij}]\in GL_n(k)$ such that $A^{-1}=[\alpha_{ij}]$ with $\alpha_{\ell n}=1$
 for some $\ell\in [1,n]$. For an $f\in \Form_{n,d}$ of the form $\sum_{r=0}^{d-e}f_{d-r}(x_1,...,x_{n-1})x_n^{r}$ $(f_j\in \Form_{n-1,j})$ with
$f_{e}\not=0$ define $g$ to be $f_{A^{-1}}$. Then $P=(0,...,0,1)\in V_p(f)$ and $Q=(A)^{-1}P=(\alpha_{1n},...,\alpha_{nn})\in V_p(g)$, and the tangent cone
$V_a(f')$
to $V_p(f)$ at $P$ and the tangent cone $V_a(g')$ to $V_p(g)$ at $Q$ are isomorphic in the sense that there exists an $A'\in GL_{n-1}(k)$ such that
$g'=f'_{A'}$.
\end{theorem}
\begin{proof}
Since $f(y_1,...,y_{n-1},1)=f_e(y_1,...,y_{n-1})+\cdots +f_d(y_1,...,y_{n-1})$, we see $V_a(f')=V_a(f_e)$. Let $X_i=\sum_{j=1}^n a_{ij}x_j$.
If $x_\ell=1=\alpha_{\ell n}$ and $x_{j}=y_j+\alpha_{jn}$ ($j\not=\ell$), then for any $i\in [1,n]$
\[
 X_i=\sum_{j\not=\ell}a_{ij}(y_j+\alpha_{jn})+a_{i \ell}\alpha_{\ell n}=\delta_{i,n}+\sum_{j\not=\ell}a_{ij}y_j=\delta_{i,n} +Y_i
\]
so that $g(y_1+\alpha_{1n},...,y_{\ell-1}+\alpha_{\ell-1 n},1,y_{\ell+1}+\alpha_{\ell+1 n},...,y_n+\alpha_{nn})$ can be written as
\begin{eqnarray*}
&& f_e(Y_1,...,Y_{n-1})(1+Y_n)^{d-e}+\cdots + f_d(Y_1,...,Y_{n-1})=f_e(Y_1,...,Y_{n-1})+\sum_{j=e+1}^d g_j(\check{y}),
\end{eqnarray*}
where $g_j\in \Form_{n-1,j}$ and $\check{y}=[y_1,...,y_{\ell-1},y_{\ell+1},...,y_n]$.  Therefore, writing the polynomial $f_e(Y_1,...,Y_{n-1})$ in $\check{y}$
as $f'_e$,  we have $V_a(g')=V_a(f_e')$. Note that linear forms $Y_1,...,Y_{n-1}$ in $\check{y}$ are linearly independent, for the coefficients $a_{ij}$ of
$y_j$ ($j\not=\ell$) in $Y_i$ ($i\in [1,n-1]$) and $a_{n\ell}$ determine the $[\ell,n]$ component of the adjugate matrix $\hat{A}=(\det A)A^{-1} $ of $A$.
\end{proof}

Let $d,\ n-1,\ r\geq 1$ be integers and $\delta\in k^*$ with $\ord(\delta)=r$. If $A\in GL_n(k)$ is diagonal and of order $r$, it takes the form
$\diag[\delta^{i_1},\dots,\delta^{i_n}]$ for some integers $i_1,\dots,i_n\in [0,r-1]$, so that $M_{A^{-1}}=\delta^jM$ for any monomial $M\in k[x]$ with
certain $j\in [0,r-1]$. Therefore the set ${\cal M}_d$ of all monomials of degree $d$ is the disjoint union of
${\cal M}_d(j)=\{M\in {\cal M}_d : M_{A^{-1}}=\delta^j M\}$ ($j\in [0,r-1]$). By Lemma 2.2  we have

%%% lemma 2.10<<lemma 2.9 <<lemma 1.1
\begin{lemma} Let $A$ and ${\cal M}_d(j)$ be as above. Then an $f\in Form_{n,d}\backslash\{0\}$ satisfies $f_{A^{-1}}\sim f$,
if and only if $f$ is a non-zero linear combination of monomials in ${\cal M}_d(j)$ for some $j$ such that ${\cal M}_d(j)\not=\emptyset$.
\end{lemma}

Let $A$ be as in Lemma 2.9 and $M_{A^{-1}}=\delta^j M$, where $M$ is a monomial in $k[x_1,\dots,x_n]$ and $j\in [0,r-1]$. We call $j$, which is considered to
belong to $\Z/r\Z$, the index of $M$ for $A$. Monomials $x_ix_j^{d-1}$ ($i,j\in [1,n]$) are called singularity-checking monomials of  $d$-forms
in $k[x_1,...,x_n]$.
If $d\geq 3$, there exist $n^2$ such monomials.
Define $\ell_{ij}\in \Z/r\Z$ by ${x_ix_j^{d-1}}_{A^{-1}}=\delta^{\ell_{ij}}x_ix_j^{d-1}$ ($i,j\in [1,n]$), and let
$I_{x_j}=I_{x_j}(A)=\{\ell_{ij}:i\in [1,n]\}$, $I_{x_1,...,x_n}=I_{x_1,...,x_n}(A)=\cap_{i=1}^n I_{x_i}$.
Obviously any $(d-1)$-form not containing ${x_j}^{d-1}$ vanishes at $e'_j$, where $e'_j$
stands for the $j$-th row of the unit matrix $E_n\in GL_n(k)$. Let $d\geq 2$, and $i,j\in[1,n]$. A monomial $M\in k[x_1,...,x_n]$ satisfies
both $M_{x_i}\sim x_j^{d-1}$ and $M_{x_i}\not=0$ if and only if $M=x_ix_j^{d-1}$. Therefore the following lemma holds.
%%%  lemma 2.11 << lemma 2.10 <<lemma 1.2
\begin{lemma}
Let $d\geq 2$, and assume that  $f\in \Form_{n,d}\backslash\{0\}$ contains none of $n$ forms $x_ix_j^{d-1}$ $(i\in [1,n])$ for some $j\in [1,n]$,
then $f_{x_i}(e'_j)=0$ for any $i$, namely the projective algebraic set $V_p(f)$ is singular at $(e'_j)$.
In particular if $I_{x_1,...,x_n}(A)=\emptyset$, then any $f\in \Form_{n,d}^{\{A\}}\backslash\{0\}$ is singular.
\end{lemma}
\begin{proof}
In order to show the latter part, let $f_{A^{-1}}=\delta^\ell f$ ($\ell\in \Z/r\Z$). Since $\bigcup_{j=1}^n I_{x_i}^c=\Z/r\Z$, $I_{x_j}^c\ni\ell$ for some
$j$, hence $V_p(f)$ is singular at $(e_j')$.
\end{proof}

%%%  corollary 2.12 << corolllary 2.11 << corollary 1.3
\begin{corollary}
Let $d\geq 3$ and $f\in \Form_{n,d}\backslash\{0\}$. If $f$ contains at most $n-1$ singularity-checking monomials of degree $d$, then
the projective algebraic set $V_p(f)$ is singular at $(e'_\ell)\in \PP{n-1}(k)$ for some $\ell\in [1,n]$.
\end{corollary}
\begin{proof}
Let $X=[x_{ij}]\in M_{n,n}(k[x_1,...,x_n])$ such that $x_{ij}=x_ix_j^{d-1}$ ($i,j\in [1,n]$). Then the matrix components of $X$ are distinct. By the assumption
$f$ contains no elements of a certain column of $X$, the $\ell$-th column say. Thus $V_p(f)$ is singular at $(e'_\ell)$.
\end{proof}

Let $p$ be a prime, $a$ a positive integer, $q=p^a$, $\varepsilon\in k^*$ with $\ord(\varepsilon)=q$. A subgroup $G$ of $PGL_4(k)$ isomorphic to
$\ZZZ{q}$, i.e. the additive group of the ring $\Z/q\Z$. $G$ has a generator $(A)$, where $A\in GL_4(k)$ is of order $q$. We may assume
$A=\diag[\varepsilon^i,\varepsilon^j,\varepsilon^\ell,\varepsilon^m]$ with $\gcd(i,j,\ell,m,q)=1$ ($i,j,\ell,m\in [0,q-1]$). Since
$(A)=(\diag[1,\varepsilon^{(j-i)},\varepsilon^{(\ell-i)},\varepsilon^{(m-i)}])$ in $PGL_4(k)$, we may assume
%$A=\diag[1,\varepsilon^i,\varepsilon^j,\varepsilon^\ell]$ with $0\leq i\leq j\leq \ell<q$ and $\gcd(i,j,\ell,q)=1$. We may assume $c=1$, for we may replace
%$\varepsilon$ by $\varepsilon^c$ because of $\gcd(c,q)=1$.
%If $i=j=0$, then $A=D_{0}$, where $D_{0}=\diag[1,1,\varepsilon,1]$. If $i=0<j=\ell$, then $A=D_1$, where
%$D_1=\diag[1,1,\varepsilon,\varepsilon]$. If $i=0<j<\ell$, we may assume $A=D_\ell$, where $D_\ell=\diag[1,1,\varepsilon,\varepsilon^\ell]$ with
%$1<\ell<q$, for either $\varepsilon^j$ or $\varepsilon^\ell$ is of order $q$. Suppose $i>0$. Let $I=\{i,j,\ell\}$. If $|I|=1$, then $i=j=\ell=1$, hence
%$G$ is conjugate to $\langle (D_{0})\rangle$, so we may assume $A=D_{0}$. If $|I|=2$, we may assume $A=D_\ell$. Finally suppose $|I|=3$.
%Then we may assume $A=D_{j,\ell}$, where $D_{j,\ell}=\diag[1,\varepsilon,\varepsilon^j,\varepsilon^\ell]$ ($1<j<\ell<q$), for
%one of $\varepsilon^s$ ($s\in I$) is of order $q$. Thus we have shown
$A=\diag[1,\varepsilon^{i_1},\varepsilon^{i_2},\varepsilon^{i_3}]$ with $0\leq i_1\leq i_2\leq i_3<q$ and $\gcd(i_1,i_2,i_3,q)=1$, i.e.,
$\gcd(i_m,p)=1$ for some $m\in [1,3]$. There exists uniquely $t\in [1,3]$ such that $i_t>0$ and $i_{t-1}=0$. Note that
we may assume $i_t=1$, for $A$ takes the form $\diag[1,{\varepsilon'}^{j_1},{\varepsilon'}^{j_2},{\varepsilon'}^{j_3}]$ with $j_m=1$
so that $A$ is conjugate to $A'=\diag[1,...,1,\varepsilon',{\varepsilon'}^{i_2},...,{\varepsilon'}^{i_s}]$, where $1\leq i_2\leq ...\leq i_s<q$ and
${\varepsilon'}=\varepsilon^{i_m}$. If $t=3$, then $G$ is conjugate to $\langle(D_0)\rangle$, where
$D_0=\diag[1,1,\varepsilon,1]$. If $t=2$, and $i_2=i_3$, then $G$ is conjugate to $\langle(D_1)\rangle$, where $D_1=\diag[1,1,\varepsilon,\varepsilon]$.
If $t=2$ and $i_2<i_3=\ell$, hence $q\geq 3$ , then $G$ is conjugate to $\langle(D_\ell)\rangle$, where
$D_\ell=\diag[1,1,\varepsilon,\varepsilon^\ell]$ ($\ell\in [2,q-1]$). Suppose $t=1$, and let $i_1=i$, $i_2=j$, $i_3=\ell$ and $I=\{i,j,\ell\}$. If $|I|=1$,
then $i=j=\ell=1$, hence $G$ is conjugate to $\langle (D_{0})\rangle$.
If $|I|=2$, hence $q\geq 3$, then $A=\diag[1,\varepsilon,\varepsilon,\varepsilon^\ell]$  or
$A=\diag[1,\varepsilon,\varepsilon^\ell,\varepsilon^\ell]$ ($\ell\in [2,q-1]$). In the first case $G$ is conjugate to $\langle(D_\ell)\rangle$ for
some $\ell\in [2,q-1]$, because $(A)$ is conjugate to $(\diag[1,1,\varepsilon^{-1},\varepsilon^{-\ell'}])$, where $\ell'=q-\ell+1\in [2,q-1]$.
In the second case, if $\gcd(\ell,p)=1$, $G$ is conjugate to $\langle(D_\ell)\rangle$ for
some $\ell\in [2,q-1]$ (replace $\varepsilon$ by $\varepsilon^{-\ell}$), while if  $\gcd(\ell,p)=p$, hence $a\geq 2$, $G$ is conjugate to
$\langle(B_{j})\rangle$, where $B_{j}=\diag[1,\varepsilon,\varepsilon^{j},\varepsilon^{j}]$ for $j=\ell\in [2,q-1]$ which is divisible by $p$. However,
$\langle(B_{j})\rangle$ is conjugate to $\langle (\diag[1,1,\varepsilon',{\varepsilon'}^{\ell}])\rangle$, where $\varepsilon'=\varepsilon^{1-j}$
and ${\varepsilon'}^{\ell}=\varepsilon^{-j}$ ($\ell\in [2,q-1]$).
Finally suppose $|I|=3$, hence $q\geq 4$. Then $G$ is conjugate to $\langle(D_{j,\ell})\rangle$, where
$D_{j,\ell}=\diag[1,\varepsilon,\varepsilon^j,\varepsilon^\ell]$ with $1<j<\ell<q$. Thus we have shown

%%% lemma 2.13 << lemma 2.12 << lemma 1.4
\begin{lemma} Let $a$ be a positive integer, $p$ a prime, $q=p^a\geq 4$, and $G$  a subgroup of $PGL_4(k)$ isomorphic to $\ZZZ{q}$, and let
 $D_\ell$ $(0\leq\ell<q)$, $D_{j,\ell}$ $(1<j<\ell<q)$
be as above. Then $G$ is conjugate to one of the cyclic groups $\langle(D_\ell)\rangle$ $(\ell\in [0,q-1])$
and $\langle(D_{j,\ell})\rangle$ $(1<j<\ell<q)$.
\end{lemma}

Assume $q=5$. Since $D_2^2=\diag[1,1,\varepsilon^2,\varepsilon]$, $\langle(D_2)\rangle$ and $\langle(D_3)\rangle$ are conjugate in $PGL_4(k)$, namely
$\langle(D_2)\rangle\cong\langle(D_3)\rangle$. The equalities
\[
 D_{3,4}^2=\diag[1,\varepsilon^2,\varepsilon,\varepsilon^3],\ \ {\rm and}\ \ D_{2,4}=\varepsilon^4\diag[\varepsilon,\varepsilon^2,\varepsilon^3,1]
\]
imply $\langle(D_{3,4})\rangle\cong \langle(D_{2,3})\rangle$ and $\langle(D_{2,4})\rangle\cong\langle(D_{2,3})\rangle$, respectively. Therefore we have
%%% lemma 2.14 << lemma 2.13 << lemma 1.5
\begin{lemma}
A subgroup of $PGL_4(k)$ isomorphic to $\ZZZ{5}$ is conjugate to one of the five cyclic subgroups $\langle(D_{0})\rangle$, $\langle(D_1)\rangle$,
$\langle(D_2)\rangle$, $\langle(D_4)\rangle$ and $\langle(D_{2,3})\rangle$.
\end{lemma}

Finally assume $q=7$. The equalities $D_2^4=\diag[1,1,\varepsilon^4,\varepsilon]$ and $D_3^5=\diag[1,1,\varepsilon^5,\varepsilon]$ imply
$\langle(D_4)\rangle\cong \langle(D_2)\rangle$ and $\langle(D_5)\rangle\cong \langle(D_3)\rangle$, respectively. Since
$D_{2,4}^5=\varepsilon^5\diag[\varepsilon^2,1,\varepsilon^5,\varepsilon]$ and $D_{2,6}=\varepsilon^6\diag[\varepsilon,\varepsilon^2,\varepsilon^3,1]$, we have
$\langle(D_{2,5})\rangle\cong \langle(D_{2,4})\rangle$ and $\langle(D_{2,6})\rangle\cong\langle(D_{2,3})\rangle$. Moreover, the equalities
\begin{eqnarray*}
&&D_{3,4}^2=\diag[1,\varepsilon^2,\varepsilon^6,\varepsilon],\ D_{3,5}^3=\diag[1,\varepsilon^3,\varepsilon^2,\varepsilon],\
  D_{3,6}^5=\diag[1,\varepsilon^5,\varepsilon,\varepsilon^2],\\
&&D_{4,5}^2=\diag[1,\varepsilon^2,\varepsilon,\varepsilon^3],\ D_{4,6}^2=\diag[1,\varepsilon^2,\varepsilon,\varepsilon^5]
\end{eqnarray*}
imply, respectively,
\begin{eqnarray*}
&&\langle(D_{3,4})\rangle\cong\langle(D_{2,3})\rangle,\ \langle(D_{3,5})\rangle\cong\langle(D_{2,3})\rangle,\
  \langle(D_{3,6})\rangle\cong\langle(D_{2,5})\rangle,\\
&&\langle(D_{4,5})\rangle\cong\langle(D_{2,3})\rangle,\ \langle(D_{4,6})\rangle\cong\langle(D_{2,5})\rangle.
\end{eqnarray*}
Since $D_{5,6}^3=\diag[1,\varepsilon^3,\varepsilon,\varepsilon^4]$, we also have
$\langle(D_{5,6})\cong\langle(D_{3,4})\rangle\cong\langle(D_{2,6})\rangle$. Thus we arrive at
%%% lemma 2.15 << lemma 2.14 << lemma 1.6
\begin{lemma}
A subgroup of $PGL_4(k)$ isomorphic to $\ZZZ{7}$ is conjugate to one of the seven cyclic groups $\langle(D_{0})\rangle$, $\langle(D_1)\rangle$,
$\langle(D_2)\rangle$, $\langle(D_3)\rangle$, $\langle(D_6)\rangle$, $\langle(D_{2,3})\rangle$ and $\langle(D_{2,4})\rangle$.
\end{lemma}

A Kummer surface is projectively isomorphic to a quartic surface $V_p(c_0g_0+2c_1g_1+2c_2g_2+2c_3g_3+4c_4g_4)$ \cite[Theorem 10.3.18]{dol2}, where $c_j\in k$ satisfying
$c_0(c_0^2-c_1^2-c_2^2-c_3^2+c_4^2)+2c_1c_2c_3=0$, and $g_j\in k[x,y,z,t]$ with the form
\begin{eqnarray*}
&&g_0=x^4+y^4+z^4+t^4,\ g_4=xyzt,\\
&&g_1=x^2y^2+z^2t^2,\ g_2=x^2z^2+y^2t^2,\ g_3=x^2t^2+y^2z^2.
\end{eqnarray*}
In order to characterize the linear subspace $\langle g_0,...,g_4\rangle\subset \Form_{4,4}$ as an eigenspace we define a subgroup $G$ of
$PGL_4(k)$ of order $16$ to be $\langle (A_1),...,(A_4)\rangle$, where
\[
 A_1=\diag[-1,1,-1,1],\ A_2=\diag[1,-1,-1,1],\ A_3=\hat{(12)(34)},\ A_4=\hat{(13)(24)}.
\]
Note that $((12)(34))((13)(24))=((13)(24))((12)(34))=(14)(23)$.
Let $\eta:G\rightarrow GL_4(k)$ be the order-preserving section on $G$ of $\pi_4$ such that $\eta((A_i))=A_i$ ($i\in [1,4]$), and
$\rho:\{A_1,...,A_4\}\rightarrow \langle -1\rangle$ such that $\rho(A_i)=1$ for every $i$.
%%% proposition 2.16 << proposition 2.15
\begin{proposition} Let the notations be as above. Then $\langle g_0,...,g_4\rangle=\Form_{4,4}(G,\eta,\rho)$.
\end{proposition}
\begin{proof}
We classify all monomials in $\Form_{4,4}$ into five classes ${\cal M}(4)$,${\cal M}(3,1)$,${\cal M}(2,2)$,${\cal M}(2,1,1)$, and ${\cal M}(1,1,1,1)$,
where ${\cal M}(1,1,1,1)=\{xyzt\}$, and
\begin{eqnarray*}
&&{\cal M}(4)=\{x^4,y^4,z^4,t^4\},\ {\cal M}(2,2)=\{x^2y^2,x^2z^2,x^2t^2,y^2z^2,y^2t^2,z^2t^2\},\\
&&{\cal M}(3,1)=\{x^3y,x^3z,x^3t,y^3x,y^3z,y^3t,z^3x,z^3y,z^3t,t^3x,t^3y,t^3z\},\\
&&{\cal M}(2,1,1)=\{x^2yz,x^2yt,x^2zt,y^2xz,y^2xt,y^2zt,z^2xy,z^2xt,z^2yt,t^2xy,t^2xz,t^2yz\}.
\end{eqnarray*}
Let $V_1=\langle {\cal M}(4)\rangle $, $V_2=\langle {\cal M}(3,1)\rangle$, $V_3=\langle {\cal M}(2,2)\rangle$, $V_4=\langle {\cal M}(2,1,1)$, and
$V_5={\cal M}(1,1,1,1)$, so that $\Form_{4,4}=V_1\oplus V_2\oplus V_3\oplus V_4\oplus V_5 $, namely $f\in \Form_{4,4}$ takes the form
$f_1+\cdots +f_5$, where $f_i\in V_i$ ($i\in [1,5]$). Since ${V_i}_{{A_j}^{-1}}=V_i$ ($i\in [1,5]$, $j\in [1,4]$), an $f\in \Form_{4,4}$ belongs to
$\Form_{4,4}(G,\eta,\rho)$ if and only if $f_i\in \Form_{4,4}(G,\eta,\rho)$ ($i\in [1,5]$).
It is clear that $g_j$ are linearly independent and they belong to $\Form_{4,4}(G,\eta,\rho)$. Let $f\in \Form_{4,4}(G,\eta,\rho)$.
Assume $f$ contains $x^4$, so that $f=x^4+...$. Since $f_{{A_j}^{-1}}=f$, $f_1=g_0$.
Since no monomial $M$ in ${\cal M}(3,1)\cup {\cal M}(2,1,1)$ satisfies $M_{{A_1}^{-1}}=M_{{A_2}^{-1}}=M$, $f$ contains no
$M\in {\cal M}(3,1)\cup {\cal M}(2,1,1)$. Assume $f=ax^2y^2+bx^2z^2+cx^2t^2+...$ $(a,b,c\in k)$.  Then $f_3=a(x^2y^2+z^2t^2)+b(x^2z^2+y^2t^2)+c(x^2t^2+y^2z^2)$
so that $f_3\in \langle g_1,g_2,g_3\rangle$. Now it follows that $f\in \langle g_0,...,g_4\rangle$.
\end{proof}
\section{ $\ZZZ{5}$-invariant nonsingular quartic forms}
We shall describe $\ZZZ{5}$-invariant nonsingular quartic forms in $k[x,y,z,t]$.
Let $\varepsilon\in k^*$ be of order $5$, the diagonal matrices $D_j$ and $D_{j,\ell}$ be as in the previous section, and
$A_0=D_{0}$, $A_1=D_1$, $A_2=D_2$, $A_3=D_4$ and $A_4=D_{2,3}$. Let $f^{[i,j]}\in \Form_{4,4}(A_i;\varepsilon^j)$, i.e.,
$f_{A^{-1}_i}^{[i,j]}=\varepsilon^j f^{[i,j]}$ ($i,j\in [0,4]$). Computing indices of the singularity-checking quartic monomials for $A_i$,
we obtain the following table.
{\scriptsize
\begin{eqnarray*}
\begin{array}{lcccccccccccccccc}
         &x^4 &x^3y &x^3z &x^3t &y^3x &y^4 &y^3z &y^3t &z^3x &z^3y &z^4 &z^3t &t^3x &t^3y &t^3z &t^4\\
    D_{0}&0   &0    &1    &0    &0    &0   &1    &0    &3    &3    &4   &3    &0    &0    &1    &0   \\
    D_{1}&0   &0    &1    &1    &0    &0   &1    &1    &3    &3    &4   &4    &3    &3    &4    &4   \\
    D_{2}&0   &0    &1    &2    &0    &0   &1    &2    &3    &3    &4   &0    &1    &1    &2    &3   \\
    D_{4}&0   &0    &1    &4    &0    &0   &1    &4    &3    &3    &4   &3    &2    &2    &3    &1   \\
  D_{2,3}&0   &1    &2    &3    &3    &4   &0    &1    &1    &2    &3   &4    &4    &0    &1    &2
  \end{array}
\end{eqnarray*}
}
Consequently $f^{[i,j]}$ is singular unless $[i,j]=[4,1]$.
Since $\Form_{4,4}(A_4;\varepsilon)$ contains a nonsingular element $x^3y+y^3t+t^3z+z^3x$, $\Form_{4,4}(A_4,\varepsilon)$  is the only one nonsingular
eigenspace for ${A_4}$ in $\Form_{4,4}$.
Let $T=[e_1,e_2,e_4,e_3]$, where $e_i$ is the $i$-th column vector of the unit matrix $E_4\in GL_4(k)$. Now $\sigma_{T^{-1}}(\Form_{4,4}(A_4;\varepsilon))=
\Form_{4,4}(A_4';\varepsilon)=\Form_{4,4}(D;1)$, where $A'_4=\diag[1,\varepsilon,\varepsilon^3,\varepsilon^2]$, and
$D=\varepsilon A_4'$, which is equal to\\
$$\langle x^3y,y^3z,z^3t,t^3x,x^2z^2,y^2t^2,xyzt\rangle,$$
 and contains
the four-dimensional subspace\\ $\langle g_0,g_1,g_2,g_3\rangle$, where $g_0=x^3y+y^3z+z^3t+t^3x$, $g_1=x^2z^2$, $g_2=y^2t^2$ and $g_3=xyzt$.
We consider the quartic form $f^{\lambda,\mu,\xi}=g_0+\lambda g_1+ \mu g_2+\xi g_3$, for any $f\in \langle g_1,g_2,g_3\rangle$ is singular.
Note that for any nonsingular quartic form $f\in \Form_{4,4}(D;1)$
there exists a diagonal matrix $T\in GL_4(k)$ such that $f_{T^-1}=f^{\lambda,\mu,\xi}$ for some $[\lambda,\mu,\xi]\in k^3$, that is, $f$ is projectively
equivalent to $f^{\lambda,\mu,\xi}$.

%%% theorem 3.1
\begin{theorem}  For a nonsingular quartic form $f(x,y,z,t)$  $|\Paut(f)|$ is divisible by $5$ if and only if  $f$ is projectively equivalent to
$f^{\lambda,\mu,\xi}$ for some $[\lambda,\mu,\xi]\in k^3$.
\end{theorem}
\begin{proof}
Let $D=\diag[\varepsilon,\varepsilon^2,\varepsilon^4,\varepsilon^3]$. First assume $5|\ |\Paut(f)|$. Then $\Paut(f)$ contains a cyclic group $C_5$ of order
5 that is conjugate with $\langle (D)\rangle$, hence there exists  a $T\in GL_4(k)$ such that $(TDT^{-1})\in C_5$. Since $f_{(TDT^{-1})^{-1}}\sim f$,
$g=f_{T^{-1}}$ is nonsingular and satisfies $g_{D^{-1}}=f_{T^{-1}(TD^{-1}T^{-1})}\sim f_{T^{-1}}=g$ so that $g\in \Form_{4,4}(D;1)$. Conversely,
assume $f_{T^{-1}}=f^{\lambda\mu,\xi}$ for some $[\lambda,\mu,\xi]\in k^3$ and $T\in GL_4(k)$. Then $(f_{T^{-1}})_{D^{-1}}=f_{T^{-1}}$ so that
$f_{(TDT^{-1})^{-1}}=f_{TD^{-1}T^{-1}}=(f_{D^{-1}T^{-1}})_T=f$. Thus $(TDT^{-1})\in \Paut(f)$.
\end{proof}

%%% propositopn 3.2
\begin{proposition}
Let $G$ be the projective automorphism group of a nonsingular quartic form $f(x,y,z,t)$. In the decomposition $\Pi p^{\nu(p)}$ of $|G|$ into prime factors
it holds that $\nu(5)\leq 1$.
\end{proposition}
\begin{proof}
As is known, the projective automorphism group of a nonsingular $d$-form is a finite group \cite{mat} if $d\geq 3$. Let
$G=\Paut(f)$ ($f\in \Form_{4,4,nons}$), $p=5$ and $c=\nu(p)$. Suppose $c>1$. Then G contains a subgroup $H$ of order $p^2$, which is isomorphic to
$\Z_p\times \Z_p$ by Theorem 5.1, i.e., $H=\langle (A),(B)\rangle$ with $\ord(A)=\ord(B)=5$ and $(A)(B)=(B)(A)$. We may assume
$A=\diag[1,\varepsilon,\varepsilon^3,\varepsilon^2]$ and $f=f^{\lambda,\mu,\xi}$.
We can directly see that any $X\in GL_4(k)$ such that $AX\sim XA$, namely
$(A)(X)=(X)(A)$ in $PGL_4(k)$, is diagonal. Consequently, the conditions $B=\diag[\alpha,\beta,\gamma,1]$ and $f_{B^{-1}}\sim f$ yield
$\alpha^2\beta=\alpha^{-1}\beta^3\gamma=\alpha^{-1}\gamma^3=1$, hence $\alpha=\gamma^3$, $\beta=\gamma^{-6}$ and $\gamma^{20}=1$, i.e.,
$B=\diag[\gamma^3,\gamma^{-6},\gamma,1]$ with $\gamma^{20}=1$. Thus $|H|\leq 20<p^2$, a contradiction.
\end{proof}

We shall find  a condition for the quartic form $f=f^{\lambda,\mu,\xi}\in \Form_{4,4}$ to be singular.  Note that if every partial derivative
$f_{x_i}(x_1,...,x_4)$ vanishes at any $[x,y,z,t]\in k^4\backslash\{0\}$ such that $f(x,y,z,t)$, then $xyzt\not=0$. To verify this, suppose $t=0$. Then
if $xz=0$, then  $x=y=z=0$. If $xz\not=0$, then the equalities $f_{x_1}=f_{x_3}=0$ imply $y^3z=3x^3y$, hence 3$yf_{x_2}=10y^3z$, i.e., $y=0$, so that
$f_{x_2}=x^3=0$, a contradiction. Since
$f_{A^{-1}}^{\lambda,\mu,\xi}=f^{\mu,\lambda,\xi}$ for $A=[e_4,e_1,e_2,e_3]$, it has been shown that if every partial derivative $f_{x_i}$ vanishes at
$[x,y,z,t]\in k^4$ with $xyzt=0$, then $x=y=z=t=0$. Now suppose that  $(x,y,z,t)\in \PP{3}$ is a singular point of $V_p(f)$, hence $xyzt\not=0$.
We may assume $t=1$. We write $f_x$, $f_y$, $f_z$, and $f_t$ for $f_{x_i}(x,y,z,1)$ ($i\in [1,4]$), respectively. Clearly $f_x=f_y=f_z=f_t=0$ if and only if
$g_j=0$ ($j=[1,4]$), where
\[
 g_1=xf_x-f_t,\ g_2=yf_y-f_t,\ g_3=zf_z-f_t,\ g_4=f_t.
\]
 Since $(3g_1+g_2-3g_3)=10(x^3y-z^3)$ vanishes, $g_2=3(y^3z-x)=0$, hence $y^{10}=1$. Now that $y^{10}=1$ and $x=y^3z$, we have
$g_1=g_3=2y^{6}h_1$, $g_4=h_2$ and $g_2=0$, where
\[
 h_1=y^4 z^3-y^7z+\lambda z^4+\mu y^6,\ \ h_2=z^3+3y^3z+2\mu y^2+\xi y^4z^2.
\]
$V_p(f)$ is singular at $(x,y,z,1)\in \PP{3}$ if and only if $y^{10}=1$, $x=y^3z\not=0$, $h_1=h_2=0$.
%%% Let $h_3=3y^4z^2+y^7+\xi y^8z+2\lambda z^3$ so that $2y^6 h_1+h_2=y^6zh_3$.

We discuss first the case $\lambda\mu\xi=0$.
%%% lemma 3.3
\begin{lemma} \ \\
$(1)$ The quartic form $f^{\lambda,0,0}$ is singular if and only if $\lambda^4=3^{-6}4^4$.\\
$(2)$ The quartic form $f^{0,\mu,0}$ is singular if and only if $\mu^4=3^{-6}4^4$.
\end{lemma}
\begin{proof}
It suffices to prove (1). We write $f$ for $f^{\lambda,0,0}$. Since $h_1=z(\lambda z^3+y^4z^2-y^7)$ and $h_2=z(z^2+3y^3)$, $(x,y,z,1)\in \PP{3}$ is a
singular point of $V_p(f)$ if and only if there exists $z\in k^*$ such that $\lambda z^3+y^4(z^2-y^3)=0$ and $z^2+3y^3=0$, equivalently
$3\lambda z+4y^4=0$ and $z^2+3y^3=0$, for some $y\in k$ satisfying $y^{10}=1$. There exists $z\in k^*$ satisfying this condition if and only if
$\lambda\not=0$ and $(-4\lambda^{-1}y^4/3)^2+3y^3=0$, i.e., $y^5=-27\lambda^2/16=0$, for some $y$ such that $y^{10}=1$. Thus
$V(f)$ is singular if and only if $1=3^64^{-4}\lambda^4$.
\end{proof}

%%% lemma 3.4
\begin{lemma}
The quartic form $f^{0,0,\xi}$ is singular if and only if $\xi^4=4^4$.
\end{lemma}
\begin{proof}
We write $f$ for $f^{0,0,\xi}$.
Since $h_1=y^4z(z^2-y^3)$ and $h_2=z(z^2+3y^3+\xi y^4z)$, $(x,y,z,1)\in \PP{3}$ is a singular point of $V_p(f)$ if and only if there exists $z\in k^*$
such that $z^2-y^3=0$ and $z^2+\xi y^4z+3y^3=0$,  equivalently $z^2-y^3=0$ and $\xi y^4z+4y^3=0$
for some $y$ satisfying $y^{10}=1$. There exists $z\in k^*$ satisfying this condition if and only if $\xi\not=0$ and
$(-4y^{-1}\xi^{-1})^2-y^3=0$, i.e., $y^5=4^2\xi^{-2}$, for some $y$ satisfying $y^{10}=1$. Thus $V_p(f)$ is singular if and only if $1=4^4\xi^{-4}$.
\end{proof}

%%% lemma 3.5
\begin{lemma} \ \\
$(1)$ The quartic form $f^{\lambda,0,\xi}$ is singular if and only if $(16-18\lambda\xi)^2-(27\lambda^2-\xi^2+\lambda\xi^3)^2=0$. \\
$(2)$ The quartic form $f^{0,\mu,\xi}$ is singular if and only if $(16-18\mu\xi)^2-(27\mu^2-\xi^2+\mu\xi^3)^2=0$.
\end{lemma}
\begin{proof}
It suffices to prove (1), for $f_{A^{-1}}^{\lambda,0,\xi}=f^{0,\lambda,\xi}$, where $A=[e_4,e_1,e_2,e_3]$. If $\lambda\xi=0$,
we are done by Lemma 3.3 and Lemma 3.4. Assume $\lambda\xi\not=0$. We write $f$ for $f^{\lambda,0,\xi}$.
 $V_p(f)$ is singular at $(x,y,z,1)$, if and only if there exists $z\in k^*$ such that $h_1=zh_1'=0$ and $h_2=zh_2'=0$ for some $y$ satisfying $y^{10}=1$.
 Let $h_3'=3\lambda z^2+4y^4z+\xi y^8$ and $h_4'=(4-3\lambda\xi)z+\xi y^4-9\lambda y^9$ so that $3h_1'+y^4h_2'=zh_3'$ and $h_3'-3\lambda h_2'=y^4h_4'$, and
write $h_2'$ and $h_4'$ as $z^2+az+b$ and $\alpha z+\beta$, respectively. Now $V_p(f)$ has a singular point, if and only if there exists $z\in k^*$
such that $h_2'=0$ and $h_4'=0$ for some $y$ satisfying $y^{10}=1$, i.e., $r=(-\beta)^2+a\alpha(-\beta)+b\alpha^2=0$ for some $y$ satisfying $y^{10}=1$.
Since $r=3y^8\{27\lambda^2-\xi^2+\lambda\xi^3+y^5(16-18\lambda\xi)\}$, $f$ is singular if and only if
$(27\lambda^2-\xi^2+\lambda\xi^3)^2-(16-18\lambda\xi)^2=0$.
\end{proof}

%%% lemma 3.6
\begin{lemma}
The quartic form $f^{\lambda,\mu,0}$ is singular if and only if $4^4(1+3\lambda^2\mu^2)^2-(27\lambda^2+6\lambda\mu+27\mu^2-16\lambda^3\mu^3)^2=0$.
\end{lemma}
\begin{proof}
If $\lambda\mu=0$, we are done by Lemma 3.3. We write $f$ for $f^{\lambda,\mu,0}$ and assume $\lambda\mu\not=0$.
 $V_p(f)$ is singular at $(x,y,z,1)$ if and only if there exists $z\in k^*$ such that $h_1=h_2=0$ for some $y$ satisfying $y^{10}=1$. Let $\delta=y^5$, and
\[
 h_3=z^2-2\lambda y^9z+\frac{1}{3}y^3(1-4\lambda\mu \delta),\ \ h_4=y\{4+(6\lambda^2+2\lambda\mu)\delta\}z-\lambda+3\mu+4\lambda^2\mu\delta
\]
so that $2y^6h_1+h_2=z(2\lambda y^6h_2+3h_3)$ and $h_2=(z+2\lambda y^9)h_3+\frac{2}{3}y^2h_4$. Thus, $V_p(f)$ has a singular point if and only if
there exists $z\in k^*$ such that $h_3=h_4=0$ for some $y$ satisfying $y^{10}=1$. Write $h_3=z^2+az+b$ and $h_4=\alpha z+\beta$. Then
there exists $z\in k^*$ such that $h_3=h_4=0$ for some $y$ satisfying $y^{10}=1$ if and only if $r=\beta^2-\alpha\beta a+\alpha^2 b$ vanishes for
some $y$ satisfying $y^{10}=1$, for $a=-2\lambda y^9\not=0$. Since $3r=r_0+ r_1\delta$, where  $r_0=27\lambda^2+27\mu^2+6\lambda\mu-16\lambda^3\mu^3$ and
$r_1=16(1+3\lambda^2\mu^2)$, $f$ is singular if and
only if $r_0^2-r_1^2=0$.
%Indeed, if $f$ is singular, then $r_0^2=r_1^2$. Conversely assume $r_0^2=r_1^2$.  If $r_1\not=0$, then there exists $y$ satisfying
%$y^{10}=1$ and $r_0+y^5r_1=0$, for $-r_0/r_1=\pm 1$. Now that $3r=r_0+y^5r_1=0$, there exists $z\in k$ such that $h_3=h_4=0$. Moreover $z\in k^*$.
%In fact this is clear if $b\not=0$. If $b=0$, then $h_4=\alpha z+3\mu=0$ so that $z\not=0$, for $\mu\not=0$. If $r_1=0$, then let $y=1$. Since
%$3r=r_0+y^5r_1=0$, there exists $z\in k$  such that $h_3=h_4=0$. As in the case $r_1\not=0$, we see $z\not=0$. Thus if $r_0^2=r_1^2$, then
%there exists $z\in k^*$ such that $h_3=h_4=0$ for some $y$ satisfying $y^{10}=1$, hence $f$ is singular.
\end{proof}

Let $R=R_0^2-R_1^2\in k[u,v,w]$, where
\begin{eqnarray*}
&&R_0=27(u^2+v^2)+6uv -w^2-36(u+v)uvw +(u+v)w^3
  -16u^3v^3+8u^2v^2w^2-uvw^4,\\
&&R_1=16-18(u+v)w+48u^2v^2+20uvw^2,\\
\end{eqnarray*}
%%% theorem 3.7
\begin{theorem}
Let the polynomial $R$ be as above. The quartic form $f^{\lambda,\mu,\xi}$ is singular if and only if $R(\lambda,\mu,\xi)=0$.
\end{theorem}
\begin{proof}
Due to Lemma 3.5 and Lemma 3.6 it suffices to prove the theorem under the condition $\lambda\mu\xi\not=0$. We write $f$ for $f^{\lambda,\mu,\xi}$.
$V_p(f)$ is singular at $(x,y,z,1)$  if and only if there exists $y, z \in k^*$ such that $y^{10}=1$, $h_1=h_2=0$ and $x=y^3z$.
Since $2y^6 h_1+h_2=y^6zh_3$, where $h_3=2\lambda z^3+3y^4z^2+\xi y^8z+y^7$, there exists $z\in k^*$ such that $h_2=h_3=0$ for some $y$ satisfying
$y^{10}=1$ if and only if there exists $z\in k$ such that $h_2=h_3=0$ for some $y$ satisfying $y^{10}=1$. As is well known \cite[p.203]{lan} (recall that
$\lambda\not=0$), there
exists $z\in k$ such that $h_2=h_3=0$ for given $y\in k$ if and only if $\det S=0$, where
\begin{eqnarray*}
S&=&\left[\begin{array}{rrrrrr}
                 1&\xi y^4&       3y^3 &    2\mu y^2&          0&        0\\
                 0&          1& \xi y^4&        3y^3&   2\mu y^2&        0\\
                 0&          0&           1& \xi y^4&       3y^3& 2\mu y^2\\
              2\lambda&       3y^4& \xi y^8&         y^7&          0&        0\\
                 0&       2\lambda&        3y^4& \xi y^8&        y^7&        0\\
                 0&          0&        2\lambda&        3y^4& \xi y^8&     y^7
         \end{array}
    \right].
\end{eqnarray*}
As will be shown later $\frac{1}{4}y^{-6}\det S=\Delta_0+\delta\Delta_1$, where
\begin{eqnarray*}
\Delta_0&=&27(\lambda^2+\mu^2)-\xi^2+6\lambda\mu-36(\lambda^2\mu+\lambda\mu^2)\xi+(\lambda+\mu)\xi^3
   -16\lambda^3\mu^3+8\lambda^2\mu^2\xi^2-\lambda\mu\xi^4,\\
\Delta_1&=&16-18(\lambda+\mu)\xi+48\lambda^2\mu^2+20\lambda\mu\xi^2,\\
\end{eqnarray*}
Plainly there exists some $y$ such that  $y^{10}=1$ and $\det S=0$ if and only if $\Delta_0^2-\Delta_1^2=0$. Thus if $\lambda\mu\xi\not=0$,
$f$ is singular if and only if $R(\lambda,\mu,\xi)=0$.

%Denote by $s_i$ the $i$-th row of $S$, and let $\delta=y^5$. Then,
%$s_4-2\lambda s_1$, $s_4-2\lambda s_1-y^4(3-\lambda\xi) s_2$, and
%$s_4-2\lambda s_1-y^4(3-\lambda\xi) s_2-y^3\{\delta(-2\xi+2\lambda\xi^2)-\lambda\} s_3$ respectively takes the following forms
%\begin{eqnarray*}
%&&[0,y^4(3-\lambda\xi),y^3(\delta\xi-6\lambda),y^2(\delta-4\lambda\mu),0,0],\\
%&&[0,0,y^3\{\delta(-2\xi+2\lambda\xi^2)-6\lambda\},y^2\{\delta(-8+6\lambda\xi)-4\lambda\mu\},y\delta(-6\mu+4\lambda\mu\xi),0],\\
%&&[0,0,0,y^2\{\delta(-8+12\lambda\xi)-4\lambda\mu+2\xi^2-2\lambda\xi^3\},y\{\delta(18\lambda-6\mu+4\lambda\mu\xi)+6\xi-6\lambda\xi^2\},
%                \delta(12\lambda\mu)+4\mu\xi-4\lambda\mu\xi^2].
%\end{eqnarray*}
In order to compute $\det S$ we
replace three rows $s_6$, $s_5$ and $s_4$ of $S$ by $s_6-2\lambda s_3$, $s_5-2\lambda s_2-y^4(3-\lambda\xi) s_3$, and
$s_4-2\lambda s_1-y^4(3-\lambda\xi) s_2-y^3\{\delta(-2\xi+2\lambda\xi^2)-\lambda\} s_3$, respectively. One can easily see
that $\frac{1}{4}y^{-6}\det S$ is equal to the following determinant $\Delta$;
\[
 \left|\begin{array}{lll}
           \delta(-4+6\lambda\xi)-2\lambda\mu+\xi^2-\lambda\xi^3 &\delta(9\lambda-3\mu+2\lambda\mu\xi)+3\xi-3\lambda\xi^2 &\delta(6\lambda\mu)+2\mu\xi-2\lambda\mu\xi^2\\
           \delta(-\xi+\lambda\xi^2)-3\lambda  &\delta(-4+3\lambda\xi)-2\lambda\mu &\delta(-3\mu+2\lambda\mu\xi)\\
           3-\lambda\xi  &\delta(\xi)-6\lambda  &\delta-4\lambda\mu \end{array} \right|
\]
Expanding the determinant according to the first row, we get  $\Delta=\Delta_0+\delta\Delta_1$.
\end{proof}

%%%% standard representation of S_5
We shall describe $\A{5}$-invariant and $\SSS{5}$-invariant nonsingular quartic forms. Let
\[
 s_1=(123),\ s_2=(12)(34),\ s_3=(12)(45),\ s_5=s_1s_2s_3=(12345),\  t_1=(12).
\]
Note that  $\A{5}=\langle s_1,s_2,s_3\rangle=\langle s_1,s_5\rangle$ and $\SSS{5}=\langle t_1,s_5\rangle$. It is known that
there are two  faithful representations $\varphi$  of $\SSS{5}$ in $PGL_4(k)$, such that $\varphi(\SSS{5})$ has a nonsingular eigenspace, and that
their restrictions to $\A{5}$ are the only faithful representation of $\A{5}$ such that $\varphi(\A{5})$ has a nonsingular eigenspace \cite{dol}\cite{mar}.
Let $R_{11}=\diag[1,1,\omega,\omega^2]$, with $\ord(\omega)=3$, and $R_{1j}$ ($j\in [2,4]$) be as follows;
\begin{eqnarray*}
&&R_{12}=\left[\begin{array}{cccc}
               1&0&0&0\\
               0&-\frac{1}{3}&\frac{2}{3}&\frac{2}{3}\\
               0&\frac{2}{3}&-\frac{1}{3}&\frac{2}{3}\\
               0&\frac{2}{3}&\frac{2}{3}&-\frac{1}{3}\end{array}\right],\
  R_{13}=\left[\begin{array}{cccc}
               -\frac{1}{4}&\frac{\sqrt{15}}{4}&0&0\\
               \frac{\sqrt{15}}{4}& \frac{1}{4}&0&0\\
                0 & 0&                          0&1\\
                0&  0                           1&0\end{array}\right],\
  R_{14}=\left[\begin{array}{cccc}
                1&0&0&0\\
                0&1&0&0\\
                0&0&0&1\\
                0&0&1&0\end{array}\right].
\end{eqnarray*}
By Moore's theorem \cite{moo} there exists uniquely an injective group homomorphism $\phi$ of $\SSS{5}$ into $GL_4(k)$, which can be seen to be equivalent to
the standard representation $\rho_V$ of $\SSS{5}$ to be introduced shortly by showing that their characters are equal. H. Maschke \cite{mas} obtained a faithful
representation $\psi_1=\pi_4\circ \phi$ (resp. $\varphi_1=\pi_4\circ \phi$) of $\A{5}$ (resp. $\SSS{5}$)  in $PGL_4(k)$. The groups $\psi_1(\A{5})$ and
$\varphi_1(\SSS{5})$ have  only one nonsingular eigenspace $\langle h_0,h_1\rangle=\Form_{4,4}(R_{11},R_{12},R_{13},R_{14};1,1,1,1)$, where
\begin{eqnarray*}
&&h_0=-x^4+2\sqrt{15}(y^3+z^3+t^3)x +10(z^3+t^3)y+13x^2y^2+5z^2t^2\\
&&+(26x^2-6\sqrt{15}xy+20y^2)zt,\ \
h_1=x^4+y^4+2x^2y^2+4z^2t^2+4(x^2+y^2)zt.
\end{eqnarray*}
Note that $\phi(\SSS{5})$ and $\psi_1(\SSS{5})$ have the same eigenspaces in $\Form_{4,4}$.
The standard representation of the symmetric group $\SSS{5}$ is the restriction $\rho_V$ of the following representation $\rho$ of $\SSS{5}$ in
$k^5$ such that $\rho(\tau)e_i=e_{\tau^{-1}(i)}$, where $e_i$ ($i\in [1,5]$) stands for the $i$-th column vector of the unit matrix $E_5$,
to the invariant subspace
$V=\{x\in k^5 : x_1+\cdots +x_5=0\}$ \cite{ful}. With respect to the basis $f_j=e_j-e_5$ ($j\in [1,4]$) we have for $t_1$, $s_1$ and $s_5$ in $\SSS{5}$
\[
 \rho_V(t_1))=\left[\begin{array}{cccc}
                    0&1&0&0\\
                    1&0&0&0\\
                    0&0&1&0\\
                    0&0&0&1\end{array}\right],\
\rho_V(s_1)=\left[\begin{array}{cccc}
              0&0&1&0\\
              1&0&0&0\\
              0&1&0&0\\
              0&0&0&1\end{array}\right],\
\rho_V(s_5)=\left[\begin{array}{rrrr}
                      -1&-1&-1&-1\\
                       1& 0& 0& 0\\
                       0& 1& 0& 0\\
                       0& 0& 1& 0\end{array}\right].
\]
Note that unless $\tau=id$, $\rho_V(\tau)\not\in k^*E_4$. In particular $(\rho_V(\tau))=(E_4)$ if and only if $\tau=id$.
It is immediate that $\det(\rho_V(s_5)-\lambda E_4)=\lambda^4+\lambda^3+\lambda^2+\lambda+1$ and that
$(\rho_V(s_5)-\lambda E_4)(\lambda^{-1},\lambda^{-2},\lambda^{-3},\lambda^{-4})^t=0$ provided $\lambda^5=1$ and $\lambda\not=1$. Let $\varepsilon
\in k^*$ with $\ord(\varepsilon)=5$. Then $T=[t_{ij}]\in M_4(k)$ with $t_{ij}=\varepsilon^{-ij}=(\varepsilon^{-1})^{ij}$ is nonsingular, for we have
%%% lemma 3.8
\begin{lemma}
Assume that $p=\ord(\delta)$ is a prime $(\delta\in k^*)$. \\
$(1)$ The square matrix $A=[a_{ij}]\in M_{p-1}(k)$ with $a_{ij}=\delta^{ij}$ $(i,j\in [1,p-1])$
is nonsingular, and $A^{-1}=p^{-1}(B-C)$, where $B=[b_{ij}],C=[c_{ij}]\in M_{p-1}(k)$ with $b_{ij}=\delta^{-ij}$ and $c_{ij}=1$.\\
$(2)$  The square matrix $A=[a_{ij}]\in M_p(k)$ with $a_{ij}=\delta^{(i-1)(j-1)}$ $(i,j\in [1,p])$ is nonsingular, and $A^{-1}=[\alpha_{ij}]$
where $\alpha_{ij}=p^{-1}\delta^{-(i-1)(j-1)}$.
\end{lemma}
\begin{proof}
(1) Clearly $\sum_{j=1}^{p-1}\delta^j=-1$. Note that for any integer $h$, $\ord(\delta^h)=p$ unless  $p$ divides $h$. Thus $AC=-C$ and $AB=pE-C$, so that
$A(B-C)=pE$, where $E$ stands for the unit matrix in $M_{p-1}(k)$. (2) It is immediate that $\sum_{i=1}^p a_{i\ell}\alpha_{\ell j}=\delta_{ij}$.
\end{proof}

$U=[u_{ij}]\in M_4(k)$ with $u_{ij}=\varepsilon^{ij}$ satisfies $(U-I)T=5E_4$, where each component of $I\in M_4(k)$ is equal to one.

 Let $S=T\diag[\varepsilon^4,\varepsilon^3,\varepsilon^2,\varepsilon][e_1,e_2,e_4,e_3]$, where $E_4=[e_1,e_2,e_3,e_4]$ is the unit matrix. We easily
verify that $S^{-1}\rho_V(s_5)S=\diag[\varepsilon,\varepsilon^2,\varepsilon^4,\varepsilon^3]$ which will be denoted by $\tau_{12345}$ and
\[
 S^{-1} \rho_V(t_1)S=\left[\begin{array}{rrrr}
                            \eta &2\eta-1&-\eta+1&-2\eta+1\\
                            2\eta-1&-\eta+1&-2\eta+1&\eta \\
                           -\eta+1 &-2\eta+1& \eta&2\eta-1\\
                           -2\eta+1&\eta &2\eta-1& -\eta+1\end{array}\right], \ \ (\eta=(3+\varepsilon+\varepsilon^4)/5)
\]
which will be denoted by $\tau_{12}$.
%% lemma 3.9
\begin{lemma} Let $F_0(x,y,z,t)=x^3y+y^3z+z^3t+t^3x+3xyzt$ and $F_1=x^2z^2+y^2t^2+2xyzt$. Then
 $ F_{i,\tau^{-1}}=F_i$ for $i\in [0,1]$ and $\tau\in \{\tau_{12},\tau_{12345}\}$.
\end{lemma}
\begin{proof}
 Any monomial $M\in \{x^3y,y^3z,z^3t,t^3x,x^2z^2,y^2t^2,xyzt\}$ satisfies $M_{\tau_{12345}}=M$. So we may assume $\tau=\tau_{12}$.
Since $\varepsilon^2+\varepsilon^3=(5\eta-3)^2-2$ and $\sum_{i=1}^4\varepsilon^i=-1$, it follows that  $\eta^2-\eta=-1/5$, hence
$(xz+yt)_{\tau^{-1}}=xz+yt$ so that $F_{1,\tau^{-1}}=F_1$.
Denoting the first row of the matrix $\tau_{12}$ by $[\alpha,\beta,\gamma,\delta]$, we have
$F_{\tau^{-1}}(x,y,z,t)=G(x,y,z,t,\alpha,\beta,\gamma,\delta)$.
Let $a_j,x_j$ ($j\in [1,4]$) be algebraically independent indeterminates, $F(x_1,x_2,x_3,x_4)=F_{0}(x_1,x_2,x_3,x_4)$,
$y_j=\sum_{i=1}^4a_{\sigma^{j-1}(i)}x_i$, where $\sigma=(1234)\in \SSS{4}$,  and
\[
 F_{0}(y_1,y_2,y_3,y_4)=G(x_1,...,x_4,a_1,...,a_4)=\sum_{i_1+\cdots +i_4=4} g_{i_1i_2i_3i_4}(a_1,...,a_4)x_1^{i_1}x_2^{i_2}x_3^{i_3}x_4^{i_4}.
\]
Regarding $G$ as a polynomial in $x$ or $a$, we have
\begin{eqnarray*}
&& (\sigma G)(x_1,...,x_4,a)=G(x_{\sigma(1)},...,x_{\sigma(4)},a)=F(y_4,y_1,y_2,y_3)=F(y_1,y_2,y_3,y_4),\\
&& (\sigma G)(x,a_1,...,a_4)=G(x,a_{\sigma(1)},...,a_{\sigma(4)})=F(y_2,y_3,y_4,y_1)=F(y_1,y_2,y_3,y_4).
\end{eqnarray*}
Consequently,  $g_{i_1\cdots i_4}=g_{j_1\cdots j_4}$ if $[j_1,...,j_4]=[i_1,...,i_4]\sigma^\ell$ for some $\ell\in [0,3]$,  and
$\sigma g_{i_1...i_4}= g_{[i_1...i_4]\sigma}$.
Let $\Form_{4,4,1}=\{x_1^{i_1}x_2^{i_2}x_3^{i_3}x_4^{i_4}\in \Form_{4,4}\ :\ i_1\geq i_2,i_3,i_3\}$. Namely,
\begin{eqnarray*}
\Form_{4,4,1}&=&\{x_1^4,x_1^3x_2,x_1^3x_3,x_1^3x_4,x_1^2x_2^2,x_1^2x_3^2,x_1^2x^2_4,x_1^2x_2x_3,x_1^2x_2x_4,x_1^2x_3x_4,x_1x_2x_3x_4\}.
\end{eqnarray*}
Evidently $\Form_{4,4}=\cup_{i=0}^3 \sigma^i \Form_{4,4,1}$. In order to describe polynomials $g_{i_1\cdots i_4}\in \Form_{4,4}$,
we introduce polynomials in indeterminates $a,b,c,d$ as follows.
\begin{eqnarray*}
f_1&=&a^4+b^4+c^4+d^4,\ f_2=a^3b+b^3c+c^3d+d^3a,\ f_3=da^3+ab^3+bc^3+cd^3,\\
f_4&=&a^3c+b^3d+c^3a+d^3b,\ f_5=a^2b^2+b^2c^2+c^2d^2+d^2a^2,\ f_6=a^2c^2+b^2d^2,\\
f_7&=&a^2bc+b^2cd+c^2da+d^2ab,\ f_8=da^2b+ab^2c+bc^2d+cd^2a,\\
f_9&=&cda^2+dab^2+abc^2+bcd^2,\ f_{10}=abcd.
\end{eqnarray*}
Then $g_{i_1\cdots i_4}=g_{i_1\cdots i_4}(a,b,c,d)$  turn out to satisfy $g_{4000}=f_2+3f_{10}$, and
\begin{eqnarray*}
g_{3100}&=&f_4+3f_5+3f_7,\ g_{3010}=f_3+3f_7+3f_8,\ g_{3001}=f_1+3f_8+3f_9,\\
g_{2200}&=&3(f_3+f_6+f_7+f_8),\ g_{2020}=3(f_5+2f_9+2f_{10}),\ g_{2002}=3(f_3+f_6+f_7+f_8),\\
g_{2110}&=&3(f_2+2f_6+3f_8+2f_9),\ g_{2101}=3(f_2+f_4+f_5+3f_9+4f_{10}),\\
g_{2011}&=&3(f_3+f_4+f_5+2f_7+8f_{10}),\ g_{1111}=3(f_1+2f_6+8f_7+4f_8).
\end{eqnarray*}
By use of equalities $\eta(-\eta+1)=(2\eta-1)^2=1/5$, $\eta^2+(-\eta+1)^2=3/5$ and $\eta^3-(-\eta+1)^3=(2\eta-1)\{\eta^2+\eta(-\eta+1)+(-\eta+1)^2\}$, we can
evaluate $\zeta_i=f_i(\alpha,\beta,\gamma,\delta)$ as follows;
\begin{eqnarray*}
\zeta_1&=&9/25,\ \zeta_2=3/25,\ \zeta_3=-3/25,\ \zeta_4=1/25,\ \ \ \ \zeta_5=6/25,\\
\zeta_6&=&2/25,\ \zeta_7=2/25,\ \zeta_8=-1/25,\ \zeta_9=-2/25,\ \zeta_{10}=-1/25.
\end{eqnarray*}
Now  $h_{i_1\cdots i_4}=g_{i_1\cdots i_4}(\alpha,\beta,\gamma,\delta)$ turns out to be
\begin{eqnarray*}
h_{4000}&=&0,\ h_{3100}=1,\ h_{3010}=0,\ h_{3001}=0,\ h_{2200}=0,\ h_{2020}=0,\ h_{2002}=0,\\
h_{2110}&=&0,\ h_{2101}=0,\ h_{2011}=0,\ h_{1111}=3.
\end{eqnarray*}
Consequently $F_{0,\tau_{12}^{-1}}(x,y,z,t)=G(x,y,z,t,\alpha,\beta,\gamma,\delta)=F_{0}(x,y,z,t)$.
\end{proof}

Since the representations $\rho_V$ and $\phi$ of $\SSS{5}$ are equivalent, there exists a $T\in GL_4(k)$ such that $\rho_V=T^{-1}\phi T$, hence
$(TS)^{-1}\phi TS=S^{-1}\rho_V S$. Define a faithful representation $\psi_1'$ (resp. $\varphi_1'$) of $\A{5}$ (resp. $\SSS{5}$)  in $PGL_4(k)$
by  $(TS)^{-1}\psi_1(TS)$ (resp. $(TS)\varphi_1'(TS)$), i.e., $\psi_1'(\tau)=(TS)^{-1}\psi_1(\tau)(TS)$ for every $\tau\in \A{5}$. Now Lemma 3.9 implies
%%% theorem 3.10
\begin{theorem} The nonsingular eigenspace of $\varphi_1'(\SSS{5})$ and $\psi_1'(\A{5})$ in $\Form_{4,4}$ is $\langle F_0,F_1\rangle$.
\end{theorem}
\ \\
%%%% another representation of S_5 in PGL_4(k)
Let  $R_{21}=\diag[1,1,\omega,\omega^2]$ ($\omega^2+\omega+1=0$), and $R_{2j}\in GL_4(k)$ ($j\in [2,4]$) be as follows;
\begin{eqnarray*}
&&R_{22}=\left[\begin{array}{cccc}
          \frac{1}{\sqrt{3}}&0&0&\frac{\sqrt{2}}{\sqrt{3}}\\
          0&-\frac{1}{\sqrt{3}}&\frac{\sqrt{2}}{\sqrt{3}}&0\\
          0&\frac{\sqrt{2}}{\sqrt{3}}&\frac{1}{\sqrt{3}}&0\\
          \frac{\sqrt{2}}{\sqrt{3}}&0&0&-\frac{1}{\sqrt{3}}\end{array}\right],
R_{23}=\left[\begin{array}{cccc}
          \frac{\sqrt{3}}{2}&\frac{1}{2}&0&0\\
          \frac{1}{2}&-\frac{\sqrt{3}}{2}&0&0\\
          0&               0&             0&1\\
          0&               0&             1&0\end{array}\right],
R_{24}=\left[\begin{array}{cccc}
          0&1&0&0\\
          -1&0&0&0\\
          0&0&0&1\\
          0&0&-1&0\end{array}\right].
\end{eqnarray*}
The map $\psi_2$ of $\{s_1,s_2,s_3\}$ into $PGL_4(k)$ such that $\psi_2(s_j)=(R_{2j})$ ($j\in [1,3]$) can be extended to
an injective group homomorphism of $\A{5}$ into $PGL_4(k)$, and  the map $\varphi_2$ of $\{s_1,s_2,s_3,t_1\}$ into $PGL_4(k)$ such that
$\varphi_2(s_j)=(R_{2j})$ ($j\in [1,3]$) and $\varphi_2(t_1)=(R_{24})$ can be extended to an injective group homomorphism of $\SSS{5}$ into $PGL_4(k)$
\cite{mas}. The nonsingular eigenspace of $\psi_2(\A{5})$ in $\Form_{4,4}$ is $\Form_{4,4}(R_{21},R_{22},R_{23};1,1,1)=\langle f_0,f_1\rangle$, while the
nonsingular eigenspaces of $\varphi_2(\SSS{5})$ are $\Form_{4,4}(R_{21},R_{22},R_{23},R_{24};1,1,1,1)=\langle f_0\rangle$ and
$\Form_{4,4}(R_{21},R_{22},R_{23},R_{24};1,1,1,-1)=\langle f_1\rangle$, where $f_0(x,y,z,t)$ and $f_1(x,y,z,t)$ take the following forms \cite{mar};
\begin{eqnarray*}
&&x^4+y^4+2\sqrt{3}(-x^3y+y^3x)+2\sqrt{2}\sqrt{3}(z^3x+t^3y)+2\sqrt{2}(-z^3y+t^3x)+6(x^2y^2+z^2t^2)\\
&&+6\sqrt{3}(-x^2+y^2)zt-12xyzt,\\
&&x^4-y^4+\frac{2}{\sqrt{3}}(x^3y+y^3x)+\frac{2\sqrt{2}}{\sqrt{3}}(z^3x-t^3y)+2\sqrt{2}(z^3y+t^3x)+2\sqrt{3}(x^2+y^2)zt.
\end{eqnarray*}

Let $U=R_{21}R_{22}R_{23}$. Then $U^5=-E_4$. Indeed, there exists an $S\in GL_4(k)$ such that
$S^{-1}US=\diag[-\varepsilon, -\varepsilon^2, -\varepsilon^3, -\varepsilon^4]$, where $\ord(\varepsilon)=5$. In fact, writing $\sqrt{6}$ for
$\sqrt{2}\sqrt{3}$, we have
\[
  U=\left[\begin{array}{cccc}
           \frac{1}{2}& \frac{\sqrt{3}}{6}& \frac{\sqrt{6}}{3} & 0\\
           -\frac{\sqrt{3}}{6}& \frac{1}{2} & 0 &\frac{\sqrt{6}}{3}\\
          \frac{\sqrt{6}}{6}\omega& -\frac{\sqrt{2}}{2}\omega& 0& \frac{\sqrt{3}}{3}\omega\\
          \frac{\sqrt{2}}{2}\omega^2& -\frac{\sqrt{6}}{6}\omega^2&-\frac{\sqrt{3}}{3}\omega^2& 0
           \end{array}
      \right],\\
\]
so that $\det(U-\lambda E_4)=\lambda^4-\lambda^3+\lambda^2-\lambda +1$.  Denote the transpose of the following row vector
$e'_\lambda$in $k[\lambda]^4$ by $e_\lambda$:
\[
 [\sqrt{6}(1+\omega\lambda-2\omega\lambda^2),\sqrt{2}(-3+(2+\omega)\lambda),2((1-\omega)\lambda+3\omega\lambda^2-3\omega\lambda^3),
2\sqrt{3}(1-2\lambda+\lambda^2)].
\]
One can verify that the transpose of $(U-\lambda E_4)e_\lambda$ is equal to $[0,0,6\omega(\lambda^4-\lambda^3+\lambda^2-\lambda+1),0]$.
Note that $\omega^i\varepsilon^j$ ($i\in [0,1],j\in [0,3]$) is a basis of the field $\QQ(\omega,\varepsilon)$ over $\QQ$.
From now on we assume $\lambda^4-\lambda^3+\lambda^2-\lambda+1=0$, hence $e_\lambda$
is a nonzero eigenvector of $U$ for the eigenvalue $\lambda$. Let $S=[e_{-\varepsilon},e_{-\varepsilon^2},e_{-\varepsilon^3},e_{-\varepsilon^4}]$. Then
$S^{-1}US=\diag[-\varepsilon,-\varepsilon^2,-\varepsilon^3,-\varepsilon^4]$. In order to calculate $S^{-1}$ we introduce matrices $\Omega$ and $\Sigma$:
\begin{eqnarray*}
\Omega&=&\left[\begin{array}{cccc}
-1-\omega& -1-2\omega& -1& -1\\
1-\omega& 3& 3& 3\\
-1+\omega& 3\omega& 3\omega& 0\\
1& 0& -1& -1\end{array}\right],\ \
\Sigma=\left[\begin{array}{cccc}
           \varepsilon& \varepsilon^2& \varepsilon^3&\varepsilon^4\\
           \varepsilon^2& \varepsilon^4 & \varepsilon &\varepsilon^3 \\
           \varepsilon^3& \varepsilon & \varepsilon^4 &\varepsilon^2 \\
           \varepsilon^4& \varepsilon^3 & \varepsilon^2 &\varepsilon \end{array}
          \right].
\end{eqnarray*}
Since $1=\lambda-\lambda^2+\lambda^3-\lambda^4$,
\[
 e_\lambda=\diag[\sqrt{6},\sqrt{2},2,2\sqrt{3}]\
\Omega\ \diag[1,-1,1,-1]\left[\begin{array}{l}\lambda\\ \lambda^2\\ \lambda^3\\ \lambda^4\end{array}\right].
\]
Consequently $S=\diag[\sqrt{6},\sqrt{2},2,2\sqrt{3}]\ \Omega\  \Sigma$. By Lemma 3.7 $\Sigma^{-1}$ is known. Thus

\begin{eqnarray*}
 \Omega^{-1}&=&\left[\begin{array}{cccc}
                 \frac{\omega^2}{2}& \frac{1-\omega}{6}& 0& 1\\
                 \frac{5+4\omega}{6}& \frac{1+2\omega}{6}& 0& \frac{-1+\omega}{3}\\
                 \frac{-1+\omega^2}{3}& \frac{\omega^2}{3}& \frac{\omega^2}{3}& \frac{-\omega+\omega^2}{3}\\
                 \frac{1-\omega}{6}& \frac{2-\omega^2}{6}& -\frac{\omega^2}{3}& \frac{\omega-\omega^2}{3}\end{array}\right],\
5\Sigma^{-1}=\left[\begin{array}{cccc}
                                     \varepsilon^4&\varepsilon^3&\varepsilon^2&\varepsilon\\
                                     \varepsilon^3&\varepsilon&\varepsilon^4&\varepsilon^2\\
                                     \varepsilon^2&\varepsilon^4&\varepsilon&\varepsilon^3\\
                                     \varepsilon&\varepsilon^2&\varepsilon^3&\varepsilon^4\end{array} \right]
                            -\left[\begin{array}{cccc}
                                   1&1&1&1\\
                                   1&1&1&1\\
                                   1&1&1&1\\
                                   1&1&1&1\end{array}\right]
                        .
\end{eqnarray*}

Putting  $V=[v_{ij}]=S^{-1}R_{21}S$ and $W=[w_{ij}]=S^{-1}R_{24}S$, we obtain $V$ and $W$ as follows.
%%% lemma 3.11
\begin{lemma}
\begin{eqnarray*}
5v_{11}&=&-2\varepsilon-\varepsilon^3-2\varepsilon^4\\
5v_{12}&=& -9\varepsilon-5\varepsilon^3-6\varepsilon^4+\omega(-3-3\varepsilon-4\varepsilon^3) \\
5v_{13}&=&-3-5\varepsilon^2-2\varepsilon^4+\omega(4+4\varepsilon+7\varepsilon^3)\\
5v_{14}&=&  1-\varepsilon^3+\omega(1+\varepsilon-2\varepsilon^3)\\
5v_{21}&=&-3-2\varepsilon^3-5\varepsilon^4+\omega(4+7\varepsilon+4\varepsilon^2)\\
5v_{22}&=& -\varepsilon-2\varepsilon^2-2\varepsilon^3\\
5v_{23}&=& 1-\varepsilon+\omega(1-2\varepsilon+\varepsilon^2)\\
5v_{24}&=&  -5\varepsilon-9\varepsilon^2-6\varepsilon^3+\omega(-3-4\varepsilon-3\varepsilon^2)\\
5v_{31}&=& -6\varepsilon^2-9\varepsilon^3-5\varepsilon^4+\omega(-3-3\varepsilon^3-4\varepsilon^4)\\
5v_{32}&=& 1-\varepsilon^4+\omega(1+\varepsilon^3-2\varepsilon^4)\\
5v_{33}&=& -2\varepsilon^2-2\varepsilon^3-\varepsilon^4\\
5v_{34}&=&  -3-5\varepsilon-2\varepsilon^2+\omega(4+4\varepsilon^3+7\varepsilon^4)\\
5v_{41}&=& 1-\varepsilon^2+\omega(1-2\varepsilon^2+\varepsilon^4)\\
5v_{42}&=& -3-2\varepsilon-5\varepsilon^3+\omega(4+7\varepsilon^2+4\varepsilon^4)\\
5v_{43}&=& -6\varepsilon-5\varepsilon^2-9\varepsilon^4+\omega(-3-4\varepsilon^2-3\varepsilon^4)\\
5v_{44}&=& -2\varepsilon-\varepsilon^2-2\varepsilon^4.
\end{eqnarray*}
\begin{eqnarray*}
5\sqrt{3}w_{11}&=& 3\varepsilon^2-3\varepsilon^3 \\
5\sqrt{3}w_{12}&=& 1-\varepsilon+2\varepsilon^2-2\varepsilon^3+\omega(2\varepsilon+2\varepsilon^2+\varepsilon^4)\\
5\sqrt{3}w_{13}&=& 2-7\varepsilon+6\varepsilon^2-6\varepsilon^3+\omega(11\varepsilon^2-7\varepsilon^3+11\varepsilon^4)\\
5\sqrt{3}w_{14}&=& 3-3\varepsilon+\omega(6+3\varepsilon^2+6\varepsilon^4)\\
5\sqrt{3}w_{21}&=& 2-6\varepsilon-7\varepsilon^2+6\varepsilon^4+\omega(-7\varepsilon+11\varepsilon^3+11\varepsilon^4)\\
5\sqrt{3}w_{22}&=& -3\varepsilon+3\varepsilon^4\\
5\sqrt{3}w_{23}&=& 3-3\varepsilon^2+\omega(6+6\varepsilon^3+3\varepsilon^4)\\
5\sqrt{3}w_{24}&=& 1-2\varepsilon-\varepsilon^2+2\varepsilon^4+\omega(2\varepsilon^2+\varepsilon^3+2\varepsilon^4)\\
5\sqrt{3}w_{31}&=& 1+2\varepsilon-\varepsilon^3-2\varepsilon^4+\omega(2\varepsilon+\varepsilon^2+2\varepsilon^3)\\
5\sqrt{3}w_{32}&=& 3-3\varepsilon^3+\omega(6+3\varepsilon+6\varepsilon^2)\\
5\sqrt{3}w_{33}&=& 3\varepsilon-3\varepsilon^4\\
5\sqrt{3}w_{34}&=& 2+6\varepsilon-7\varepsilon^3-6\varepsilon^4+\omega(11\varepsilon+11\varepsilon^2-7\varepsilon^4)\\
5\sqrt{3}w_{41}&=& 3-3\varepsilon^4+\omega(6+6\varepsilon+3\varepsilon^3)\\
5\sqrt{3}w_{42}&=& 2-6\varepsilon^2+6\varepsilon^3-7\varepsilon^4+\omega(11\varepsilon-7\varepsilon^2+11\varepsilon^3)\\
5\sqrt{3}w_{43}&=& 1-2\varepsilon^2+2\varepsilon^3-\varepsilon^4+\omega(\varepsilon+2\varepsilon^3+2\varepsilon^4)\\
5\sqrt{3}w_{44}&=& -3\varepsilon^2+3\varepsilon^3.
\end{eqnarray*}
\end{lemma}
Put
\[
 5\sqrt{3}W=\left[\begin{array}{cccc}
                3(\varepsilon^2-\varepsilon^3)&u&v&w\\
                x&3(-\varepsilon+\varepsilon^4)&\ell&m\\
                y&p&3(\varepsilon-\varepsilon^4)&    n\\
                z&q&r &3(-\varepsilon^2+\varepsilon^3)\end{array}\right].
\]
Then we can easily verify
%%% lemmma 3.12
\begin{lemma}
$(1)$ $xu=-9+12(\varepsilon^2+\varepsilon^3)=rn$,\ $yv=-9+12(\varepsilon+\varepsilon^4)=qm$,\\
\ \ \  $zw=9(\varepsilon-\varepsilon^4)$,\ $p\ell=9(\varepsilon^2-\varepsilon^3)$.\\
$(2)$ $-9+12(\varepsilon^2+\varepsilon^3)=\{\sqrt{3}(1+2\varepsilon+2\varepsilon^2)\}^2$ and
$-9+12(\varepsilon+\varepsilon^4)=\{\sqrt{3}(1+2\varepsilon^2+2\varepsilon^4)\}^2$.
\end{lemma}

Define $\alpha,\beta,\gamma\in k^*$ by
$\alpha z=\alpha^{-1} w=3(\varepsilon-\varepsilon^4)$, $\beta q=\beta^{-1} m=\sqrt{3}(1+2\varepsilon^2+2\varepsilon^4)$, and
$\gamma r=\gamma^{-1} n=\sqrt{3}(1+2\varepsilon+2\varepsilon^2)$, respectively.
%\begin{eqnarray*}
%&&\alpha z=\alpha^{-1} w=3(\varepsilon-\varepsilon^4),\  \beta q=\beta^{-1} m=\sqrt{3}(1+2\varepsilon^2+2\varepsilon^4),\
%\gamma r=\gamma^{-1} n=\sqrt{3}(1+2\varepsilon+2\varepsilon^2),
%\end{eqnarray*} respectively.
Keeping in mind that $(a+\omega b)(a'+\omega b')=aa'-bb'+\omega(ab'+a'b-bb')$ ($a,a',b,b'\in \QQ(\varepsilon)$), we obtain
\begin{eqnarray*}
&&\alpha=\varepsilon^2+\varepsilon^4+\omega(-1+\varepsilon^4),\ \ \alpha^{-1}=-(\varepsilon+\varepsilon^3)+\omega(1-\varepsilon),\\
&&\beta=\frac{1}{\sqrt{3}}\{3+2\varepsilon-\varepsilon^2-2\varepsilon^3+\omega(2+\varepsilon^3+2\varepsilon^4)\},\\
&&\beta^{-1}=\frac{1}{\sqrt{3}}\{4\varepsilon+2\varepsilon^3+3\varepsilon^4+\omega(-3+2\varepsilon-2\varepsilon^3+3\varepsilon^4)\},\\
&&\gamma=\frac{1}{\sqrt{3}}\{-2\varepsilon-4\varepsilon^2-3\varepsilon^3+\omega(3+2\varepsilon-2\varepsilon^2-3\varepsilon^3)\},\\
&&\gamma^{-1}=\frac{1}{\sqrt{3}}\{-4+\varepsilon-3\varepsilon^2-\varepsilon^3+\omega(-2-\varepsilon-2\varepsilon^3)\},
\end{eqnarray*}
so that
\begin{eqnarray*}
&&\alpha\beta^{-1}=\frac{1}{\sqrt{3}}\{-1+3\varepsilon+2\varepsilon^2+4\varepsilon^4+\omega(2\varepsilon+2\varepsilon^3+\varepsilon^4)\},\\
&&\alpha^{-1}\beta=\frac{1}{\sqrt{3}}\{-3\varepsilon^2-4\varepsilon^3-2\varepsilon^4+\omega(1-2\varepsilon-5\varepsilon^2-4\varepsilon^3)\},\\
&&\alpha\gamma^{-1}=\frac{1}{\sqrt{3}}\{-3\varepsilon-2\varepsilon^2-4\varepsilon^4+\omega(5-\varepsilon+4\varepsilon^2+2\varepsilon^3)\},\\
&&\alpha^{-1}\gamma=\frac{1}{\sqrt{3}}\{-1+2\varepsilon+4\varepsilon^2+3\varepsilon^3+\omega(-2-2\varepsilon-\varepsilon^2)\},\\
&&\beta\gamma^{-1}=\varepsilon^3+\varepsilon^4+\omega(-1+\varepsilon^3),\ \beta^{-1}\gamma=-(\varepsilon+\varepsilon^2)+\omega(1-\varepsilon^2).
\end{eqnarray*}
Let $T=\diag[\alpha,\beta,\gamma,1]$. Then $5(ST)^{-1}R_{24}ST$ is equal to
\[
 \left[\begin{array}{cccc}
  \sqrt{3}(\varepsilon^2-\varepsilon^3)& 1+2\varepsilon^3+2\varepsilon^4& 1+2\varepsilon^2+2\varepsilon^4& \sqrt{3}(\varepsilon-\varepsilon^4)\\
  1+2\varepsilon^3+2\varepsilon^4& \sqrt{3}(-\varepsilon+\varepsilon^4)& \sqrt{3}(\varepsilon^2-\varepsilon^3)& 1+2\varepsilon^2+2\varepsilon^4\\
 1+2\varepsilon^2+2\varepsilon^4&\sqrt{3}(\varepsilon^2-\varepsilon^3)& \sqrt{3}(\varepsilon-\varepsilon^4)& 1+2\varepsilon+2\varepsilon^2\\
 \sqrt{3}(\varepsilon-\varepsilon^4)&1+2\varepsilon^2+2\varepsilon^4& 1+2\varepsilon+2\varepsilon^2& \sqrt{3}(-\varepsilon^2+\varepsilon^3)
       \end{array}  \right].
\]
By use of computer we obtain
\begin{eqnarray*}
&&f_{0\ (ST)^{-1}}=c_0(x^3y-y^3t+t^3z+z^3x-\frac{\sqrt{3}}{2}x^2t^2-\frac{\sqrt{3}}{2}y^2z^2),\\
&&f_{1\ (ST)^{-1}}=c_1(x^3y+y^3t+t^3z-z^3x+\sqrt{3}x^2t^2-\sqrt{3}y^2z^2+3\sqrt{3}xyzt),
\end{eqnarray*}
where
\begin{eqnarray*}
&&c_0=-27\cdot 320\sqrt{3}\{10+21\varepsilon+18\varepsilon^2+6\varepsilon^3+\omega(6+18\varepsilon+21\varepsilon^2+10\varepsilon^3)\},\\
&&c_1=9\cdot 320\sqrt{3}\{-7\varepsilon-10\varepsilon^2-2\varepsilon^3+\omega(2+10\varepsilon+7\varepsilon^2)\}.
\end{eqnarray*}
We summarize these computations. Let $S'=STR_{14}$, where $R_{14}=[e_1,e_2,e_4,e_3]$. Then
$S'^{-1}R_{21}R_{22}R_{23}S'=-\diag[\varepsilon,\varepsilon^2,\varepsilon^4,\varepsilon^3]$, and
\begin{eqnarray*}
&&R_{24}'=S'^{-1}R_{24}S'=\frac{1}{5}\left[\begin{array}{cccc}
  \sqrt{3}(\varepsilon^2-\varepsilon^3)& 1+2\varepsilon^3+2\varepsilon^4& \sqrt{3}(\varepsilon-\varepsilon^4)& 1+2\varepsilon^2+2\varepsilon^4\\
  1+2\varepsilon^3+2\varepsilon^4& \sqrt{3}(-\varepsilon+\varepsilon^4)& 1+2\varepsilon^2+2\varepsilon^4& \sqrt{3}(\varepsilon^2-\varepsilon^3)\\
\sqrt{3}(\varepsilon-\varepsilon^4) &1+2\varepsilon^2+2\varepsilon^4&\sqrt{3}(-\varepsilon^2+\varepsilon^3)& 1+2\varepsilon+2\varepsilon^2  \\
1+2\varepsilon^2+2\varepsilon^4&\sqrt{3}(\varepsilon^2-\varepsilon^3)& 1+2\varepsilon+2\varepsilon^2 & \sqrt{3}(\varepsilon-\varepsilon^4)
       \end{array}  \right],\\
&&f'_0(x,y,z,t)=x^3y-y^3z+z^3t+t^3x-\frac{\sqrt{3}}{2}x^2z^2-\frac{\sqrt{3}}{2}y^2t^2,\\
&&f'_1(x,y,z,t)=x^3y+y^3z+z^3t-t^3x+\sqrt{3}x^2z^2-\sqrt{3}y^2t^2+3\sqrt{3}xyzt.
\end{eqnarray*}
Then $\varphi'_2=(S')^{-1}\varphi_2(S')$ (resp. $\psi'_2=(S')^{-1}\psi_2(S')$) is an injective group homomorphism of $\SSS{5}$ (resp. $\A{5}$) into
$PGL_4(k)$ such that $\varphi'_2(s_5)=(\diag[\varepsilon,\varepsilon^2,\varepsilon^4,\varepsilon^3])$ and $\varphi'_2(t_1)=(R_{24}')$.
%%% theorem 3.13
\begin{theorem}
The nonsingular eigenspaces of $\varphi'_2(\SSS{5})$ in $\Form_{4,4}$ are $\langle f'_0\rangle$ and $\langle f'_1\rangle$. The nonsingular
eigenspace of $\psi'_2(\A{5})$ in $\Form_{4,4}$ is $\langle f'_0,f'_1\rangle$.
\end{theorem}
\begin{proof}
It is known that the dimension of the nonsingular eigenspace of $\psi'_2(\A{5})$ is equal to two \cite{mar}.
\end{proof}

Let $T_0=\diag[\alpha,\alpha^{-3},\alpha^{9},-\alpha^{-27}]$ and $T_1=T_0\diag[1,1,1,-1]$ with $\alpha^{80}=-1$. Then
\begin{eqnarray*}
&&f'_{0\ T_0^{-1}}=x^3y+y^3z+z^3t+t^3x+\frac{\sqrt{3}}{2}\theta(x^2z^2-y^2t^2)\ \ \ (\theta=\alpha^{20},\ hence \ \theta^4=-1),\\
&&f'_{1\ T_1^{-1}}=x^3y+y^3z+z^3t+t^3x+\sqrt{3}\theta(x^2z^2+y^2t^2) -3\sqrt{3}\theta^3 xyzt.
\end{eqnarray*}

%%%%% section 4
\section{$\ZZZ{7}$-invariant nonsingular quartic forms}
We shall describe $\ZZZ{7}$-invariant nonsingular quartic forms in $k[x,y,z,t]$. Let $p=7$ , $\varepsilon\in k^*$ with $\ord(\varepsilon)=7$, and
let diagonal matrices $D_j$ and $D_{j,\ell}$ be as in the section 2. Put
\begin{eqnarray*}
A_0=D_{0},\ A_1=D_1,\ A_2=D_2,\ A_3=D_3,\ A_4=D_6,\ A_5=D_{2,3},\ A_6=D_{2,4}.
\end{eqnarray*}
A subgroup of $PGL_4(k)$ isomorphic to $\ZZZ{p}$ is conjugate to one of these seven cyclic groups $\langle(A_i)\rangle$ by Lemma 2.12.
Assume $f^{[i,j]}\in \Form_{4,4}(A_i;\varepsilon^j)$ ($i,j\in [0,6]$).  Calculating indices of
singularity-checking quartic monomials for $A_i$, we obtain the following table.
{\scriptsize
\begin{eqnarray*}
\begin{array}{lcccccccccccccccc}
         &x^4 &x^3y &x^3z &x^3t &y^3x &y^4 &y^3z &y^3t &z^3x &z^3y &z^4 &z^3t &t^3x &t^3y &t^3z &t^4\\
    D_{0}&0   &0    &1    &0    &0    &0   &1    &0    &3    &3    &4   &3    &0    &0    &1    &0   \\
    D_{1}&0   &0    &1    &1    &0    &0   &1    &1    &3    &3    &4   &4    &3    &3    &4    &4   \\
    D_{2}&0   &0    &1    &2    &0    &0   &1    &2    &3    &3    &4   &5    &6    &6    &0    &1   \\
    D_{3}&0   &0    &1    &3    &0    &0   &1    &3    &3    &3    &4   &6    &2    &2    &3    &5   \\
    D_{6}&0   &0    &1    &6    &0    &0   &1    &6    &3    &3    &4   &2    &4    &4    &5    &3   \\
  D_{2,3}&0   &1    &2    &3    &3    &4   &5    &6    &6    &0    &1   &2    &2    &3    &4    &5   \\
  D_{2,4}&0   &1    &2    &4    &3    &4   &5    &0    &6    &0    &1   &3    &5    &6    &0    &2
  \end{array}
\end{eqnarray*}
}
We see easily that $V_p(f^{[i,j]})$ is singular at some $(e'_\ell)$ ($\ell\in [1,4]$) unless $[i,j]=[3,3]$ or $[i,j]=[6,0]$. Every
$f\in \Form_{4,4}(A_3;\varepsilon^3)\backslash\{0\}$ is singular. Indeed, $f$ takes the form $a_1x^3t+a_2y^3t+a_3z^3x+a_4z^3y+a_5t^3z+b_1x^2yt+b_2y^2xt$,
so we may assume $a_1a_2\not=0$, otherwise, $V_p(f)$ is singular at $(e'_1)$ or $(e'_2)$. Thus we may assume $f=x^3t+y^3t+az^3x+bz^3y+ct^3z+dx^2yt+ey^2xt$,
and $x^3+y^3+dx^2y+exy^2=(x-\alpha_1y)(x-\alpha_2y)(x-\alpha_3y)$ ($\alpha_i\in k^*$). So $V_p(f)$ is singular at $(\alpha_1,1,0,0)$. On the contrary
$\Form_{4,4}(A_6;1)$ contains a nonsingular element $x^4+y^3t+t^3z+z^3y$.
Any nonsingular form $f$ in
$\Form_{4,4}(A_6;1)=\langle x^4,y^3t,t^3z,z^3y,xyzt\rangle$ takes the form $ax^4+by^3t+ct^3z+dz^3y+exyzt$ with $abcd\not=0$.
We can easily find a nonsingular diagonal matrix $D=\diag[\alpha,\beta,\gamma,\delta]$ such that $f=f^{[6,0]}_{D^{-1}}=x^4+y^3t+t^3z+z^3y+\lambda xyzt$.
It is immediate that $f^\lambda=f_{T^{-1}}=x^3y+y^3z+z^3x+t^4+\lambda xyzt$ for $T=[e_4,e_3,e_2,e_1]$ and that $f^\lambda_{A^{-1}}=f^\lambda$ for
$A=T^{-1}A_6T=\diag[\varepsilon^4,\varepsilon^2,\varepsilon,1]$.

%% lemma 4.1
\begin{lemma}
The quartic form $f^\lambda$ is singular if and only if $(\lambda/4)^4=1$.
\end{lemma}
\begin{proof}
To begin with, for $D=\diag[1,1,1,\sqrt{-1}]$ we have $f^\lambda_{D^{-1}}=f^{\sqrt{-1}\lambda}$, and
\begin{eqnarray*}
&&f^\lambda_x=3x^2y+z^3+\lambda yzt,\ f^\lambda_y=x^3+3y^2z+\lambda xzt,\\
&&f^\lambda_z=y^3+3z^2x+\lambda xyt,\ f^\lambda_t=4t^3+\lambda xyz.
\end{eqnarray*}
Assume that the above four derivatives vanish at $[x,y,z,t]\in k^4\backslash\{0\}$. The equalities $xf^\lambda_x=yf^\lambda_y=zf^\lambda_z=0$ imply
$x^3y=y^3z=z^3x=-(\lambda xyzt)/4$. Now we see that $f^0$ is nonsingular. Assume $\lambda\not=0$. If $xyz=0$, then $t=0$, hence
$3x^2y+z^3=x^3+3y^2z=y^3+3z^2x=0$ so that $x=y=z=0$, namely $[x,y,z,t]=0$, a contradiction. Thus $xyzt\not=0$, for $t=0=f_t^\lambda$ imply $xyz=0$.
Now, since $tf_t^\lambda=0$, namely $t^4=-(\lambda xyzt)/4$, we have  $(xyzyt)^4=(xyzt)^4(-\lambda/4)^4$ , i.e., $(\lambda/4)^4=1$.
Conversely, $V_p(f^{-4})$ is singular at $(1,1,1,1)$.
\end{proof}

%%% theorem 4.2
\begin{theorem}
For a nonsingular quartic form $f(x,y,z,t)$  $|\Paut(f)|$ is divisible by $7$ if and only if  $f$ is projectively equivalent to
$f^{\lambda}$ for some $\lambda\in k$ such that $\lambda^4\not=4^4$.
\end{theorem}
\begin{proof}
We can argue as in the proof of Theorem 3.1.
\end{proof}

%%% proposition 4.3
\begin{proposition}
Let $G$ be the projective automorphism group of a nonsingular quartic form. In the decomposition $\Pi p^{\nu(p)}$ of $|G|$ into prime factors it holds
that $\nu(7)\leq 1$.
\end{proposition}
\begin{proof}
Let $c=\nu(7)$ and $p=7$. Assume $G=\Paut(f)$ and $c\geq 2$, where $f$ is a nonsingular quartic form. By Sylow theorem \cite{hal} there exists a
subgroup $H$ of $G$ such that $|H|=p^2$. Since any $\Z_{p^2}$-invariant $g\in \Form_{4,4}$ is singular by Theorem 5.1, $H$ is isomorphic to $\Z_p\times \Z_p$.
By Theorem 4.2 we may assume $f=f^\lambda$ for some $\lambda\in k$ with $\lambda^4\not=4^4$ and $(A)\in H$, where
$A=\diag[\varepsilon^4,\varepsilon^2,\varepsilon,1]$ with $\ord(\varepsilon)=p$.  We can directly show that any $X\in GL_4(k)$ satisfying
$(A)(X)=(X)(A)$ is diagonal. Thus, if $(B)\in H$, we may assume $B=\diag[\alpha,\beta,\gamma,1]$, and it follows that
$\alpha^3\beta=\beta^3\gamma=\gamma^3\alpha=1$, hence $B=\diag[\gamma^{-3},\gamma^9,\gamma,1]$ with $\gamma^{28}=1$. Consequently
$|H|\leq 28<p^2$, a contradiction.
\end{proof}

%%%%%

We shall describe all faithful representations i.e., injective group homomorphisms of $PSL_2(\F_7)$ into $PGL_4(k)$, up to equivalence.
We denote the group $PSL_2(\F_7)$ by $G$ in this section. A system of defining relations with respect to three generators $x$, $y$ and $z$ of
$G$ \cite[\S6 in chapter 2]{suz} is
\[
 x^7=y^3=z^2=1,\ y^{-1}xy=x^2,\ z^{-1}yz=y^{-1},\ zxz=x^{-1}zx^{-1}.
\]
For $x$, $y$ and $z$ we may take $(a)$, $(b)$ and $(c)$, where
\begin{eqnarray*}
a&=&\left[\begin{array}{cc} 1&1\\ 0&1
                             \end{array}
\right],\ \
b=\left[\begin{array}{cc} 5&0\\ 0&3
                             \end{array}
\right],\ \
c=\left[\begin{array}{cc} 0&-1\\ 1&0
                             \end{array}\right].
\end{eqnarray*}
If there exists a faithful representation $\varphi:\ G\rightarrow PGL_4(k)$, then $(A)=\varphi((a))$, $(B)=\varphi((b))$ and $(C)=\varphi((c))$ in
$PGL_4(k)$ are distinct and they satisfy
\begin{eqnarray*}
&& \ord((A))=7,\ \ord((B))=3,\ \ord((C))=2,\ (B)^{-1}(A)(B)=(A)^2,\\
&& (C)^{-1}(B)(C)=(B)^{-1},\  (C)(A)(C)=(A)^{-1}(C)(A)^{-1}.
\end{eqnarray*}
Conversely, if there exist distinct $(A)$, $(B)$ and $(C)$ in $PGL_4(k)$ satisfying above conditions, there exists
a faithful representation $\varphi:\ G\rightarrow PGL_4(k)$ such that $(A)=\varphi((a))$, $(B)=\varphi((b))$ and $(C)=\varphi((c))$.
Let $\varepsilon\in k^*$ with $\ord(\varepsilon)=7$. By Lemma 2.15 a subgroup of $PGL_4(k)$ isomorphic to $\Z_7$ is conjugate to one of the following groups
$\langle (A_j)\rangle$ ($j\in [0,6]$), where
\begin{eqnarray*}
A_0&=&\diag[1,1,\varepsilon,1],\ A_1=\diag[1,1,\varepsilon,\varepsilon],\ A_2=\diag[1,1,\varepsilon,\varepsilon^2],\ A_3=\diag[1,1,\varepsilon,\varepsilon^3],\\
A_4&=&\diag[1,1,\varepsilon,\varepsilon^6],\ A_5=\diag[1,\varepsilon,\varepsilon^2,\varepsilon^3],\ A_6=\diag[1,\varepsilon,\varepsilon^2,\varepsilon^4].
\end{eqnarray*}
Note that $\diag[\varepsilon^i,\varepsilon^j,\varepsilon^\ell,\varepsilon^m]^n=\diag[\delta^i,\delta^j,\delta^\ell,\delta^m]$, where $\delta=\varepsilon^n$.
We must find $X$, $Y$ and $Z$ in $GL_4(k)$ with $\ord(X)=7$, $\ord(Y)=3$ and $\ord(Z)=2$ for $A$, $B$ and $C$.
Since $\langle (X)\rangle$ is conjugate to $\langle (A_j)\rangle$ for some $j$, we may assume $X=A_j^i$ for some $i\in [1,6]$, hence $X=A_j$, replacing $
\varepsilon^i$ by $\varepsilon$. One can easily see that there exists nonsingular $Y$ such that $XY\sim YX^2$ only if $j=6$.
Therefore we may assume $X=\diag[\varepsilon^4,\varepsilon^2,\varepsilon,1]$ and $Y=\diag[\alpha,\beta,\gamma,1][e_3,e_1,e_2,e_4]$ with $\alpha\beta\gamma=1$,
for $\ord(Y)=3$. Since $D^{-1}YD=[e_3,e_1,e_2,e_4]$ for $D=\diag[1,\beta\gamma,\gamma,1]$, we may assume $Y=[e_3,e_1,e_2,e_4]$.

It remains to fix $Z$ such that $\ord((Z))=2$, $YZ\sim ZY^{-1}$ and $ZXZ \sim X^{-1}ZX^{-1}$.
The condition $YZ=\lambda ZY^{-1}$ for some $\lambda\in k^*$ implies
\begin{eqnarray*}
 &&Z=\left[\begin{array}{cccc}\alpha& \beta& \gamma & \mu\\
                            \beta& \gamma & \alpha& \mu\lambda\\
                            \gamma&\alpha & \beta & \mu\lambda^2\\
                            \nu & \nu     & \nu   & \delta\end{array}
    \right]\diag[1,\lambda^2,\lambda,1],
\end{eqnarray*}
where $\lambda^3=1$ and $\delta=\delta\lambda$. Since the $(4,4)$ components of $ZXZ$ and $X^{-1}ZX^{-1}$ are
$(\varepsilon+\varepsilon^2+\varepsilon^4)\mu\nu+\delta^2$ and $\delta$ respectively, $\delta\not=0$, for $\delta=0$ implies that $Z$ is singular.
Thus $\lambda=1$, and we may assume $\delta=1$. If $\mu=0$, then $\nu=0$. Indeed, $\mu=0$ implies $ZXZ=X^{-1}ZX^{-1}$, hence if $\nu\not=0$,  the components
$(4,j)$ ($j\in [1,3]$) give
\begin{eqnarray*}
&&\varepsilon^4\alpha+\varepsilon^2\beta+\varepsilon\gamma+1=\varepsilon^3,\
\varepsilon^4\beta+\varepsilon^2\gamma+\varepsilon\alpha+1=\varepsilon^5,\
\varepsilon^4\gamma+\varepsilon^2\alpha+\varepsilon\beta+1=\varepsilon^6,
\end{eqnarray*}
therefore $(\varepsilon^4+\varepsilon^2+\varepsilon)(\alpha+\beta+\gamma)=\varepsilon^3+\varepsilon^5+\varepsilon^6-3$.
On the other hand, $(4,1)$ component of $Z^2\sim E_4$ yields $\alpha+\beta+\gamma+1=0$, a contradiction. Similarly, if $\nu=0$, then $\mu=0$. If
$\mu\nu\not=0$, then there exists $\eta\not=0$ such that  $\diag[\eta,\eta,\eta,1]^{-1}\ Z\ \diag[\eta,\eta,\eta,1]$ is a symmetric matrix whose
$(4,j)$ components $(j\in [1,3])$ are equal to $\tau$. Thus we must find $Z$ with the form
\begin{eqnarray*}
 Z=\left[\begin{array}{cccc}\alpha& \beta& \gamma & \tau\\
                            \beta& \gamma & \alpha& \tau\\
                            \gamma&\alpha & \beta & \tau\\
                            \tau & \tau   & \tau   & 1\end{array}
    \right]
\end{eqnarray*}
such that $Z^2\sim E_4$ and $ZXZ\sim X^{-1}ZX^{-1}$. We have
\begin{eqnarray*}
&&Z^2=\left[\begin{array}{cccc}\zeta_1&\zeta_2&\zeta_2&\zeta_3\\
                             \zeta_2&\zeta_1&\zeta_2&\zeta_3\\
                             \zeta_2&\zeta_2&\zeta_1&\zeta_3\\
                             \zeta_3&\zeta_3&\zeta_3&\zeta_4\end{array}
            \right],\ {\rm where\ }\left[\begin{array}{c}\zeta_1\\\zeta_2\\\zeta_3\\\zeta_4\end{array}\right]
                    =\left[\begin{array}{c}\alpha^2+\beta^2+\gamma^2+\tau^2\\
                                           \alpha\beta+\beta\gamma+\gamma\alpha+\tau^2\\
                                           (\alpha+\beta+\gamma+1)\tau\\
                                           3\tau^2+1\end{array}
                                                               \right],\\
&&X^{-1}ZX^{-1}=\left[\begin{array}{cccc}\varepsilon^6\alpha&\varepsilon\beta   &\varepsilon^2\gamma&\varepsilon^3\tau\\
                                       \varepsilon\beta   &\varepsilon^3\gamma&\varepsilon^4\alpha&\varepsilon^5\tau\\
                                       \varepsilon^2\gamma&\varepsilon^4\alpha&\varepsilon^5\beta &\varepsilon^6\tau\\
                                       \varepsilon^3\tau  &\varepsilon^5\tau  &\varepsilon^6\tau  &1\end{array}
                 \right],\ \ \
ZXZ=\left[\begin{array}{cccc}\varepsilon_1&\varepsilon_4&\varepsilon_5&\eta_1\\
                             \varepsilon_4&\varepsilon_2&\varepsilon_6&\eta_2\\
                             \varepsilon_5&\varepsilon_6&\varepsilon_3&\eta_3\\
                             \eta_1       &\eta_2       &\eta_3       &\eta_4\end{array}\right]\ {\rm with\ }\\
%&&=\left[\begin{array}{cccc}
%\varepsilon^4\alpha^2+\varepsilon^2\beta^2+\varepsilon\gamma^2+\tau^2&
%\varepsilon^4\alpha\beta+\varepsilon^2\beta\gamma+\varepsilon\gamma\alpha+\tau^2&
%\varepsilon^4\gamma\alpha+\varepsilon^2\alpha\beta+\varepsilon\beta\gamma+\tau^2&
%\eta_1\\%(\varepsilon^4\alpha+\varepsilon^2\beta+\varepsilon\gamma+1)\tau\\
%\varepsilon^4\alpha\beta+\varepsilon^2\beta\gamma+\varepsilon\gamma\alpha+\tau^2&
%\varepsilon^4\beta^2+\varepsilon^2\gamma^2+\varepsilon\alpha^2+\tau^2&
%\varepsilon^4\beta\gamma+\varepsilon^2\gamma\alpha+\varepsilon\alpha\beta+\tau^2&
%\eta_2\\%(\varepsilon^4\beta+\varepsilon^2\gamma+\varepsilon\alpha+1)\tau\\
%\varepsilon^4\gamma\alpha+\varepsilon^2\alpha\beta+\varepsilon\beta\gamma+\tau^2&
%\varepsilon^4\beta\gamma+\varepsilon^2\gamma\alpha+\varepsilon\alpha\beta+\tau^2&
%\varepsilon^4\gamma^2+\varepsilon^2\alpha^2+\varepsilon\beta^2+\tau^2&
%\eta_3\\%(\varepsilon^4\gamma+\varepsilon^2\alpha+\varepsilon\beta+1)\tau\\
%\eta_1&\eta_2&\eta_3&\eta_4\end{array}\right] \\
&&\left[\begin{array}{c}\varepsilon_1\\
                                    \varepsilon_2\\
                                    \varepsilon_3\\
                                    \varepsilon_4\\
                                    \varepsilon_5\\
                                    \varepsilon_6\end{array}\right]
                                  =\left[\begin{array}{c}\varepsilon^4\alpha^2+\varepsilon^2\beta^2+\varepsilon\gamma^2+\tau^2\\
                                                         \varepsilon^4\beta^2+\varepsilon^2\gamma^2+\varepsilon\alpha^2+\tau^2\\
                                                         \varepsilon^4\gamma^2+\varepsilon^2\alpha^2+\varepsilon\beta^2+\tau^2\\
                                                         \varepsilon^4\alpha\beta+\varepsilon^2\beta\gamma+\varepsilon\gamma\alpha+\tau^2\\
                                                         \varepsilon^4\gamma\alpha+\varepsilon^2\alpha\beta+\varepsilon\beta\gamma+\tau^2\\
                                                         \varepsilon^4\beta\gamma+\varepsilon^2\gamma\alpha+\varepsilon\alpha\beta+\tau^2\end{array}\right],\
\left[\begin{array}{c}\eta_1\\\eta_2\\\eta_3\\\eta_4\end{array}\right]=\left[\begin{array}{c}
(\varepsilon^4\alpha+\varepsilon^2\beta+\varepsilon\gamma+1)\tau\\
(\varepsilon^4\beta+\varepsilon^2\gamma+\varepsilon\alpha+1)\tau\\
(\varepsilon^4\gamma+\varepsilon^2\alpha+\varepsilon\beta+1)\tau\\
(\varepsilon^4+\varepsilon^2+\varepsilon)\tau^2+1\end{array}
               \right].
\end{eqnarray*}
%% lemma 4.4
\begin{lemma} $Z=Z(\tau,\alpha,\beta,\gamma)$ satisfies $Z^2\sim E_4$ and $ZXZ\sim X^{-1}ZX^{-1}$ if and only if
\begin{eqnarray*}
\left[
\begin{array}{c}\tau\\
                \alpha\\
                \beta\\
                \gamma\end{array}\right]&=&\left[\begin{array}{l}\pm\sqrt{2}\\
                                                        \varepsilon+\varepsilon^6\\
                                                        \varepsilon^2+\varepsilon^5\\
                                                        \varepsilon^3+\varepsilon^4\end{array}\right] {\rm \ or\ \ }
\left[\begin{array}{c}\tau\\
                  \alpha\\
                  \beta\\
                  \gamma\end{array}\right]=\left[\begin{array}{l}0\\
                                               \frac{1}{7}(-2-\varepsilon+2\varepsilon^2+2\varepsilon^5-\varepsilon^6)\\
                                               \frac{1}{7}(-2-\varepsilon^2+2\varepsilon^3+2\varepsilon^4-\varepsilon^5)\\
                                               \frac{1}{7}(-2+2\varepsilon-\varepsilon^3-\varepsilon^4+2\varepsilon^6)\end{array}\right].
\end{eqnarray*}
\end{lemma}
\begin{proof}
First assume $\tau\not=0$.
Since the (1,4) component of $Z^2$ vanishes, $\alpha+\beta+\gamma=-1$. In addition the condition $ZXZ\sim X^{-1}ZX^{-1}$ implies
 $\varepsilon^2(ZXZ)_{14}=(ZXZ)_{24}$ and
$\varepsilon^3(ZXZ)_{14}=(ZXZ)_{34}$, namely
\[
\varepsilon^2(\varepsilon^4\alpha+\varepsilon^2\beta+\varepsilon\gamma+1)=\varepsilon^4\beta+\varepsilon^2\gamma+\varepsilon\alpha+1,\
\varepsilon^3(\varepsilon^4\alpha+\varepsilon^2\beta+\varepsilon\gamma+1)=\varepsilon^4\gamma+\varepsilon^2\alpha+\varepsilon\beta+1.
\]
Therefore
\begin{eqnarray*}
\left[\begin{array}{ccc}1     &1      &1   \\
                        \varepsilon^6-\varepsilon&0                        &\varepsilon^3-\varepsilon^2\\
                        1-\varepsilon^2          &\varepsilon^5-\varepsilon&0\end{array}\right]
\left[\begin{array}{c}\alpha\\\beta\\\gamma\end{array}\right]
&=&\left[\begin{array}{c}-1\\1-\varepsilon^2\\1-\varepsilon^3\end{array}\right],
\end{eqnarray*}
hence,
\begin{eqnarray*}
\left[\begin{array}{c}\alpha\\\beta\\\gamma\end{array}\right]
&=&\frac{1}{d}\left[\begin{array}{l}-2-\varepsilon^2+3\varepsilon^3-\varepsilon^4+2\varepsilon^5-\varepsilon^6\\
                                    -1\varepsilon+3\varepsilon^2-2\varepsilon^3-\varepsilon^4+2\varepsilon^6\\
                                    3+2\varepsilon-2\varepsilon^2-\varepsilon^3-\varepsilon^4-\varepsilon^5\end{array}\right]
=\left[\begin{array}{c}\varepsilon+\varepsilon^6\\
                       \varepsilon^2+\varepsilon^5\\
                       \varepsilon^3+\varepsilon^4\end{array}\right],
\end{eqnarray*}
where $d=-\varepsilon+3\varepsilon^4-\varepsilon^5-\varepsilon^6$ so that $d^{-1}=(1+4\varepsilon^3+\varepsilon^4+\varepsilon^5)/14$.
Since $Z^2\sim E_4$ implies $\alpha^2+\beta^2+\gamma^2+\tau^2=3\tau^2+1$, we have $\tau^2=2$. Now $Z^2=7E_4$, for $\alpha^2+\beta^2+\beta^2=5$.
Using equalities $\alpha\beta\gamma=1$, $\alpha\beta=\gamma+\alpha$, $\beta\gamma=\alpha+\beta$ and $\gamma\alpha=\beta+\gamma$, one can verify
$(ZXZ)_{ij}/(X^{-1}ZX^{-1})_{ij}=1+2(\varepsilon+\varepsilon^2+\varepsilon^4)$ for $[i,j]$ such that $1\geq i\geq j\geq 4$, namely
$ZXZ=(1+2\varepsilon+2\varepsilon^2+2\varepsilon^4)X^{-1}ZX^{-1}$.

Second assume $\tau=0$. Clearly $Z^2\sim E_4$ if and only if $\alpha^2+\beta^2+\gamma^2=1$ and $\alpha\beta+\beta\gamma+\gamma\alpha=0$, namely
$(\alpha+\beta+\gamma)^2=1$ and $\alpha\beta+\beta\gamma+\gamma\alpha=0$, hence $Z^2=E_4$ . Let $\theta=\alpha+\beta+\gamma\in \{-1,1\}$. Meanwhile
$ZXZ\sim X^{-1}ZX^{-1}$ if and only if $ZXZ=X^{-1}ZX^{-1}$, namely
\begin{eqnarray*}
T\left[\begin{array}{c}\alpha^2\\\beta^2\\\gamma^2\end{array}\right]
&=&\left[\begin{array}{c}\varepsilon^6\alpha\\\varepsilon^3\gamma\\\varepsilon^5\beta\end{array}\right],\
T\left[\begin{array}{c}\alpha\beta\\\beta\gamma\\\gamma\alpha\end{array}\right]
=\left[\begin{array}{c}\varepsilon\beta\\\varepsilon^4\alpha\\\varepsilon^2\gamma\end{array}\right],{\rm \ where\ }
T=\left[\begin{array}{lll}
               \varepsilon^4&\varepsilon^2&\varepsilon\\
               \varepsilon  &\varepsilon^4&\varepsilon^2\\
               \varepsilon^2&\varepsilon  &\varepsilon^4\end{array}\right].
\end{eqnarray*}
Denoting the adjugate matrix of $T$ by $\tilde{T}$,  we have
\begin{eqnarray*}
&&\det T\left[\begin{array}{c}\alpha^2\\\beta^2\\\gamma^2\end{array}\right]
              =\tilde{T}\left[\begin{array}{c}\varepsilon^6\alpha\\\varepsilon^3\gamma\\\varepsilon^5\beta\end{array}\right],\
\det T \left[\begin{array}{c}\alpha\beta\\\beta\gamma\\\gamma\alpha\end{array}\right]
              =\tilde{T}\left[\begin{array}{c}\varepsilon\beta\\\varepsilon^4\alpha\\\varepsilon^2\gamma\end{array}\right],{\rm \ where\ }\\
&&\tilde{T}=\left[\begin{array}{lll}
               \varepsilon-\varepsilon^3& \varepsilon^2-\varepsilon^6&\varepsilon^4-\varepsilon^5\\
               \varepsilon^4-\varepsilon^5&\varepsilon-\varepsilon^3 &\varepsilon^2-\varepsilon^6\\
               \varepsilon^2-\varepsilon^6&\varepsilon^4-\varepsilon^5&\varepsilon-\varepsilon^3\end{array}\right].
\end{eqnarray*}
Since $\alpha^2+\beta^2+\gamma^2=1$, $\alpha\beta+\beta\gamma+\gamma\alpha=0$,  $\gamma=\theta-\alpha-\beta$,
and $\det T=-3+\varepsilon^3+\varepsilon^5+\varepsilon^6$, we have
\begin{eqnarray*}
&&\left[\begin{array}{ll}
        2\varepsilon+\varepsilon^3-2\varepsilon^4-2\varepsilon^5+\varepsilon^6& 2\varepsilon^2-\varepsilon^3-2\varepsilon^4-\varepsilon^5+2\varepsilon^6\\
        2\varepsilon-\varepsilon^2-2\varepsilon^3-\varepsilon^4+2\varepsilon^5& \varepsilon+\varepsilon^2-2\varepsilon^4+2\varepsilon^5-2\varepsilon^6
         \end{array}\right]\left[\begin{array}{c}\alpha\\\beta\end{array}\right]\\
  &&=\left[\begin{array}{l}\theta(-1+\varepsilon+\varepsilon^2-\varepsilon^4-\varepsilon^5+\varepsilon^6)-3+\varepsilon^3+\varepsilon^5+\varepsilon^6\\
                          \theta(1+\varepsilon-\varepsilon^3-\varepsilon^4+\varepsilon^5-\varepsilon^6)\end{array}\right].
\end{eqnarray*}
Hence, denoting the determinant of the coefficient matrix of $\alpha,\beta$ by $\delta$ which equals to $7(-3+\varepsilon+\varepsilon^2+\varepsilon^4)$,
we have
\begin{eqnarray*}
&&\delta\left[\begin{array}{c}\alpha\\\beta\end{array}\right]
=\left[\begin{array}{ll}\varepsilon+\varepsilon^2-2\varepsilon^4+2\varepsilon^5-2\varepsilon^6&
                                 -2\varepsilon^2+\varepsilon^3+2\varepsilon^4+\varepsilon^5-2\varepsilon^6\\
                             -2\varepsilon+\varepsilon^2+2\varepsilon^3+\varepsilon^4-2\varepsilon^5&
                                  2\varepsilon+\varepsilon^3-2\varepsilon^4-2\varepsilon^5+\varepsilon^6\end{array}\right]\\
&&\times \left[\begin{array}{l}\theta(-1+\varepsilon+\varepsilon^2-\varepsilon^4-\varepsilon^5+\varepsilon^6)-3+\varepsilon^3+\varepsilon^5+\varepsilon^6\\
                          \theta(1+\varepsilon-\varepsilon^3-\varepsilon^4+\varepsilon^5-\varepsilon^6)\end{array}\right].
\end{eqnarray*}
Consequently
\[
\left[\begin{array}{c}
           \delta\alpha\\
           \delta\beta\\
           \delta\gamma\end{array}\right]
=7\left[ \begin{array}{l}
           \theta(-1+\varepsilon^4)-\varepsilon^2+\varepsilon^4-\varepsilon^5+\varepsilon^6\\
           \theta(-1+\varepsilon)+\varepsilon-\varepsilon^3-\varepsilon^4+\varepsilon^5\\
           \theta(-1+\varepsilon^2)-\varepsilon+\varepsilon^2+\varepsilon^3-\varepsilon^6\end{array}\right].
 \]
If $\theta=1$, then one can show that $\delta^2(X^{-1}ZX^{-1})_{11}=\delta^2\varepsilon^6\alpha$ and
$\delta^2(ZXZ)_{11}=\delta^2(\varepsilon^4\alpha^2+\varepsilon^2\beta^2+\varepsilon\gamma^2)$ are distinct, hence $X^{-1}ZX^{-1}=ZXZ$
does not hold. Therefore $\theta\not=1$. Assume $\theta=-1$. Since $\delta^{-1}=-(4+\varepsilon+\varepsilon^2+\varepsilon^4)/98$, it follows that
\begin{eqnarray*}
&&\left[\begin{array}{c}\alpha\\\beta\\\gamma\end{array}\right]=\left[\begin{array}{c}\alpha_{0}\\\beta_{0}\\\gamma_{0}\end{array}\right]
=\frac{1}{7}\left[\begin{array}{l}
      -2-\varepsilon+2\varepsilon^2+2\varepsilon^5-\varepsilon^6\\
      -2-\varepsilon^2+2\varepsilon^3+2\varepsilon^4-\varepsilon^5\\
      -2+2\varepsilon-\varepsilon^3-\varepsilon^4+2\varepsilon^6\end{array}\right], \\
&&\left[\begin{array}{c}
        \alpha^2\\
        \beta^2\\
        \gamma^2\end{array}\right]
=\frac{1}{7}\left[\begin{array}{l}
                  2-\varepsilon^2-\varepsilon^5\\
                  2-\varepsilon^3-\varepsilon^4\\
                  2-\varepsilon-\varepsilon^6\end{array} \right],\
  \left[\begin{array}{c}
        \alpha\beta\\
        \beta\gamma\\
        \gamma\alpha\end{array}\right]
=\frac{1}{7}\left[\begin{array}{l}
        \varepsilon-\varepsilon^3-\varepsilon^4+\varepsilon^6\\
        -\varepsilon+\varepsilon^2+\varepsilon^5-\varepsilon^6\\
        -\varepsilon^2+\varepsilon^3+\varepsilon^4-\varepsilon^5\end{array}\right].
\end{eqnarray*}
Now one can see easily that $Z^2=E_4$ and $X^{-1}ZX^{-1}=ZXZ$.
\end{proof}

We are now in a position to describe all faithful representations of $PSL_2(\F_7)$ up to equivalence.
Let $A=\diag[\varepsilon^4,\varepsilon^2,\varepsilon,1]$, $B=[e_3,e_1,e_2,e_4]$ where $e_i$ stands for the $i$-th column vector of the unit matrix $E_4$.
Denote by $C_{\pm\sqrt{2}}$ and $C_{0}$ the matrices $Z(\tau,\alpha,\beta,\gamma)$ for
$[\tau,\alpha,\beta,\gamma]=[\pm\sqrt{2},\varepsilon+\varepsilon^6,\varepsilon^2+\varepsilon^5,\varepsilon^3+\varepsilon^4]$ and
$[\tau,\alpha,\beta,\gamma]=[0,\alpha_{0},\beta_{0},\gamma_{0}]$ respectively. There exist faithful representations $\varphi_{\pm\sqrt{2}}$ of $PSL_2(\F_7)$ in
$PGL_4(k)$ such that:\\ $\varphi_{\pm\sqrt{2}}((a))=(A)$, $\varphi_{\pm\sqrt{2}}((b))=(B)$ and $\varphi_{\pm}((c))=(C_{\pm\sqrt{2}})$. There exists a
faithful representation $\varphi_{0}$ of $PSL_4(\F_7)$ in  $PGL_4(k)$ such that $\varphi_{0}((a))=(A)$, $\varphi_{0}((b))=(B)$ and $\varphi_{0}((c))=(C_{0})$.
Moreover, any faithful representation $\varphi$ of $PSL_2(\F_7)$ in $PGL_4(k)$ is equivalent to one of $\varphi_{\pm\sqrt{2}}$ and $\varphi_{0}$, i.e.,
there exists an inner automorphism $\sigma$ of $PGL_4(k)$ such that $\sigma\circ\varphi$ is equal to $\varphi_{\pm\sqrt{2}}$ or $\varphi_{0}$.
Note that $\varphi_{\sqrt{2}}$ and $\varphi_{-\sqrt{2}}$ are equivalent, for $\varphi_{-\sqrt{2}}=(\diag[-1,-1,-1,1])\varphi_{\sqrt{2}}(\diag[-1,-1,-1,1])$.
However $\varphi_{\sqrt{2}}$ and $\varphi_{0}$ are not equivalent,
for $G_{\sqrt{2}}=\Ima\ \varphi_{\sqrt{2}}$ does not fix a point in $P^3(k)$, while
$G_{0}=\Ima\ \varphi_{0}$ fixes the point $(0,0,0,1)$ in $P^3(k)$. Replacing $\varepsilon$ by $\varepsilon^i$ ($i\in [1,6]$), we obtain faithful
representations $\varphi_{\sqrt{2},i}$ and $\varphi_{0,i}$ of $PSL_2(\F_7)$. Obviously $\varphi_{\sqrt{2}}=\varphi_{\sqrt{2},1}$ and
$\varphi_{0}=\varphi_{0,1}$.
%%% proposition 4.5
\begin{proposition} Let the representations $\varphi_{\sqrt{2}}$,  $\varphi_{\sqrt{2},i}$, $\varphi_{0}$ and $\varphi_{0,i}$ of $PSL_2(\F_7)$ in $PGL_4(k)$ be
as above, and $I_1=\{1,2,4\}$, $I_2=\{3,5,6\}$. Denote the automorphism of $PGL_4(k)$ sending $(A)$ to $({A'}^{-1})$ by $\kappa$, where $A'$ stands for
the transposed matrix of $A\in GL_4(k)$.\\
$(1)$ The representations $\varphi_{\sqrt{2},i}$ $(i\in I_1)$ are equivalent. The representations $\varphi_{\sqrt{2},i}$ $(i\in I_2)$ are equivalent.
The representations $\varphi_{\sqrt{2},1}$ and $\varphi_{\sqrt{2},6}$ are not equivalent, but $\varphi_{\sqrt{2},6}=\kappa\circ\varphi_{\sqrt{2},1}$ and
$\Ima\ \varphi_{\sqrt{2},1}=\Ima\ \varphi_{\sqrt{2},6}$. \\
$(2)$ The representations $\varphi_{0,i}$ $(i\in I_1)$ are equivalent. The representations $\varphi_{0,i}$ $(i\in I_2)$ are equivalent.
The representations $\varphi_{0,1}$ and $\varphi_{0,6}$ are not equivalent, but $\varphi_{0,6}=\kappa\circ\varphi_{0,1}$ and
$\Ima\ \varphi_{0,1}=\Ima\ \varphi_{0,6}$. \\
$(3)$ Any faithful representation of $PSL_2(\F_7)$ ind $PGL_4(k)$ is equivalent to one of $\varphi_{\sqrt{2},1}$, $\varphi_{\sqrt{2},6}$,
$\varphi_{0,1}$ and $\varphi_{0,6}$. \\
$(4)$ A subgroup of $PGL_4(k)$ isomorphic to $PSL_2(\F_7)$ is conjugate to $\Ima\ \varphi_{\sqrt{2}}$ or $\Ima\ \varphi_{0}$.
\end{proposition}
\begin{proof}
(1) Let $\sigma=(123)\in \SSS{4}$, so that $B^{-1}=\hat{\sigma}$.
Substituting $\varepsilon^i$ for $\varepsilon$ which determines $A$ and $C_{\sqrt{2}}$, we obtain matrices $A_i$ and $C_{\sqrt{2},i}$ as follows.
\begin{eqnarray*}
&&A_1=\diag[\varepsilon^4,\varepsilon^2,\varepsilon,1],\ A_2=\diag[\varepsilon,\varepsilon^4,\varepsilon^2,1],\ A_3=\diag[\varepsilon^5,\varepsilon^6,\varepsilon^3,1],\\
&&A_4=\diag[\varepsilon^2,\varepsilon,\varepsilon^4,1],\ A_5=\diag[\varepsilon^6,\varepsilon^3,\varepsilon^5,1],\ A_6=\diag[\varepsilon^3,\varepsilon^5,\varepsilon^6,1],\\
&&C_{\sqrt{2},1}=C(\sqrt{2},\alpha,\beta,\gamma),\ C_{\sqrt{2},2}=C(\sqrt{2},\beta,\gamma,\alpha,),\ C_{\sqrt{2},3}=C(\sqrt{2},\gamma,\alpha,\beta),\\
&&C_{\sqrt{2},4}=C(\sqrt{2},\gamma,\alpha,\beta),\ C_{\sqrt{2},5}=C(\sqrt{2},\beta,\gamma,\alpha,),\ C_{\sqrt{2},6}=C(\sqrt{2},\alpha,\beta,\gamma).
\end{eqnarray*}
Therefore, $B^{-i}BB^i=B$ for any integer $i$, and
\begin{eqnarray*}
&& B^{-1}A_1B=A_2,\ B^{-1}C_{\sqrt{2},1}B=C_{\sqrt{2},2},\ \ B^{-2}A_1B^2=A_4,\ B^{-2}C_{\sqrt{2},1}B^2=C_{\sqrt{2},4},\\
&& B^{-1}A_6B=A_5,\ B^{-1}C_{\sqrt{2},6}B=C_{\sqrt{2},5},\ \ B^{-2}A_6B^2=A_3,\ B^{-2}C_{\sqrt{2},6}B^2=C_{\sqrt{2},3}.
\end{eqnarray*}
In addition $\kappa((A_1))=(A_6)=(A_1)^6$, $\kappa((B))=(B)$, and $\kappa((C_{\sqrt{2},1}))=(C_{\sqrt{2},1})=(C_{\sqrt{2},6})$, so that
$\varphi_{\sqrt{2},6}=\kappa\circ\varphi_{\sqrt{2},1}$. In particular, $\kappa((A_1))$, $\kappa((B))$ and $\kappa((C_{\sqrt{2},1}))$ belong to
$\Ima\ \varphi_{\sqrt{2},1}$, hence the finite group $\Ima\ \varphi_{\sqrt{2},1}$ is $\kappa$-invariant. Let
$[a_1,...,a_4]=[\varepsilon^4,\varepsilon^2,\varepsilon,1]$ and $[b_1,...,b_4]=[\varepsilon^3,\varepsilon^5,\varepsilon^6,1]$. Then there exists
a $\sigma\in \SSS{4}$ such that $[a_1,...,a_4]\sim [b_{\sigma(1)},...,b_{\sigma(4)}]$ if and only if $\{a_1,...,a_4\}={b_{i}/b_{j}:\ i\in [1,4]}$
for some $j\in [1,4]$. The last equality is impossible, for
\[
[b_1,...,b_4]=[\varepsilon^3,\varepsilon^5,\varepsilon^6,1]\sim [\varepsilon^4,\varepsilon^6,1,\varepsilon]
\sim[\varepsilon^5,1,\varepsilon,\varepsilon^2]\sim[1,\varepsilon^2,\varepsilon^3,\varepsilon^4].
\]
Consequently no $T\in GL_4(k)$ satisfies $(A_1)=(T^{-1})(A_6)(T)$ by Lemma 2.1. Hence $\varphi_{\sqrt{2},1}$ and $\varphi_{\sqrt{2},6}$ are not
equivalent.
One can show (2) similarly. Now (3) and (4) follow, for $\varphi_{\sqrt{2}}=\varphi_{\sqrt{2},1}$ and $\varphi_{0}=\varphi_{0,1}$.
\end{proof}

%% remark 4.6
\begin{remark}
Edge describes a representation of $PSL_2(\F_7)$ in $PGL_4(k)$ equivalent to $\varphi_{\sqrt{2}}$ {\rm \cite[p.166]{edg}}. The representation $\varphi_{0}$ is
essentially the representation of $PSL_2(\F_7)$ in $PGL_3(k)$ {\rm \cite[p.54]{lev}}.
\end{remark}

Let $\delta=-\varepsilon-\varepsilon^2-\varepsilon^4+\varepsilon^3+\varepsilon^5+\varepsilon^6$, hence $\delta^2=-7$. Then
\[
 \delta C_{0}= \left[\begin{array}{cccc}
                    \varepsilon-\varepsilon^6&\varepsilon^2-\varepsilon^5&\varepsilon^4-\varepsilon^3&0\\
                    \varepsilon^2-\varepsilon^5&\varepsilon^4-\varepsilon^3&\varepsilon-\varepsilon^6&0\\
                    \varepsilon^4-\varepsilon^3&\varepsilon-\varepsilon^6&\varepsilon^2-\varepsilon^5  &0\\
                    0                          &0                       &0                           &\delta
                  \end{array}\right].
\]
For a square matrix $X=[x_{ij}]$ ($i,j\in [1,4]$) $\lceil X\rceil$ stands for the square matrix $[x_{i,j}]$ ($i,j\in [1,3]$). It is easy to see that
$(\lceil A_i\rceil)$, $(\lceil B\rceil)$, $(\lceil \delta C_{0}\rceil)\in PGL_3(k)$ satisfy the defining relations of $PSL_2(\F_7)$ for $i\in \{1,6\}$.
There exists a faithful representation $\psi_i$ of $PSL_2(\F_7)$ such that $\psi_i((a))=(\lceil A_i\rceil)$, $\psi((b))=(\lceil B\rceil)$ and
$\psi((c))=(\lceil \delta C_{0}\rceil)$. Note that $\psi_i(PSL_2(\F_7))$ fixes no points of $P^2(k)$. One can easily see that 1) the representations
$\psi_1$ and $\psi_6$ are not equivalent, 2) $\psi_6=\kappa'\circ\psi_1$,  3) $\Ima\ \psi_6=\Ima\ \psi_1$, and $\Ima\ \psi_1=\Ima\ \psi_6$. Here $\kappa'$ is
the automorphism of $PGL_3(k)$ such that $\kappa'((A))=(A'^{-1})$ for $A\in GL_3(k)$, where $A'$ is the transposed matrix of $A$.
%% corollary 4.7
\begin{corollary}
A faithful representation of $PSL_2(\F_7)$ in $PGL_3(k)$ is equivalent to the representation $\psi_1$ or $\psi_6$.
In particular, a subgroup of $PGL_3(k)$ isomorphic to $PSL_2(\F_7)$ is conjugate to $\Ima\ \psi_1$. There exist no faithful representations
of $PSL_2(\F_7)$ in $PGL_2(k)$.
\end{corollary}
\begin{proof}
Let $\psi'$ be a faithful representation of $PSL_2(\F_7)$ in  $PGL_3(k)$ such that $\psi'((a))=(X)$, $\psi'((b))=(Y)$ and $\psi'((c))=(Z)$
with $\ord(X)=7$, $\ord(Y)=3$ and $\ord(Z)=2$, where $X,Y,Z\in GL_4(k)$.
Then we obtain a faithful representation $\varphi'$ of $PSL_2(\F_7)$  in $PGL_4(k)$ such that $\varphi'((a))=(X\oplus 1)$,
$\varphi'((b))=(Y\oplus 1)$, and $\varphi'((c))=(Z\oplus 1)$. Since $\varphi'(PSL_2(\F_7))$ fixes the point $(0,0,0,1)$, the representations
$\varphi'$ is equivalent to $\varphi_{0,i}$ for some $i\in\{1,6\}$: there exists an $S\in GL_4(k)$ such that $S^{-1}A_iS\sim X\oplus 1$,
$S^{-1}BS\sim Y\oplus 1$ and $S^{-1}C_{0,i}S\sim Z\oplus 1$. Any nonzero column vector $u\in k^4$ such that
$A_i u\sim u$, $Bu\sim u$ and $C_{0,i}u\sim u$ is proportional to $e_4$, for $A_i$ has four 1-dimensional eigenspaces $\langle e_j\rangle$ ($j\in [1,4]$).
Consequently the common eigenspace of $\{S^{-1}A_iS,S^{-1}BS,S^{-1}C_{0,i}S\}$ is $\langle S^{-1}e_4\rangle$, while it coincides with the common eigenspace
of $\{X\oplus 1,Y\oplus 1,Z\oplus 1\}$ which contain $e_4$. Thus $S^{-1}e_4\sim e_4$.
Similarly any nonzero row vector $v\in k^4$ such that $vA_i\sim v$, $vB\sim v$ and $vC_{0,i}\sim v$ is proportional to ${e_4}'$, the fourth row vector
 of $E_4$. Consequently $e_4'S\sim e_4'$. Thus we may assume $S=T\oplus 1$ for some $T\in GL_3(k)$. Now
$T^{-1}\lceil A_i\rceil T\sim X$, $T^{-1}\lceil B\rceil T\sim Y $ and $T^{-1}\lceil \delta C_{0,i}\rceil T\sim Z$. Hence $\psi_i$ and $\psi'$ are equivalent.
The equality $\Ima\ \psi_1=\Ima\ \psi_6$ is already shown.
Suppose that there exists a faithful representation $\eta$ of $PSL_2(\F_7)$ in  $PGL_2(k)$. Then we obtain a faithful representation $\psi'$ of
$PSL_2(\F_7)$ in  $PGL_3(k)$ in a trivial manner. However $\psi'(PSL_2(\F_7))$ fixes the point $(0,0,1)$, a contradiction.
\end{proof}

Let $A_3$, $B_3$ and $C_3\in SL_3(k)$ be as follows.
\begin{eqnarray*}
A_3=\diag[\varepsilon^4,\varepsilon^2,\varepsilon],\ B_3=\left[\begin{array}{ccc}0&1&0\\ 0&0&1\\ 1&0&0\end{array}\right],\
C_3=\frac{1}{\sqrt{-7}}\left[\begin{array}{ccc}
                             \varepsilon-\varepsilon^6& \varepsilon^2-\varepsilon^5& \varepsilon^4-\varepsilon^3\\
                             \varepsilon^2-\varepsilon^5& \varepsilon^4-\varepsilon^3& \varepsilon-\varepsilon^6\\
                             \varepsilon^4-\varepsilon^3& \varepsilon-\varepsilon^6& \varepsilon^2-\varepsilon^5\end{array}\right],
\end{eqnarray*}
where $\sqrt{-7}=-\varepsilon-\varepsilon^2+\varepsilon^3-\varepsilon^4+\varepsilon^5+\varepsilon^6$. We can show that
\begin{eqnarray*}
&&\ord(A_3)=7,\ \ord(B_3)=3,\ \ord(C_3)=2,\ {B_3}^{-1}A_3B_3={A_3}^2,\\
&&{C_3}^{-1}BC_3={B_3}^{-1},\ C_3A_3C_3={A_3}^{-1}C_3{A_3}^{-1}.
\end{eqnarray*}
So there exists a faithful representations $\Psi$ of $G$ in $GL_3(k)$ such that $\Psi((a))=A_3$, $\Psi((b))=B_3$, and $\Psi((c))=C_3$.
This representation  gives rise to the faithful representation $\Psi_1$ of $G$ in $GL_4(k)$ such that $\Psi_1(g)=\Psi(g)\oplus 1$ for any $g\in G$.
Let $\psi=\pi_3\circ \Psi$ and $\psi_1=\pi_4\circ \Psi_1$. Clearly they are faithful representations of $G$, and $\psi_1=\varphi_{0,1}$.

%% proposition 4.8
\begin{proposition}
The nonsingular eigenspace of $\varphi_{0,1}(PSL_2(\F_7))$ in $\Form_{4,4}$ is $\langle x^3y+y^3z+z^3x,t^4\rangle$. Any quartic form
$a(x^3y+y^3z+z^3x)+bt^4$ with $ab\not=0$ is projectively equivalent to $x^3y+y^3z+z^3x+t^4$.
\end{proposition}
\begin{proof}
Let $A=\diag[\varepsilon^4,\varepsilon^2,\varepsilon,1]$, $B=[e_3,e_1,e_2,e_4]$ and $C=\delta C_{0}$ whose first row takes the form $[\alpha,\beta,\gamma,0]$
with $\alpha=\varepsilon-\varepsilon^6$, $\beta=\varepsilon^2-\varepsilon^5$, $\gamma=\varepsilon^4-\varepsilon^3$ and $\delta=-(\alpha+\beta+\gamma)$.
Assume that a nonsingular quartic form $f(x,y,z,t)$ satisfies 1) $f_{A^{-1}}\sim f$, 2) $f_{B^{-1}}\sim f$ and 3) $f_{C^{-1}}\sim f$.
Any $h\in \Form_{4,4}(A;\varepsilon^j)$ ($j\in [0,6]$) is singular unless $j=0$. So $f\in \langle x^3y,y^3z,z^3x,t^4,xyzt\rangle$.
By the condition 2) the nonsingular $f$ takes the form $a(x^3y+y^3z+z^3x)+bt^4+cxyzt$ ($ab\not=0$), where $c=0$ by the condition 3).
Let $x_1,x_2,x_3,a_1,a_2,a_3$ be indeterminates, $F(x,y,z)=x^3y+y^3z+z^3x$, $y_i=\sum_{j=1}^3a_{\sigma^{i-1}(j)}x_j$ where $\sigma=(123)\in \SSS{3}$, and
$F(y_1,y_2,y_3)=G(x_1,x_2,x_3,a_1,a_2,a_3)=\sum_{i_1+i_2+i_3=4}g_{i_1i_2i_3}(a_1,a_2,a_3)x_1^{i_1}x_2^{i_2}x_3^{i_3}$. Then
$g_{[i_1,i_2,i_3]\sigma}=g_{[i_1,i_2,i_3]}$, for $(\sigma G)(x_1,x_2,x_3)=\sum_{i_1+i_2+i_3=4}g_{[i_1,i_2,i_3]\sigma}x_1^{i_1}x_2^{i_2}x_3^{i_3}$ and
\begin{eqnarray*}
(\sigma G)(x_1,x_2,x_3)=(\sigma^2 F)(y_1,y_2,y_3)=F(y_1,y_2,y_3)=G(x_1,x_2,x_3).
\end{eqnarray*}
Clearly $F_{C^{-1}}(x,y,z)=G(x,y,z,\alpha,\beta,\gamma)$.
Setting $a=a_1$, $b=a_2$ and $c=a_3$, we have
\begin{eqnarray*}
&&g_{400}=a^3b+b^3c+c^a,\ g_{310}=ab^3+bc^3+ca^3+3(a^2b^2+b^2c^2+c^2a^2),\\
&&g_{301}=a^4+b^4+c^4+3(a^2bc+ab^2c+abc^2),\\
&&g_{220}=3(ab^3+bc^3+ca^3+a^2bc+ab^2c+abc^2),\\
&&g_{211}=3(a^3b+b^3c+c^a+a^2b^2+b^2c^2+c^2a^2)+6(a^2bc+ab^2c+abc^2).
\end{eqnarray*}
Put $p=\varepsilon+\varepsilon^6$, $q=\varepsilon^2+\varepsilon^5$, $r=\varepsilon^3+\varepsilon^5$. Then
\begin{eqnarray*}
&&p^2=2+q,\ q^2=2+r,\ r^2=2+p,\ pq=-1-q,\ qr=-1-r,\ rp=-1-p,\\
&&\alpha^2=-2+q,\ \beta^2=-2+r,\ \gamma^2=-2+p,\ \alpha\beta=r-p,\ \beta\gamma=p-q,\ \gamma\alpha=q-r.
\end{eqnarray*}
Using these equalities, we obtain
\begin{eqnarray*}
&&\alpha^4+\beta^4+\gamma^4=21,\ \alpha^3\beta+\beta^3\gamma+\gamma^3\alpha=0,\ \alpha\beta^3+\beta\gamma^3+\gamma\alpha^3=7,\\
&&\alpha^2\beta^2+\beta^2\gamma^2+\gamma^2\alpha^2=14,\ \alpha^2\beta\gamma+\alpha\beta^2\gamma+\alpha\beta\gamma^2=-7.
\end{eqnarray*}
Hence, we can evaluate $g'_{i_1i_2i_3}=g_{i_1i_2i_3}(\alpha,\beta,\gamma)$ as follows.
\begin{eqnarray*}
&&g'_{400}=0,\ g'_{310}=49,\ g'_{301}=0,\ g'_{220}=0,\ g'_{211}=0.
\end{eqnarray*}
Consequently $f_{C^{-1}}=49f$, for $\delta^4=49$.\\
Obviously $f_{T^{-1}}=x^3y+y^3z+z^3x+t^4$ for $T=\diag[a^{-1/4},a^{-1/4},a^{-1/4},b^{-1/4}]$.
\end{proof}

%% proposition 4.9
\begin{proposition}{\rm \cite[p.200]{edg}}
The nonsingular eigenspace of $\varphi_{\sqrt{2},1}(PSL_2(\F_7))$ in $\Form_{4,4}$ is $\langle 2(x^3y+y^3z+z^3x)+t^4+(6\sqrt{2}) xyzt\rangle$.
\end{proposition}
\begin{proof}
Let $G=\varphi_{\sqrt{2},1}(PSL_2(\F_7))$, $A=\diag[\varepsilon^4,\varepsilon^2,\varepsilon,1]$, $B=[e_3,e_1,e_2,e_4]$, and
$C=Z(\sqrt{2},\alpha,\beta,\gamma)$, where
$\alpha=\varepsilon+\varepsilon^6$, $\beta=\varepsilon^2+\varepsilon^5$ and $\gamma=\varepsilon^3+\varepsilon^4$. Then $G=\langle (A),(B),(C)\rangle$.
Let $f\in \Form_{4,4}$ be nonsingular and $G$-invariant. Since $f\in \Form_{4,4}(A;1)$ and $f_{B^{-1}}\sim f$, $f$ takes the form
$a(x^3y+y^3z+z^3x) +bt^4+cxyzt$ ($a,b,\in k^*,\ c\in k$). Let $\sigma=(123)\in \SSS{4}$, $x_1,...,x_4$, $a_1,a_2,a_3$ be indeterminates,
$y_i=\sum_{j=1}^3a_{\sigma^{i-1}(j)}x_j+\tau x_4$ ($i\in [1,3],\ \tau=\sqrt{2}$), $y_4=\tau(\sum_{j=1}^3x_j)+x_4$, and define a polynomial
$$G(x_1,...,x_4,a_1,a_2,a_3)=\sum_{i_1+...+i_4=4}g_{i_1...i_4}(a_1,a_2,a_3)x_1^{i_1}\cdots x_4^{i_4}$$ to be $f(y_1,...,y_4)$. Then,
$\sigma^{-1} G=G$ as a polynomial in $x_1,x_2,x_3,x_4$, hence $g_{[i_1,i_2,i_3,i_4]\sigma}=g_{[i_1,i_2,i_3,i_4]}$,
for $\sigma^{-1} G=(\sigma^2f)(y_1,y_2,y_3,y_4)=f(y_1,y_2,y_3,y_4)$. Note that $f_{C^{-1}}(x,y,z,t)=G(x,y,z,t,\alpha,\beta,\gamma)$.
In addition to equalities $\alpha+\beta+\gamma=-1$ and $\alpha\beta\gamma=1$, we have the following equalities.
\begin{eqnarray*}
&&\left[\begin{array}{c}\alpha^2\\ \beta^2\\ \gamma^2\end{array}\right]
=\left[\begin{array}{c}2+\beta\\ 2+\gamma\\ 2+\alpha\end{array}\right],\
\left[\begin{array}{c} \alpha\beta\\ \beta\gamma\\ \gamma\alpha\end{array}\right]
=\left[\begin{array}{c}-1-\beta\\ -1-\gamma\\ -1-\alpha\end{array}\right],\\
&&\left[\begin{array}{c}\alpha^3\\ \beta^3\\ \gamma^3\end{array}\right]
=\left[\begin{array}{c}-1+2\alpha-\beta\\ -1+2\beta-\gamma\\ -1+2\gamma-\alpha\end{array}\right],\
\left[\begin{array}{c}\alpha^2\beta\\ \beta^2\gamma\\ \gamma^2\alpha\end{array}\right]
=\left[\begin{array}{c}2+2\beta+\gamma\\ 2+2\gamma+\alpha\\ 2+2\alpha+\beta\end{array}\right],\
\left[\begin{array}{c}\alpha\beta^2\\ \beta\gamma^2\\ \gamma\alpha^2\end{array}\right]
=\left[\begin{array}{c}-1+\alpha\\ -1+\beta\\ -1+\gamma\end{array}\right],\\
&&\left[\begin{array}{c}\alpha^4\\ \beta^4\\ \gamma^4\end{array}\right]
=\left[\begin{array}{c}6+4\beta+\gamma\\ 6+4\gamma+\alpha\\ 6+4\alpha+\beta\end{array}\right],\
\left[\begin{array}{c}\alpha^2\beta^2\\ \beta^2\gamma^2\\ \gamma^2\alpha^2\end{array}\right]
=\left[\begin{array}{c} 3+2\beta+\gamma\\ 3+2\gamma+\alpha\\ 3+2\alpha+\beta\end{array}\right],\\
&&\left[\begin{array}{c}\alpha^3\beta\\ \beta^3\gamma\\ \gamma^3\alpha\end{array}\right]
=\left[\begin{array}{c}-4-3\beta-\gamma\\ -4-3\gamma-\alpha\\ -4-3\alpha-\beta\end{array}\right],\
\left[\begin{array}{c}\alpha\beta^3\\ \beta\gamma^3\\ \gamma\alpha^3\end{array}\right]
=\left[\begin{array}{c}-1-2\beta\\ -1-2\gamma\\ -1-2\alpha\end{array}\right].
\end{eqnarray*}
Thus $f_{C^{-1}}(x,y,z,t)$ takes the form
\begin{eqnarray*}
&&t^4(12a+q+2\tau c)+t^3(x+y+z)(-8\tau a+4\tau b+2c)\\
&&+t^2\left[(x^2+y^2+z^2)(18a+12b-4\tau c)+(xy+yz+zx)(-6a+24b-\tau c)\right]\\
&&+t\left[5(x^3+y^3+z^3)-27(x^2y+y^2z+z^2x)-6(xy^2+yz^2+zx^2)+72xyz\right]\tau a\\
&&+t\left[x^3+y^3+z^3+3(x^2y+y^2z+z^2x)+3(xy^2+yz^2+zx^2)+6xyz\right]\tau b\\
&&+t\left[-3(x^3+y^3+z^3)+5(x^2y+y^2z+z^2x)-2(xy^2+yz^2+zx^2)+17xyz\right]c\\
&&+\left[-8s_{400}+18s_{310}+10s_{301}-6s_{220}-12s_{211}\right]a\\
&&+\left[s_{400}+4s_{310}+4s_{301}+6s_{220}+12s_{211}\right]\tau^2 b\\
&&+\left[s_{400}+4s_{310}-3s_{301}-s_{220}-2s_{211}\right]\tau c,
\end{eqnarray*}
where $s_{400}=x^4+y^4+z^4$, $s_{220}=x^2y^2+y^2z^2+z^2x^2$, and
\[
 s_{310}=x^3y+y^3z+z^3x,\ s_{301}=xy^3+yz^3+zx^3,\ s_{211}=x^2yz+xy^2z+xyz^2.
\]
Since $f_{C^{-1}}\sim f$, the coefficients of monomials $t^2x^2$ and $t^2xy$ in $f_{C^{-1}}$ vanish, namely $18a-4\tau c=-12b$ and $-6a-\tau c=-24c$, i.e.,
$a=2b,\ \tau c=12b$. One sees easily that $f_{C^{-1}}=49b f$ if $a=2b,\ \tau c=12b$.
\end{proof}

%% lemma 4.10
\begin{lemma}{\rm \cite[\S 272, Ex.5]{bur}}
$\Paut(x^5z+y^5x+z^5y)$ is a group of order $63$ generated by $(\diag[\delta,\delta^5,1])$ and $([e_3,e_1,e_2])$, where $\delta\in k^*$ with
$\ord(\delta)=21$, and $[e_1,e_2,e_3]$ is the unit matrix in $GL_3(k)$.
\end{lemma}
\begin{proof}
Let $f=x^5z+y^5x+z^5y$, $\Hess(f)=250h$, where $h=33(xyz)^4-2(x^9y^3+y^9z^3+z^9x^3)$, and $(x,y,z)\in \PP{2}$ a singular point of $V_p(h)$.
If $xyz=0$, then $(x,y,z)$ is $(1,0,0)$, $(0,1,0)$ or $(0,0,1)$. We claim that $xyz\not=0$ is impossible. Indeed, $xyz\not=0$ and $x^3h_x=y^3h_y=z^3h_z=0$
imply $x^9y^3=y^9z^3=z^9x^3=11(xyz)^4/2$, hence $(xyz)^{12}=(11/2)^3(xyz)^{12}$, a contradiction. Let $G=\Paut(f)$ and $H=\Paut(h)$. As is well known,
$G$ is a subgroup of $H$. Clearly $(A), (B)\in G$, where $A=\diag[\delta,\delta^5,1]$ and $B=[e_3,e_1,e_2]$. Since $G$ acts transitively on the set of all
singular points of $V_p(h)$, so does $H$. Let $H_3$ be the isotropy subgroup of $H$ at $(0,0,1)$, and assume $(C)\in H_3$ with $C=[c_{ij}]\in GL_3(k)$.
We may assume that the third column of $C$ is equal to $e_3$. Since $h_{C^{-1}}\sim h$, it follows that $c_{31}=c_{32}=0$, hence $C=\diag[\beta,\gamma,1]$
such that $\gamma^{21}=1$ and $\beta=\gamma^{-4}$. Consequently $H_3=\langle(A)\rangle$ and $|H|=63$ so that $H=\langle(A),(B)\rangle= G$.
\end{proof}

%% proposition 4.11
\begin{proposition} Let $f=x^3y+y^3z+z^3x$ and $h=x^5z+y^5x+z^5y-5x^2y^2z^2$.  \\
$(1)$ $\Paut(f)=\Paut(h)=\psi(PSL_2(\F_7))$, and the nonsingular eigenspace of $\psi(PSL_2(\F_7))$ in $\Form_{3,6}$ is
$\langle h\rangle$.\\
$(2)$ Let $\delta_n\in k^*$ with $\ord(\delta_n)=n$ for positive integer $n$. Then $\Laut(f)=\langle \delta_4\rangle \Psi(PSL_2(\F_7))$ and
$\Laut(h)=\langle \delta_6\rangle \Psi(PSL_2(\F_7))$
\end{proposition}
\begin{proof}
Let $h_0=f$, $h_1=-\Hess(h_0)/57$, $h_2=\Hess(h_1)/250$, where $h_1=h$  and
\begin{eqnarray*}
h_2&=&x^{10}z^2+y^{10}x^2+z^{10}y-2(x^9y^3+y^9z^3+z^9x^3)-4(x^6y^5z+y^6z^5x+z^6x^5y)\\
&&-16(x^7y^2z^3+y^7z^2x^3+z^7x^2z^3)+13x^4y^4z^4.
\end{eqnarray*}
Denote the groups $\psi(PSL_2(\F_7))$, $\Paut(h_0)$, $\Paut(h_1)$ and $\Paut(h_2)$ by $G_{-1}$, $G_0$, $G_1$ and $G_2$, respectively. As is well known,
$G_0\subset G_1\subset G_2 $. Recall that $A_3=\Psi((a))$, $B_3=\Psi((b))$ and $C_3=\Psi((c))$.
It is evident that ${h_0}_{A_3^{-1}}={h_0}_{B_3^{-1}}=h_0$.
Let $[\alpha,\beta,\gamma]$ be the first row of $\sqrt{-7}C_3$. Then $49f_{C_3^{-1}}$ takes the form
\begin{eqnarray*}
&&(x^3y+y^3z+z^3x)\{\alpha\beta^3+\beta\gamma^3+\gamma\alpha^3+3(\alpha^2\beta^2+\beta^2\gamma^2+\gamma^2\alpha^2)\}\\
&&+(x^4+y^4+z^4)(\alpha^3\beta+\beta^3\gamma+\gamma^3\alpha)\\
&&+(x^3z+y^3x+z^3y)\{\alpha^4+\beta^4+\gamma^4+3(\alpha^2\beta\gamma+\beta^2\gamma\alpha+\gamma^2\alpha\beta)\}\\
&&+3(x^2y^2+y^2z^2+z^2x^2)\{(\alpha\beta^3+\beta\gamma^3+\gamma\alpha^3)+(\alpha^2\beta\gamma+\beta^2\gamma\alpha+\gamma^2\alpha\beta)\}\\
&&+3(x^2yz+y^2zx+z^2xy)c_{211},
\end{eqnarray*}
where $c_{211}=(\alpha^3\beta+\beta^3\gamma+\gamma^3\alpha)+(\alpha^2\beta^2+\beta^2\gamma^2+\gamma^2\alpha^2)
+2(\alpha^2\beta\gamma+\beta^2\gamma\alpha+\gamma^2\alpha\beta)$.
Let $p=\varepsilon+\varepsilon^6$, $q=\varepsilon^2+\varepsilon^5$, and $r=\varepsilon^3+\varepsilon^4$. Then, using the following equalities
\begin{eqnarray*}
&&\alpha^2=-2+q,\ \beta^2=-2+r,\ \gamma^2=-2+p,\ \alpha\beta=r-p,\ \beta\gamma=p-q,\ \gamma\alpha=q-r,
\end{eqnarray*}
it is not difficult to see ${h_0}_{C_3^{-1}}=h_0$. Thus $\Psi(PSL_2(\F_7))\subset \Laut(h_0)$, and $G_{-1}\subset G_0$. Since $\Psi(PSL_2(\F_7))\subset SL_3(k)$, it follows
that $\Psi(PSL_2(\F_7))\subset \Laut(h_1)$.
Let $F\in \Form_{3,d}\backslash \{0\}$ $(d>2)$.
As is known, $H=\Hess(F)=0$ if and only if $F_{T^{-1}}\in \Form_{2,d}$ for some $T\in GL_3(k)$. Assume $H\not=0$. We write $I(Q,F\cap H)$ for
the intersection number of $V_p(F)$ and $V_p(H)$ at $Q\in \PP{2}$ \cite[p.74,p.104]{ful0}. Then  $Q\in V_p(F)\cap V_p(H)$ if and only if $Q$ is
a singular point or a flex of $V_p(F)$, and  $I(Q,F\cap H)=1$ if and only if $Q$ is an ordinary flex \cite[p.116]{ful0}. Both $h_0$ and $h_1$ are
nonsingular. For instance, assume $h_{1,x}$, $h_{1,y}$, and $h_{1,z}$ vanish at $[x,y,z]\in k^3$. Then the condition $xh_{1,x}=yh_{1,y}=zh_{1,z}=0$ implies
$x^5z=y^5x=z^5y=5x^2y^2z^2$, so that $xyz=0$, hence it follows that $x=y=z=0$. Let $Q_1=(1,0,0)$, $Q_2=(0,1,0)$, and $Q_3=(0,0,1)$. Clearly
$Q_3\in V_p(h_0)\cap V_p(h_1)\cap V_p(h_2)$. Let $P=[0,0]\in k^2$. Denote the order functions at $P$ of the affine curves $h_0(x,y,1)=0$ and $h_1(x,y,1)$ by
$\ord_P^{h_0}$ and $\ord_P^{h_1}$, respectively. Then  $\ord_P^{h_0}(y)=1$ so that $\ord_P^{h_0}(x)=3$ and $\ord_P^{h_0}(h_1(x,y,1))=1=I(Q_3,f\cap h_1)$.
Hence $Q_3$ is an ordinary flex of $V_p(h_0)$. $G_{-1}Q_3\subset V_p(h_0)\cap V_p(h_1)$, and $\sum_{Q\in V_p(h_0)\cap V_p(h_1)}I(Q,h_0\cap h_1)=24$ by Bezout's
theorem. Since $|G_{-1}Q_3|\geq 3+7\cdot 3$, which will be shown shortly, it follows that $|G_{-1}Q_3|=24$ and $G_{-1}Q_3=V_p(h_0)\cap V_p(h_1)$.
In order to show that $|G_1Q_3|\geq 24$ let $b_{ij}=\sqrt{-7}A_3^iC_3B_3^{3-j}[0,0,1]$ ($[i,j]\in [0,6]\times [1,3]$), i.e., $b_{i1}=A_3^i[\alpha,\beta,\gamma]$,
$b_{i2}=A^i[\beta,\gamma,\alpha]$ and $b_{i3}=A^i[\gamma,\alpha,\beta]$. It is clear that $b_{ij}\not\sim b_{i'j}$ if $i\not=i'$. Since
$\gamma^{-1}=-\varepsilon^4(6+5\varepsilon+4\varepsilon^2+3\varepsilon^3+2\varepsilon^4+\varepsilon^5)$, it follows that
$\alpha\beta\gamma^{-2}=\varepsilon(-\varepsilon-\varepsilon^2+\varepsilon^4+\varepsilon^5)\not\in \langle \varepsilon\rangle$, hence
$b_{ij}\not\sim b_{i'j'}$ if $j\not=j'$. Thus the set $\{(A^iC_3B_3^{3-j})Q_3:[i,j]\in [0,6]\times [1,3]\}\subset G_1Q_3$ consists of $21$ points.
Furthermore $\ord_P^{h_1}(x)=1$ so that $\ord_P^{h_1}(y)=5>3$
and $\ord_P^{h_1}(h_2(x,y,1))=3=I(Q_3,h_1\cap h_2)$. Thus $Q_3$ is a higher flex of $V_p(h_1)$.  $G_1Q_3\subset V_p(h_1)\cap V_p(h_2)$, and
$\sum_{Q\in V_p(h_1)\cap V_p(h_2)}I(Q,h_1\cap h_2)=72$ by Bezout's theorem. So $G_{-1}Q_3=V_p(h_1)\cap V_p(h_2)$. Consequently $G_1$ acts transitively on
$V_p(h_1)\cap V_p(h_2)$. Let $(A)\in G_{1,Q_3}$, the isotropy subgroup of $G_1$ at $Q_3$, where $A=[a_{ij}]\in GL_3(k)$. We may assume the third column
of $A$ is $e_3$. Since $h_{1,A^{-1}}\sim h_1$, it follows that $a_{31}=a_{32}=0$, hence $A=\diag[\alpha,\beta,1]$ with $\alpha^7=1$ and $\beta=\alpha^5$.
Thus  $|G_{1,Q_3}|=7$, $|G_{1}|=7\cdot 24$, i.e., $G_1=G_{-1}$.

 Suppose $\Form_{3,6,nons}^{G_{-1}}\not=\emptyset$, and $g\in \Form_{3,6,nons}^{G_{-1}}$. $\Form_{3,6}(A_3;\varepsilon^i)$ ($i\in [0,6]$) is singular
at $Q_j\in \PP{2}(k)$ for some $j\in [1,3]$ unless $i=0$, hence  $g\in \Form_{3,6}(A_3;1)=\langle x^5z,y^5x,z^5y,x^2y^2z^2\rangle$ so that
$g=ax^5z+by^5x+cz^5y+dx^2y^2z^2$ with $abcd\not=0$ by Lemma 4.10. Indeed, if $d=0$, then $g_{T^{-1}}=x^5z+y^5x+z^5y$ so that $|\Paut(g)|=63$
by Lemma 4.10, a contradiction. Since $g_{B_3^{-1}}\sim g$, we have $a=b=c$. Since $g_{C_3^{-1}}\sim g$, the coefficient of $x^6$ in $g_{C_3^{-1}}$
must vanish, i.e., $d=-5a$. Thus the nonsingular eigenspace of $G_{-1}$ in $\Form_{3,6}$ is $\langle h_1\rangle$.

The group homomorphism $\pi_3:\Laut(h_0)\rightarrow \Paut(h_0)=G_{-1}$ is surjective. Denote this homomorphism by $\rho$ and let
$N=\Ker{\rho}$. Then $N=\Ker{\pi_3}\cap \Laut(h_0)=\langle \delta_4\rangle$, hence $\Laut(h_0)=\cup_{A\in \Psi(PSL_2(\F_7))}\langle \delta_4\rangle A$,
\end{proof}
Now we can determine the projective automorphism group of $f^0=x^3y+y^3z+z^3x+t^4$.
%% proposition 4.12
\begin{proposition} Let $\Psi_1:PSL_2(\F_7)\rightarrow GL_4(k)$ be the faithful group representation such that $\Psi_1(u)=\Psi(u)\oplus 1$ for
$u\in PSL_2(\F_7)$,  $D=\diag[\sqrt{-1},\sqrt{-1},\sqrt{-1},1]$,  $f=x^3y+y^3z+z^3x$ and $g=x^3y+y^3z+z^3x+t^4$.
Then $\Paut(g)=\{(AD^j):A\in \Phi(PSL_2(\F_7)),j\in [0,3]\}$, which coincides with $\langle \varphi_{0,1}(PSL_2(\F_7)), (D)\rangle$.
\end{proposition}
\begin{proof}
Let $f=x^3y+y^3z+z^3x$. Then $\Hess(g)=-2^3 3^4h(x,y,z)t^2$, where $h=xy^5+yz^5+zx^5-5x^2y^2z^2$. Since the projective algebraic set $V_p(h)$
in $\PP{2}(k)$ is nonsingular, $h$ is irreducible. Assume $(S)\in \Paut(g)$ with $S=[s_{ij}]\in GL_4(k)$. Then $g_{S^{-1}}\sim g$ so that
$\Hess(g)_{S^{-1}}\sim \Hess(g)$. Since the polynomial ring $k[x,y,z,t]$ is UFD, $t_{S^{-1}}\sim t$ and $h_{S^{-1}}\sim h$. The first condition
yields $s_{4j}=0$ $(j\in [1,3])$, while the second condition implies
\[
 t(s_{14}Y^5+5s_{24}XY^4+s_{24}Z^5+5s_{34}YZ^4+s_{34}X^5+5s_{14}ZX^4)=0,
\]
where $X=s_{11}x+s_{12}y+s_{13}z$, $Y=s_{21}x+s_{22}y+s_{23}z$, and $Z=s_{31}x+s_{32}y+s_{33}z$. Since $X,Y,Z$ are algebraically
independent over $k$, it follows that $s_{i4}=0$ $(i\in [1,3])$,
 hence we may assume
$S=T\oplus s_{44}$ with $T\in GL_3(k)$. Let $T'=T/s_{44}$ and $S'=T'\oplus 1$ so that $(S')=(S)$. Clearly $(S')\in \Paut(g)$ if and only if $f_{{T'}^{-1}}=f$.
The group $H=\{T'\in GL_3(k):\ f_{{T'}^{-1}}=f\}$ is nothing but $\Laut(f)$ in Proposition 4.11. Moreover
the map $\varphi:H\rightarrow \Paut(g)$ assigning $T'$ to $(T'\oplus 1)\in PGL_4(k)$ is a group isomorphism.
\end{proof}

\section{$\ZZZ{p^a}$-invariant forms in $\Form_{4,d}$}

In this section we shall prove
%% lemma 5.1
\begin{theorem}
Let $a$ be a positive integer, $p$ a prime,  $q=p^a$, $\varepsilon\in k^*$ with $\ord(\varepsilon)=q$, $d\geq 3$ an integer, $f\in \Form_{4,d}$,
$D_0=\diag[1,1,\varepsilon,1]$, $D_{j}=\diag[1,1,\varepsilon,\varepsilon^j]$ ($0<j<q$) and
$D_{j\ell}=\diag[1,\varepsilon,\varepsilon^j,\varepsilon^\ell]$ $(j,\ \ell\in [2,\ q-1]$ with $j\not=\ell)$. \\
$(1)$ Suppose $\Paut(f)\supset \langle (D_0)\rangle$.  If $q>d$, then $f$ is singular.\\
$(2)$ Suppose $\Paut(f)\supset \langle (D_j)\rangle$.  If either $(2.1)$ $q>d(d-1)$ or $(2.2)$ $[q,d]=[11,4]$, then $f$ is singular.\\
$(3)$ Suppose $\Paut(f)\supset \langle (D_{j\ell})\rangle$. If either $(3.1)$  $q>d(d-1)^2$ or $[q,d]=[5^2,4]$, then $f$ is singular.\\
$(4)$ Any $\ZZZ{q}$-invariant $d$-forms with $q>d(d-1)^2$, any $\ZZZ{p}$-invariant quartic forms with $p>7$ and any $\ZZZ{p^2}$-invariant quartic
forms with $p>3$ are singular.
\end{theorem}
\begin{proof}
For a diagonal matrix $D\in GL_4(k)$ and monomials $M_{ij}=x_ix_j^{d-1}\in \Form_{4,d}$ we have
${M_{ij}}_{D^{-1}}=\varepsilon^{\ell_{ij}}$ ($\ell_{ij}\in \Z/q\Z$). Let $I_j=I_j(D)=\{\ell_{ij}:i\in [1,4]\}$, and $I=I(D)=\cap_{j=1}^4 I_j(D)$.
Then, if $I=\emptyset$, any $f\in \Form_{4,d}^{\{D\}}\backslash\{0\}$ is singular by Lemma 2.11.\\
(1) Let $I_i=I_i(D_0)$ and $I=\cap_{i=1}^4 I_i$. Then $I_1=I_2=I_4=\{0,1\}$ and $I_3=\{d-1,d\}$ so that $I=\emptyset$. \\
(2) Let $I_i=I_i(D_j)$ ($j\in [1,q-1]$) and $I=\cap_{i=1}^4 I_i$. Then $I_1=I_2=\{0,1,j\}$, $I_3=\{d-1,d,d-1+j\}$ and $I_4=\{(d-1)j,(d-1)j+1,dj\}$.
If $j=1$, then $I_3=I_4$, hence $I=\emptyset$. We shall give a proof in the case (2.1), for the case (2.2) can be dealt with exactly in the same way.
Assume $j>1$, $I\not=\emptyset$ and $i\in I$. Since $I\subset I_1$, $i\in \{0,1,j\}$.
Suppose $i=0$. The condition $i\in I_3$ implies $0\equiv d-1+j$, hence $I_4=\{-d(d-2),\ -(d-1)^2,\ -d(d-1)\}\not\ni 0$, a contradiction.
Suppose $i=1$. The condition $i\in I_3$ implies $1\equiv d-1+j$, hence $I_4=\{-(d-1)(d-2)+1,\ -(d-1)(d-2),\ -d(d-1)\}\not\ni 1$, a contradiction.
Finally, suppose $i=j$. The condition $i\in I_3$ implies $j\equiv d-1$ or $j\equiv d$. In the former case $I_4=\{(d-1)^2,\ (d-1)^2+1,\ d(d-1)\}\not\ni d-1$,
 and in the second case $I_4=\{d(d-1),\ d(d-1)+1,\ d^2\}\not\ni d$, a contradiction. \\
(3) Let $I_i=I_i(D_{j\ell})$ ($j,\ell\in [2,q-1]$ with $j\not=\ell$) and $I=\cap_{i=1}^4 I_i$. So $I_1=\{0,1,j,\ell\}$, $I_2=\{d-1,d,d-1+j,d-1+\ell\}$,
$I_3=I_3(j,\ell)=\{(d-1)j,(d-1)j+1,(d-1)j+\ell,dj\}$, and $I_4=I_4(j,\ell)=\{(d-1)\ell,(d-1)\ell+1,(d-1)\ell+j,d\ell\}$.
Note that $I_3(\ell,j)=I_4(j,\ell)$. We shall give a proof in the case (3.1), for the case (3.2) can
be dealt with exactly in the same way. The assumption $I\not=\emptyset$ leads us to a contradiction, as follows. Let $i\in I$. Then $i\in I_1$,
for $I\subset I_1$. Assume first that $i\equiv 0\in I_1$. Since $0\not\equiv d,d-1$, either $i\equiv d-1+j$ or $i\equiv d-1+\ell$. If $i\equiv d-1+j$, then
$I_3=\{-(d-1)^2,-(d-1)^2+1,-d(d-1),-(d-1)^2+\ell\}\ni 0$, hence $-(d-1)^2+\ell\equiv 0$ so that $I_4=\{(d-1)^3,(d-1)^3+1,d(d-1)(d-2)\}\not\ni 0$,
a contradiction. If $i\equiv d-1+\ell$, the condition $0\in I_4$ leads to $0\not\in I_3$, a contradiction. Assume secondly that $i\equiv 1\in I_1$.
Since $1\not\equiv d,d-1$, either $i\equiv d-1+j$ or $i\equiv d-1+\ell$. It suffices to consider the case $1\equiv d-1+j$. Since
$1\in I_3=\{-(d-1)(d-2),-(d-1)(d-2)+1,-d(d-2),-(d-1)(d-2)+\ell\}$, $\ell\equiv d^2-3d+3=u$ so that $I_4=\{(d-1)u,(d-1)u+1,(d-1)u-(d-2),du\}\not\ni 1$,
a contradiction. Assume thirdly that $i=j\in I_1$. Since $j\in I_2$ and $j\not\equiv d-1+j$, $j\in \{d-1,d,d-1+\ell\}$. If $j\equiv d-1$, then
the condition $j\in I_3$ implies $j\equiv (d-1)^2+\ell$, hence $I_4=\{-(d-1)^2(d-2), -(d-1)^2(d-2)+1,-(d-1)^2(d-2)+d-1,-d(d-1)(d-2)\}\not\ni d-1=j=i$,
a contradiction. If $j\equiv d$, the condition $j\in I_3$ yields $d(d-1)+\ell\equiv j$, hence
$I_4=\{-d(d-1)(d-2),-d(d-1)(d-2)+1,-d(d-1)(d-2)+d,-d^2(d-2)\}\not\ni d=j=i$, a contradiction. Suppose $j\equiv d-1+\ell$, i.e., $\ell\equiv j-d+1$.
Then
\begin{eqnarray*}
I_3&=&\{(d-1)j,(d-1)j+1,dj,dj-d+1\},\\
I_4&=&\{(d-1)j-(d-1)^2,(d-1)j-d(d-2),dj-(d-1)^2,dj-d(d-1)\}.
\end{eqnarray*}
However, we can show $j\not\in I_4$ as follows. Suppose first that  $j\equiv (d-1)j\in I_3$, i.e., $(d-2)j\equiv 0$.
 Then $(d-1)j\not\equiv (d-1)j-(d-1)^2,(d-1)j-d(d-2)\in I_4$.
 If $dj-(d-1)^2-j\equiv 0$, then multiplication by $d-2$ gives $-(d-1)^2(d-2)\equiv 0$, a contradiction. If $dj-d(d-1)-j\equiv 0$, then $d(d-1)(d-2)\equiv 0$,
a contradiction. Suppose secondly $j\equiv (d-1)j+1\in I_3$, i.e., $(d-2)j\equiv -1$. Then $(d-1)j+1\not\equiv (d-1)j-(d-1)^2,(d-1)j-d(d-2)\in I_4$.
If $ dj-(d-1)^2-j\equiv 0$, multiplication by $d-2$ gives $(d-1)u\equiv 0$, a contradiction. If $dj-d(d-1)-j\equiv 0$, then $(d-1)^3\equiv 0$, a contradiction.
Suppose thirdly $j\equiv dj$, i.e., $(d-1)j\equiv 0$. Then $j\equiv dj\not\equiv dj-(d-1)^2,dj-d(d-1)\in I_4$. If $(d-1)j-(d-1)^2-j\equiv 0$, then
multiplication by $d-1$ gives $(d-1)^3\equiv 0$, a contradiction. Similarly, if $(d-1)j-d(d-2)-j\equiv 0$, then $d(d-1)(d-2)\equiv 0$, a contradiction.
Suppose fourthly $j\equiv dj-d+1\in I_3$, i.e., $(d-1)j\equiv d-1$. Clearly $j\equiv dj-d+1\not\equiv dj-(d-1)^2,dj-d(d-1)\in I_4$. If
$(d-1)j-(d-1)^2-j\equiv 0$, then multiplication by $d-1$ yields $(d-1)u\equiv 0$, a contradiction. Similarly, if $(d-1)j-d(d-2)-j\equiv 0$, then
$(d-1)^2(d-2)\equiv 0$, a contradiction. Finally  assume $i=\ell\in I_1$. In this case we can proceed exactly as in the case $i=j\in I_1$. \\
The last statement (4) follows from (1),(2),(3) and Lemma 2.13.
\end{proof}
%% Corollary 5.2
\begin{corollary}
Let $p\geq 11$ be a prime. Then any $\ZZZ{p}$-invariant quartic form in $k[x,y,z,t]$ is singular.
\end{corollary}
%%% Corollary 5.3
\begin{corollary}
Let $G$ be the projective automorphism group of a nonsingular quartic form in $k[x,y,z,t]$, and $\Pi p^{\nu(p)}$  the decomposition of the order
$|G|$ into prime factors. Then $\nu(p)=0$ if $p\geq 11$.
\end{corollary}

%%%%%  section 6
\section{$\Paut(x^3y+y^3z+z^3t+t^3x)$}
%%\section{The projective automorphism group of the form $x^3y+y^3z+z^3t+t^3x$}
 The goal of this section is the determination of the projective automorphism group of $f^{0,0,0}=x^3y+y^3z+z^3t+t^3x$. For the sake of simplicity we denote $f^{0,0,0}$ by $f$.
 Let $B=\diag[1,\beta,\beta^{-2},\beta^7]$, $B'=\diag[\beta,\beta^{-2} ,\beta^7,1]$ with $\ord(\beta)=20$, and $C=[e_4,e_1,e_2,e_3]$.
Then
$f_{B^{-1}}=\beta f$, $f_{C^{-1}}=f$, and $CBC^{-1}=\beta B^{17}$. In particular $G_{80}=\{(B)^i(C)^j\ :\  i\in [0,19],\ j\in [0,3]\}$ is a subgroup of
$PGL_4(k)$ of order $80$. An abelian subgroup $\langle (B^5)\rangle \times \langle(C)\rangle$ is a Sylow 2-subgroup of $G_{80}$.
%%%%%(1)
$G_{80}$ is isomorphic to $C_4 \times (C_5 \rtimes C_4)$. This group and the groups $G_i$ of the next section have been identified using Magma, a computational algebra system, considering the order of the elements of the group; see \cite{magma}.
If there are several groups of that order whose elements have the same orders, the number of couples of conjugate elements of the group has been considered.
Then the names of the groups have been determined using GAP; see \cite{GAP}.
%%%%%
Let $G_{20}=\langle(B)\rangle=\langle(B')\rangle$. We shall show

%%% theorem 6.1
\begin{theorem}
$\Paut(x^3y+y^3z+z^3t+t^3x)=G_{80}$.
\end{theorem}
\begin{proof}
By definition the projective automorphism group $\Paut(g)$ of a non-zero form $g(x_1,\dots,x_n)$ consists of $(A)\in PGL_n(k)$ such that
$g_{A^{-1}}\sim g$. Recall that $\Hess(g)=\det [g_{x_i,x_j}]$. As $\Hess(g_{A^{-1}})=(\det A)^2\Hess(g)_{A^{-1}}$, $g_{A^{-1}}\sim g$ implies
$\Hess(g)_{A^{-1}}\sim \Hess(g)$, hence $\Paut(g)$ is a subgroup of $\Paut(\Hess(g))$. Let $h=3^{-4}\Hess(f)$, where  $f=x^3y+y^3z+z^3t+t^3x$.
Therefore
\[
 h=x^4z^4+y^4t^4-4(x^5zt^2+x^2y^5t+xy^2z^5+yz^2t^5)+14x^2y^2z^2t^2.
\]
Hence
\begin{eqnarray*}
h_x&=& 4x^3z^4-4(5x^4zt^2+2xy^5t+y^2z^5)+28xy^2z^2t^2,\\
h_y&=& 4y^3t^4-4(5x^2y^4t+2xyz^5+z^2t^5)+28x^2yz^2t^2,\\
h_z&=& 4x^4z^3-4(x^5t^2+5xy^2z^4+2yzt^5)+28x^2y^2zt^2,\\
h_t&=& 4y^4t^3-4(2x^5zt+x^2y^5+5yz^2t^4)+28x^2y^2z^2t.
\end{eqnarray*}
Denote by ${\cal S}$ the set of all singular points of the projective algebraic set $V_p(h)$ in
$\PP{3}(k)$.\\ Clearly ${\cal S}_0=\{(1,0,0,0),(0,1,0,0),(0,0,1,0),(0,0,0,1)\}$
is a subset of ${\cal S}$. It is immediate that $(x,y,z,t)\in {\cal S}$ with $xyzt=0$ belongs to ${\cal S}_0$. We shall find all
$(x,y,z,1)\in {\cal S}\setminus {\cal S}_0$. Suppose $(x,y,z,1)$ with $xyz\not=0$ belongs to ${\cal S}$, namely $h_x=h_y=h_z=h_t=0$.
This condition is equivalent to $g_j=0$ ($j\in [1,4]$), where
\[
 g_1=(xh_x-h_t)/4,\ g_2=(yh_y-h_t)/4,\ g_3=(zh_z-h_t)/4,\ g_4=h_t/4.
\]
Since $10y(-x^2y^4+z^2)=(g_1-g_3+2g_2)$, we have $z^2=x^2y^4$, hence $z=\sigma xy^2$, where $\sigma^2=1$. Now $g_2=0$ yields $y^{10}=1$.
Now substituting $\sigma xy^2$ ($y^{10}=1$) for $z$ in $g_j$, we see that the condition $g_j=0$ ($j\in [1,4]$) is equivalent to
$y^{10}=1$, $z=\sigma xy^2$, $\ell_3=\ell_4=0$, where $\sigma^2=1$ and
\begin{eqnarray*}
&&\ell_3=x^8y^8-4x^6y^2\sigma+4x^2y^5-y^4,\ \ \ell_4=-2x^6y^2\sigma+7x^4y^6-6x^2y^5+y^4.
\end{eqnarray*}
Consequently $(x,y,z,1)\in \PP{3}(k)$ with $xyz\not=0$ belongs to ${\cal S}$ if and only if $[x,y,z]=[x,y,\sigma xy^2]$ for some $[x,\sigma,y]$ satisfying
$\sigma^2=y^{10}=1$, $\ell_3=\ell_4=0$. Define $\ell_1$,$\ell_2$ and $m_2$ as follows.
\begin{eqnarray*}
\ell_1&=&(\ell_3+\ell_4)x^{-2}y^{-8}=x^6-6\eta y^{-1}x^4+7y^{-2}x^2-2y^{-3},\\
\ell_2&=&(\ell_4y^{-2}\sigma+2\ell_1)y^{-4}\sigma=-5x^4+x^2(-6+14\eta)y^{-1}+(1-4\eta)y^{-2},\\
m_2&=&\frac{25}{4}\{\ell_1+\frac{1}{5}\ell_2(x^2-\frac{6+16\eta}{5}y^{-1})\}y^2=(1+\eta)(-x^2+y^{-1}),
\end{eqnarray*}
where $\eta=y^5\sigma$, which satisfies $\eta^2=1$. Evidently $\ell_3=\ell_4=0$ if and only if $\ell_2=m_2=0$.
We have found all $(x,y,z,1)\in {\cal S}$ with $xyz\not=0$.
Namely, in the case 1) $\sigma y^5=1$ with $y^{10}=1$ $(x,y,\sigma xy^2,1)$ belongs to ${\cal S}$ if and only if $x^2=y^{-1}$, for in this case
$\ell_2=(-5x^2+3y^{-1})(x^2-y^{-1})$. Since $y\in \langle \beta^2\rangle$, we get  the following $20$ singular points:
\begin{eqnarray*}
 {\cal S}_1&=&\{(\pm \beta^{-i},\beta^{2i},\pm (-1)^i\beta^{3i},1):i\in [0,9]\}=\{(\beta^{-j},\beta^{2j},\beta^{-7j},1):j\in [0,19]\}.
\end{eqnarray*}
In the another case 2) $\sigma y^5=-1$ with $y^{10}=1$ $(x,y,\sigma xy^2,1)$ belongs to ${\cal S}$ if and only if $\ell_2=0$, i.e., $x^4+4y^{-1}x^2-y^{-2}=0$.
Let $u=\sqrt{-2+\sqrt{5}}$ and $v=\sqrt{-2-\sqrt{5}}$. Then for every $y=\beta^{2i}$ ($i\in [0,9]$) $x^2=(-2\pm \sqrt{5})y^{-2i}$. Hence we have the following
$40$ singular points:
\begin{eqnarray*}
{\cal S}_2&=&\{(u\beta^{-j},\beta^{2j},-u\beta^{-7j},1):j\in [0,19]\},\ \ {\cal S}_3=\{(v\beta^{-j},\beta^{2j},-v\beta^{-7j},1):j\in [0,19]\}.
\end{eqnarray*}
We have shown that ${\cal S}={\cal S}_0+{\cal S}_1+{\cal S}_2+{\cal S}_3$. Moreover, ${\cal S}_1=G_{20}(1,1,1,1)$, ${\cal S}_2=G_{20}(u,1,-u,1)$ and
${\cal S}_3=G_{20}(v,1,-v,1)$, for $G_{10}=\langle (\diag[\beta,\beta^{-2},\beta^{7},1])\rangle$.

Let $Q_0=(0,0,0,1)$, $Q_1=(1,1,1,1)$, $Q_2=(u,1-u,1)$ and $Q_3=(v,1,-v,1)$. Evidently $H=\Paut(h)$ acts on ${\cal S}$. We claim that there exists no
$(A)\in H$ such that $(A)Q_0\not\in {\cal S}_0$ so that $H$ acts on ${\cal S}_0$. To prove the claim it suffices to show that
the tangent cone $T_{Q_0}$ of $V_p(h)$ at $Q_0$ is not isomorphic to the one $T_{Q_\ell}$ at $Q_\ell$ ($\ell=1,2,3$) by Theorem 2.9.  Since
$h(x,y,z,t)=g_0(x,y,z)t^5+{\rm lower\ terms\ of\ }t$, where $g_0(x,y,z)=-4yz^2$, $T_{Q_0}$ is the reducible affine algebraic set $V_a(yz^2)$ in $\AAA{3}(k)$.
Let $S_1\in GL_4(k)$ such that $(S_1)Q_1=Q_0$ and $h_{S_1}(x,y,z,t)=h(x+t,y+t,z+t,t)$, hence $h_{S_1}(x,y,z,t)=g_1(x,y,z)t^6+{\rm lower\ terms\ of\ }t$,
where $g_1(x,y,z)=8(-3x^2-3y^2-3z^2+xy+yz+4xz)$. Then the tangent cone $T_{Q_1}$ of $V_p(h)$ at $Q_1$ is isomorphic to $V_a(g_1)$ which is nonsingular.
We write $w$ for $u$ or $v$, so that $w^4+4w^2-1=0$. Let $S(w)\in GL_4(k)$ such that $(S(w))(w,1,-w,1)=Q_0$ and $h_{S(w)}(x,y,z,t)=h(x+wt,y+t,z-wt,t)$,
hence $h_{S(w)}(x,y,z,t)=g(x,y,z,w)t^6+{\rm lower\ terms\ of\ }t$, where
\[
 g(x,y,z,w)=c_{11}x^2+c_{22}y^2+c_{33}z^2+2c_{12}xy+2c_{13}zx+2c_{23}yz
\]
with
\begin{eqnarray*}
c_{11}&=&c_{33}=-44w^2+12,\ \ c_{22}=-28w^2+4,\\
c_{12}&=& 12w^3-16w,\ \ c_{13}=12(-7w^2+1),\ \ c_{23}=52w^3-16w.
\end{eqnarray*}
The tangent cone $T_{(w,1,-w,1)}$ of $V_p(h)$ at $(w,1,-w,1)$ is isomorphic to the affine algebraic set $V_a(g(w))$ which is nonsingular, for
 the determinant of the symmetric matrix $[c_{ij}]$ is equal to
$10w^2(224w^4+46w^2-24)\not=0$. Thus $T_{Q_0}$ is not isomorphic to none of $T_{Q_j}$ ($j\in [1,3]$), as desired.

 $H$ acts transitively on ${\cal S}_0$, for so does the subgroup $G_{80}$.  Let $H_0=\{(A)\in H:(A)Q_0=Q_0\}$, which contains $G_{20}$.
In order to see $H_0\subset G_{20}$, hence $H_0=G_{20}$, we assume $(A)\in H_0$.
Let $a=[a_1,a_2,a_3,0]$, $b=[b_1,b_2,b_3,0]$, $c=[c_1,c_2,c_3,0]$ and $d=[d_1,d_2,d_3,1]$ be the row vectors of $A$. By the condition $h_{A^{-1}}\sim h$,
we have $yz^2t^5_{A^{-1}}\sim yz^2t^5$, thus $b=[0,b_2,0,0]$ and $c=[0,0,c_3,0]$.
Writing $h_{A^{-1}}=(b_2c_3^2yz^2)t^5+\{b_2^4y^4+5b_2c_3^2(d_1x+d_2y+d_3z)\}t^4+\cdots$, we see that $\{b_2^4y^4+5b_2c_3^2(d_1x+d_2y+d_3z)\}\sim y^4$,
i.e., $d=[0,0,0,1]$. Now it follows immediately that $A$ is diagonal, hence $(A)\in G_{20}$.
Consequently $|H|=|{\cal S}_0|\ |H_0|=80$, and $H=G_{80}$. We have shown $G_{80}\subset \Paut(f)\subset H\subset G_{80}$.
\end{proof}

%%%  section 7
%%%  V_p(x^4+y^4+z^4+t^4+12xyzt)
\section {The quartic forms $x^4+y^4+z^4+t^4+\lambda\ xyzt$}
%%\section{The projective automorphism group of the form $x^4+y^4+z^4+t^4+\lambda\ xyzt$}
Let $M^\lambda=x^4+y^4+z^4+t^4+\lambda\ xyzt$. It holds that $|\Paut(V(M^{12}))|=1920$ and W. Burnside  conjectured that $V(M^{12})$ is  maximally
symmetric nonsingular quartic surface \cite[\S 272]{bur}.
In this section we study the projective automorphism group of $V(M^\lambda)$ and determine a form of type $f^{\lambda,\mu,\xi}$ that is
equivalent to $M^{12}$. Let $G_{16}=\{(\diag[a,b,c,1]):\ abc=a^4=b^4=c^4=1\}$, which is isomorphic to $C_4 \times C_4$,
$(\hat{\SSS{4}})=\{(\hat{\tau})=([e_{\tau(1)},e_{\tau(2)},e_{\tau(3)},e_{\tau(4)}]):\ \tau\in \SSS{4}\}$, and
$(\hat{\SSS{3}})=\{(\hat{\tau})=([e_{\tau(1)},e_{\tau(2)},e_{\tau(3)},e_4]):\ \tau\in \SSS{3}\}$. Then
$G_{384}=\langle G_{16},(\hat{\SSS{4}})\rangle$ and $G_{96}=\langle G_{16},(\hat{\SSS{3}})\rangle$ are groups of order $384$ and $96$, respectively, for
$\hat{\tau}^{-1}\diag[a_1,a_2,a_3,a_4]\hat{\tau}=\diag[a_{\tau(1)},a_{\tau(2)},a_{\tau(3)},a_{\tau(4)}]$ for any $\tau\in \SSS{4}$.

%\begin{eqnarray*}
%&& \hat{\SSS{3}}=\{\hat{\tau}=[e_{\tau(1)},e_{\tau(2)},e_{\tau(3)},e_4]:\ \tau\in \SSS{3}\}.
%\end{eqnarray*}
%%%%%(2) page 30, beginning of line 3
%$G_{16}$ is isomorphic to $C_4 \times C_4$.
%%%%%
%Then $G_{96}=\hat{\SSS{3}}G_{16}$ is a subgroup of order $96$ in $PGL_4(k)$, for
%$\hat{\tau}^{-1}\diag[a_1,a_2,a_3,a_4]\hat{\tau}=\diag[a_{\tau(1)},a_{\tau(2)},a_{\tau(3)},a_4]$.
%%%%%(3) page 30, line 5 after "a_4]."
$G_{96}$ is isomorphic to $((C_4 \times C_4) \rtimes C_3) \rtimes C_2$.
%%%%%
% Since $G_{96}\cap k^*E_4=\{E_4\}$, we can identify $G_{96}$ with $\pi_4(G_{96})\subset PGL_4(k)$.
% Let $B=\hat{\sigma}$, where $\sigma=(1234)\in \SSS{4}$, namely $B=[e_2,e_3,e_4,e_1]$, and
The following $C\in GL_4(k)$ is of order five, and satisfies $M^{12}_{C^{-1}}=M^{12}$.
\vspace{-3ex}
\begin{eqnarray}
\lefteqn{}\nonumber\\
C=\frac{1}{2}\left[\begin{array}{rrrr}
                   -1        &-\sqrt{-1}&-\sqrt{-1}& 1       \\
                   -\sqrt{-1}&-1        &1         &-\sqrt{-1}\\
                    1        &\sqrt{-1} &-\sqrt{-1}&1        \\
                   -\sqrt{-1}&-1        &-1        &\sqrt{-1}\end{array}\right],\
C^2=\frac{1}{2}\left[\begin{array}{rrrr}
                    -\sqrt{-1}& \sqrt{-1}& -1&-1\\
                     \sqrt{-1}& \sqrt{-1}&-1 & 1\\
                    -\sqrt{-1}&-\sqrt{-1}&-1 & 1\\
                     \sqrt{-1}&-\sqrt{-1}&-1 &-1\end{array}\right],&&\nonumber\\
C^3=\frac{1}{2}\left[\begin{array}{rrrr}
                     \sqrt{-1}&-\sqrt{-1}& \sqrt{-1}&-\sqrt{-1}\\
                    -\sqrt{-1}&-\sqrt{-1}& \sqrt{-1}& \sqrt{-1}\\
                    -1        &       -1 &        -1&       -1 \\
                    -1        &        1 &         1&       -1 \end{array}\right],\
C^4=\frac{1}{2}\left[\begin{array}{rrrr}
                    -1&         \sqrt{-1}&         1& \sqrt{-1}\\
                     \sqrt{-1}&-1        &-\sqrt{-1}&       -1 \\
                     \sqrt{-1}&         1& \sqrt{-1}&       -1 \\
                             1& \sqrt{-1}&         1&-\sqrt{-1}\end{array}\right].&&\nonumber
\end{eqnarray}
%%%% lemma 7.1
\begin{lemma} \label{lem 5.5}
$V_p(M^{-4})$ is projectively equivalent to a Kummer surface {\rm \cite[Theorem 10.3.18]{dol2}}. \\
$(1)$ The quartic form $M^\lambda(x,y,z,t)$ is singular if and only if $\lambda^4=4^4$.\\
$(2)$  $G_{384}=\Paut(M^\lambda)$ if $\lambda\not\in\{0,12\sqrt{-1}^{\ a},-4\sqrt{-1}^{\ b}:a\in [0,3],\ b\in [1,3]\}$.
%%%%%(4)
$G_{384}$ is isomorphic to $((((C_4 \times C_4) \rtimes C_2) \rtimes C_2) \rtimes C_3) \rtimes C_2$.
%%%%%
\\
$(3)$  $G_{1920}=G_{384}+(C)G_{384}+(C)^2G_{384}+(C)^3G_{384}+(C)^4G_{384}$ is a group of order $1920$, and $G_{1920}=\Paut(M^{12})$.
%%%%%(4')
$G_{1920}$ is isomorphic to $((C_2 \times C_2 \times C_2 \times C_2) \rtimes \A{5}) \rtimes C_2$.\\
%%%%%
$(4)$ Let $G_{64}=\{(\diag[a,b,c,1]) : a^4=b^4=c^4=1\}$. Then
$(\hat{\tau})^{-1}G_{64}(\hat{\tau})=G_{64}$  for $\tau\in \SSS{4}$, hence $G_{1536}=\langle G_{64},(\hat{\SSS{4}})\rangle$ is a group of order $1536$,
and $G_{1536}=\Paut(M^0)$.
\end{lemma}
\begin{proof}
Note first that if $\lambda^4=\nu^4$, then $\nu=\alpha\lambda$, where $\alpha^4=1$, so that $M^{\nu}={M^\lambda}_{D^{-1}}$ for $D=\diag[1,1,1,\alpha]$.
It is trivial that $M^0$ is nonsingular. Suppose $\lambda\not=0$ and that $V_p(M^\lambda)$ is singular at $(x,y,z,t)$.
Then $4^4xyzt=\lambda^4xyzt\not=0$, hence $\lambda^4=4^4$. Conversely, if $\lambda^4=4^4$, then $V_p(M^\lambda)$ is singular at
$(1,1,-4/\lambda,1)$.

Let $\mu=\lambda/12$, $h=12^{-4}\Hess(M^\lambda)$.   Then
\[
 h=(1-3\mu^4)x^2y^2z^2t^2+2\mu^3xyzt(x^4+y^4+z^4+t^4)-\mu^2\{x^4(y^4+z^4+t^4)+y^4(z^4+t^4)+z^4t^4\},
\]
so that
\begin{eqnarray*}
&&h_x/2=(1-3\mu^4)xy^2z^2t^2+\mu^3(5x^4yzt+y^5zt+yz^5t+yzt^5)-2\mu^2x^3(y^4+z^4+t^4),\\
&&h_y/2=(1-3\mu^4)x^2yz^2t^2+\mu^3(x^5zt+5xy^4zt+xz^5t+xzt^5)-2\mu^2y^3(x^4+z^4+t^4),\\
&&h_z/2=(1-3\mu^4)x^2y^2zt^2+\mu^3(x^5yt+xy^5t+5xyz^4t+xyt^5)-2\mu^2z^3(x^4+y^4+t^4),\\
&&h_t/2=(1-3\mu^4)x^2 y^2z^2t+\mu^3(x^5yz+xy^5z+xyz^5+5xyzt^4)-2\mu^2t^3(x^4+y^4+z^4).
\end{eqnarray*}
As is well known, $\Paut(M^\lambda)$ is a subgroup of $\Paut(h)$. Note that $G_{96}\cup (\hat{\SSS{4}})\subset \Paut(M^\lambda)$.  Assume $\mu\not=0$.
Clearly $V_p(g)$  is singular at $P_1=(1,0,0,0)$, $P_2=(0,1,0,0)$, $P_3=(0,0,1,0)$ and $P_4=(0,0,0,1)$.
If $V_p(h)$ is singular at $P=(x,y,z,t)$ with $xyzt=0$, then it can be shown easily that $P=P_i$ for some $i\in [1,4]$. Suppose
$V_p(h)$ is singular at  $P=(x,y,z,t)$ with $xyzt\not=0$. We may assume $t=1$.  Now $h_x=h_y=h_z=h_t=0$ if and only if
$xh_x=yh_y=zh_z=h_t=0$, i.e., $xh_x-h_t=yh_y-h_t=zh_z-h_t=h_t=0$, namely, putting $s=2\mu xyz-x^4-y^4-z^4$
\begin{eqnarray*}
%&& (2\mu xyz-y^4-z^4)(x^4-1)=(2\mu xyz-x^4-z^4)(y^4-1)=(2\mu xyz-x^4-y^4)(z^4-1)=0,\\
&&(x^4+s)(x^4-1)=(y^4+s)(y^4-1)=(z^4+s)(z^4-1)=0,\\
&& (1-3\mu^4)x^2 y^2z^2+\mu^3xyz(x^4+y^4+z^4+5)-2\mu^2(x^4+y^4+z^4)=0.
\end{eqnarray*}
Note that if $[x,y.z]=[\alpha_1,\alpha_2,\alpha_3]\in k^3$ is a solution, then
$[x,y,z]=[\alpha_{\tau(1)},\alpha_{\tau(2)},\alpha_{\tau(3)}]$ ($\tau\in \SSS{3}$) is also a solution.
Note that if $x^4=1$, then $y^4=1$ or $z^4=1$. To see this, assume $y^4,z^4\not=1$. Then the above equalities imply $y^4+s=z^4+s=0$, hence
$y^4=z^4$ and $2\mu xyz=(\beta+1)$, where $\beta=y^4$. Consequently $16\mu^4 \beta^2=(\beta+1)^4\not=0$.
Multiplying $4\mu^2$ after substituting $(\beta+1)/2\mu$ for $xyz$ in $h_t$, we get
\begin{eqnarray*}
0&=&(\beta+1)^2+\mu^4(\beta^2-6\beta+1)=\frac{1}{16}\beta^{-2}(\beta+1)^2(\beta-1)^4,
\end{eqnarray*}
hence $y^4=\beta=1$, a contradiction.  Assume $x^4=y^4=1$ and $z^4\not=1$. Then $2\mu xyz=2$, hence $\mu^4\gamma=1$, where
$\gamma=z^4$.  Now the condition $\mu^2h_t=0$ is satisfied.  So $P_{3,a,b}=(\sqrt{-1}^{\ a},\sqrt{-1}^{\ b},\mu^{-1}\sqrt{-1}^{\ -a-b},1)$ ($a,b\in [0,3]$) are
singular points of $V_p(h)$. Note that $\mu^4\not=1$, for $\mu^4=z^{-4}\not=1$. By the symmetry, $V_p(h)$ is also singular at
\[
P_{1,a,b}=(\mu^{-1}\sqrt{-1}^{\ -a-b},\sqrt{-1}^{\ a},\sqrt{-1}^{\ b},1),\ P_{2,a,b}=(\sqrt{-1}^{\ a},\mu^{-1}\sqrt{-1}^{\ -a-b},\sqrt{-1}^{\ b},1),
\]
where $a,b\in [0,3]$.
Assume $x^4=y^4=z^4=1$. Then, according as  $xyz=\sqrt{-1}^{\ a}$ ($a\in [0,3]$),  either $\mu=\sqrt{-1}^{\ -a}$ ($a\in [0,3]$) or
$\mu=-\frac{1}{3}\sqrt{-1}^{\ -a}$.
Indeed, putting $\nu=\sqrt{-1}^{\ a}\mu$ we have
\begin{eqnarray*}
h_t(x,y,z,1)=-\sqrt{-1}^{-2a}(3\nu^4-8v^3+6\nu^2_1)=(\nu-1)^3(3\nu+1).
\end{eqnarray*}
Next assume that none of $x^4,\ y^4,\ z^4$ is equal to $1$. Then
\[
 2\mu xyz=y^4+z^4,\ \ 2\mu xyz=x^4+z^4,\ \ 2\mu xyz=x^4+y^4,
\]
equivalently, $x^4=y^4=z^4=\mu xyz$. Now the condition $h_t=0$ reduces to $xyz=\mu^3$, hence $x^4=y^4=z^4=\mu^4=\mu xyz\not=1$. Therefore,
$V_p(h)$ is singular at $P_{4,a,b}=(\mu \sqrt{-1}^{\ a},\mu \sqrt{-1}^{\ b},\mu \sqrt{-1}^{\ -a-b},1)$ ($a,b\in [0,3]$). Thus, if $\mu^4\not=0,1,3^{-4}$, then
$V_p(h)$ has exactly $68$ singular points.  Assume $\mu^4=1$. In this case one of $x^4,y^4,z^4$ must be equal to 1. Therefore,
$x^4=y^4=z^4=1$ with $\mu xyz=1$. So, if $\mu=1$, the set of singular points of $V_p(h)$ consists of $P_i$ ($i\in [1,4]$) and
$Q_{ab}=(\sqrt{-1}^{\ a},\sqrt{-1}^{\ b},\sqrt{-1}^{\ -a-b},1)$ ($a,b\in [0,3]$).
If $\mu=-\frac{1}{3}$, the set of singular points of $V_p(h)$ consists of $P_i$, $P_{i,a,b}$ and $Q_{a,b}$ ($i\in [1,4],\ a,b\in [0,3]$).

Let $G_{P_4}=\{(A)\in \Paut(g):\ (A)P_4=P_4\}$. We will show $G_{P_4}=G_{96}$, to be precise  $G_{P_4}=\pi_4(G_{96})$, for $\mu$ satisfying
$\mu^4\not=0$. Clearly $G_{P_4}\supset G_{96}$. Assume $(A)\in G_{P_4}$, where $A=[a_{ij}]\in GL_4(k)$ with the 4-th column $e_4$. Comparing the
coefficients of $t^5$ in $h_{A^{-1}}$ and $h$, we see that $(\sum_{j=1}^{3}a_{1j}x_j)(\sum_{j=1}^{3}a_{2j}x_j)(\sum_{j=1}^{3}a_{3j}x_j)\sim x_1x_2x_3$.
Replacing $A$ by $A[e_{\tau(1)},e_{\tau(2)},e_{\tau(3)},e_4]$ for some $\tau\in \SSS{3}$, we may assume $a_{ij}=0$ ($i,j\in [1,3],\ i\not=j$).
Comparing the coefficients of $t^4$ in $h_{A^{-1}}$ and $h$, we see that $a_{4j}=0$ ($j\in [1,3]$). Now the nonsingular  matrix
$A=\diag[a,b,c,1]$ satisfies $h_{A^{-1}}\sim h$ if and only if $a^4=b^4=c^4=abc=1$. Thus $G_{P_4}\subset G_{96}$.

Let ${\cal P}_0=\{P_1,P_2,P_3,P_4\}$, ${\cal P}_i=\{P_{i,a,b}:\ a,b\in [0,3]\}$ ($i\in [1,4]$), and ${\cal Q}=\{Q_{ab}:\ a,b\in [0,3]\}$. To complete
the proof we first assume $\mu^4\not=0,3^{-4},1$, hence $\cup_{i=0}^4{\cal P}_i$ is the set of all singular points of $V_p(h)$.
$G_{384}$ acts transitively on $\cup_{i=1}^4{\cal P}_i$ and ${\cal P}_0$, respectively.
Any $(A)\in \Paut(g)$ maps ${\cal P}_0$ into itself. To verify this it suffices to show that $(A)P_4=P_{4,0,0}$ is impossible.
Note that there exists $T\in GL_4(k)$ such that $T[\mu,\mu,\mu,1]=[0,0,0,1]$ and $h_{T}(x,y,z,t)=h(x+\mu t,y+\mu t,z+\mu t,t)$. Since
\begin{eqnarray*}
&&h(x,y,z,t)=t^5(2\mu^3xyz)+{\rm lower\ terms\ of\ }t, \\
&&h_{T}(x,y,z,t)=t^6\{\mu^4(\mu^4-1)\{5(x^2+y^2+z^2)-6(xy+yz+zx)\}+{\rm lower\ terms\ of\ }t,
\end{eqnarray*}
the tangent cone to $V_p(h)$ at $P_4$ (resp. at $P_{4,0,0}$) is the affine algebraic set $V_a(xyz)$
(resp. $V_a(5(x^2+y^2+z^2)-6(xy+yz+zx))$). Since these two affine algebraic sets are not isomorphic,
no $(A)\in \Paut(h)$  satisfies $(A)P_4=P_{4,0,0}$  by Theorem 2.9. Since
$\Paut(h)$ acts transitively on the four-point set ${\cal P}_0$ and $|G_{P_4}|=96$, it follows that $|\Paut(h)|=384$, hence $\Paut(h)=G_{384}$.
Now $G_{384}\subset \Paut(M^\lambda)\subset \Paut(g)$, i.e., $\Paut(M^\lambda)=G_{384}$.

Secondly assume $\mu=-\frac{1}{3}$. Then $G_{384}$ acts transitively on $\cup_{i=1}^4{\cal P}_i$, ${\cal P}_0$ and ${\cal Q}$, respectively. There exists
$S\in GL_4(k)$ such that $S[1,1,1,1]=[0,0,0,1]$ and $h_{S}(x,y,z,t)=h(x+t,y+t,z+t,t)$.
Since
$h_{S}(x,y,z,t)=\frac{16}{27}\{-3(x^2+y^2+z^2)+2(xy+yz+zx)\}t^6+{\rm lower\ terms\ of\ }t$, the tangent cone to $V_p(h)$ at $Q_{00}$ is
$V_a(-3(x^2+y^2+z^2)+2(xy+yz+zx))$. The tangent cone to $V_p(h)$ at $P_{4,0,0}$ is $V_a(5(x^2+y^2+z^2)-6(xy+yz+zx))$. Thus no $(A)\in \Paut(h)$ transforms
$P_4$ into $\{P_{4,0,0}, Q_{00}\}$. Hence, as in the case $\mu^4\not\in \{0,3^{-4},1\}$, $\Paut(h)=G_{384}$.

Next assume $\mu=1$.
$\{(C)^i(B^\ell)P_4:\ \ell\in [0,3]\}$ ($i\in [0,4]$)
consists of four points in $\PP{3}(k)$ with homogeneous coordinates $c'_{ij}$, where $c'_{ij}$ denote the transposed of the $j$-th column of $C^i$.
One can verify immediately that $\{(C)^i(B^\ell)P_4:\ \ell\in [0,3],\ i\in [0,4]\}={\cal P}_0\cup {\cal Q}$. The right-hand side is the set of
all singular points of $V_p(h)$. So $\Paut(h)$ acts transitively on ${\cal P}_0\cup {\cal Q}$, and $|G_{P_4}|=96$. Consequently $|\Paut(h)|=20\cdot 96$.
Since $G_{1920}\subset \Paut(M^{12})\subset \Paut(h)$, it follows that $G_{1920}=\Paut(M^{12})=\Paut(h)$.

Finally  assume $\mu=0$.
It is clear that $G_{1536}\subset \Paut(M^0)$. Let $(A)\in \Paut(M^0)$ and $h=12^{-4}\Hess(M^0)=x^2y^2z^2t^2$. Since $(A)\in \Paut(h)$,
$(A)$ takes the form $(\diag[a,b,c,1])(\hat{\tau})$, where $a^4=b^4=c^4=1$ and $\tau\in \SSS{4}$. Hence $G_{1536}\supset \Paut(M^0)$.
\end{proof}

  Let $H_1=K_1=\diag[1,1,1,1]=[e_1,e_2,e_3,e_4]$ and
\begin{eqnarray*}
&&H_2=[-e_1,-e_2,e_3,e_4],\ H_3=[-e_1,e_2,-e_3,e_4],\  H_4=[e_1,-e_2,-e_3,e_4]=H_2H_3,\\
&&K_2=\hat{(12)(34)}=[e_2,e_1,e_4,e_3],\  K_3=\hat{(13)(24)}=[e_3,e_4,e_1,e_2],\ K_4=\hat{(14)(23)}=K_2K_3.
\end{eqnarray*}
Clearly ${\cal H}=\{H_i:\ i\in [1,4]\}$ is a subgroup of $GL_4(k)$ isomorphic to $\Z_2\times \Z_2$. Permutations $\sigma_1=id$, $\sigma_2=(12)(34)$,
$\sigma_3=(13)(24)$, and $\sigma_4=(14)(23)$ form a subgroup of $\SSS{4}$ isomorphic to $\Z_2\times \Z_2$. Let ${\cal K}=\{\hat{\sigma_i}:\ i\in [1,4]\}$.
In addition
\begin{eqnarray*}
&& \hat{\sigma_2}^{-1}H_2\hat{\sigma_2}=H_2,\ \hat{\sigma_2}^{-1}H_3\hat{\sigma_2}=-H_3,\\
&& \hat{\sigma_3}^{-1}H_2\hat{\sigma_3}=-H_2,\ \hat{\sigma_3}^{-1}H_3\hat{\sigma_3}=H_3.
\end{eqnarray*}
Thus, the map $\xi:{\cal H}\times {\cal K}\rightarrow PGL_4(k)$ defined by $\xi(H_i,K_j)=(H_iK_j)$ is a group homomorphism such that $\Ima \xi$ is abelian.
Since $\xi$ is injective, ${\cal A}_{16}=\{(H_i)(K_j):\ i,j\in [1,4]\}$ is an abelian group of order 16.  ${\cal A}_{16}$ is isomorphic to
$C_2 \times C_2 \times C_2 \times C_2$.

%%% lemma 7.2
\begin{lemma}  ${\cal A}_{16}\lhd G_{1920}$.
\end{lemma}
\begin{proof}
By Lemma 7.1 the group $G_{1920}$ is generated by $(G_{16})$, $(\hat{\SSS{3}})$, $(\hat{\sigma})$ and $(C)$, where $\sigma=(1234)\in \SSS{4}$.
(1) Let $D=\diag[a,b,c,1]$ with $a^4=b^4=c^4=abc=1$. Then $D$ and $H_i$ commute. Besides, $D^{-1}K_2D=c\diag[b^2,a^2,a^2b^2,1]K_2$, and
$D^{-1}K_3D=b\diag[c^2,a^2c^2,a^2,1]K_3$. Thus $(D)^{-1}{\cal A}_{16}(D)={\cal A}_{16}$. (2) The permutations $\tau_1=(12)$ and $\tau_2=(23)$
generate $\SSS{3}$. Recall that $\hat{\tau}^{-1}\diag[a_1,a_2,a_3,a_4]\hat{\tau}=\diag[a_{\tau(1)},a_{\tau(2)},a_{\tau(3)},a_{\tau(4)}]$. Therefore,
$(\hat{\tau}_i)^{-1}({\cal H})(\hat{\tau}_i)=({\cal H})$ for $i\in [1,2]$. In addition
\[
 \tau_1^{-1}\sigma_2\tau_1=\sigma_2,\ \tau_1^{-1}\sigma_3\tau_1=\sigma_4,\ \tau_2^{-1}\sigma_2\tau_2=\sigma_3,\ \tau_2^{-1}\sigma_3\tau_2=\sigma_2.
\]
So $(\hat{\tau})^{-1}{\cal A}_{16}(\hat{\tau})={\cal A}_{16}$ for any $\tau\in \SSS{3}$.
(3) Note that $\sigma^{-1}\sigma_2\sigma=\sigma_4$ and $\sigma^{-1}\sigma_3\sigma=\sigma_3$. Consequently
$(\hat{\sigma})^{-1}{\cal A}_{16}(\hat{\sigma})={\cal A}_{16}$. (4) Finally $(C)^{-1}{\cal A}_{16}(C)={\cal A}_{16}$, for
\begin{eqnarray*}
&&C^{-1}H_2C=\sqrt{-1}H_2K_4,\ C^{-1}H_3C=\sqrt{-1}H_3K_2,\\
&&C^{-1}K_2C=-\sqrt{-1}H_2K_4,\ C^{-1}K_3C=-\sqrt{-1}H_4K_2. \;\;\;\;\;\;
\end{eqnarray*}
\end{proof}

%%%%%%%%%%%%%%%%%%%%%%%%%%%%%%%%%%%%%%%%%%%%%%%%%%%%%
Let $\sigma=(13)(25)$, $\tau=(1234)\in \SSS{5}$, and $\rho=\sigma\tau$. Clearly $\ord(\rho)=6$.
%% lemma 7.3
\begin{lemma}
In $\SSS{5}$
\begin{eqnarray*}
&&(34)(152)=\sigma\tau,\ \ (34)=\rho^3,\ \ (152)^{-1}=\rho^{2},\ \ (23)=\tau^{-1}(34)\tau,\ \ (12)=\tau^{-1}(23)\tau,\\
&&(25)=(152)^{-1}(12)(152),\ \ (35)=(23)(25)(23),\ \ (45)=(34)(35)(34).
\end{eqnarray*}
In particular $\sigma$ and $\tau$ generate $\SSS{5}$.
\end{lemma}
\begin{proof}
As is well known, $\SSS{5}$ are generated by $(12),\ (23),\ (34)$ and $(45)$.
\end{proof}

Let $\tau_j=(jj+1)\in \SSS{5}$ ($j\in [1,4]$). As is well known , they generate $\SSS{5}$, and the defining relations with respect to
them \cite{suz} are
\[
 \tau_j^2=1\ (j\in [1,4]),\  (\tau_j\tau_{j+1})^3\ (j\in [1,3]),\  (\tau_i\tau_j)^2=1\ (|i-j|\geq 2).
\]
There exists a group homomorphism $\chi$ of $\SSS{5}$ into $GL_4(k)$ such that $\chi(\tau_i)=T_i$ ($i\in [1,4]$), where $T_i$ are given
as follows. Hence we have a group homomorphism $(\chi):\ \SSS{5}\rightarrow PGL_4(k)$, where $(\chi)(\tau)=(\chi(\tau))$ for any $\tau\in \SSS{5}$.
We note that ${M^{12}}_{T_i^{-1}}=M^{12}$ for every $i\in [1,4]$.
\begin{eqnarray*}
&&T_1=[e_3,e_2,e_1,e_4],\ T_2=[\sqrt{-1}e_2,-\sqrt{-1}e_1,e_3,e_4],\ T_3=[e_1,\sqrt{-1}e_4,e_3,-\sqrt{-1}e_2],\\
&&T_4=\frac{1}{2}\left[\begin{array}{rrrr}
1&\sqrt{-1}&-1&\sqrt{-1}\\
-\sqrt{-1}&1&-\sqrt{-1}&-1\\
-1&\sqrt{-1}&1&\sqrt{-1} \\
-\sqrt{-1}&-1&-\sqrt{-1}&1\end{array}\right].
\end{eqnarray*}

In fact we have the following lemma. Since $\A{5}$ is the minimal normal subgroup of $\SSS{5}$ \cite[chapter 3,(2.10)]{suz},\\
$\chi$ and $(\chi)$ are injective.
%% lemma 7.4
\begin{lemma}
Let $T_{ij}=T_iT_j$. Then
\begin{eqnarray*}
&&\ord(T_{i})=2\ (i\in [1,4]),\ \ord(T_{jj+1})=3\ (j\in [1,3]),\ \ord(T_{ij})=2\ (|i-j|\geq 2).
\end{eqnarray*}
\end{lemma}
\begin{proof}
One can easily verify these equalities.
\end{proof}

%%% proposition 7.5
\begin{proposition} $G_{1920}=(\chi(\SSS{5})){\cal A}_{16}$, and $G_{1920}/{\cal A}_{16}\cong (\chi(\SSS{5}))$.
\end{proposition}
\begin{proof}
Let $G=G_{1920}$, and $H=(\chi(\SSS{5}))$, which is a subgroup of $G$ of order $120$. Since ${\cal A}_{16}\lhd G$, the inverse image
$(\chi)^{-1}({\cal A}_{16}\cap H)\lhd \SSS{5}$ is $\{id\}$, $\A{5}$ or $\SSS{5}$. Since any non-unit element of ${\cal A}_{16}$ is of order two,
${\cal A}_{16}\cap H=\{e_{G}\}$.
\end{proof}

Since the representations $\chi$ and $\rho_V$ of $\SSS{5}$, which is generated by $(12)$ and $(12345)$, are equivalent, and irreducible, there
exists uniquely an $S\in GL_4(k)$ such that $S^{-1}T_1S=\tau_{12}$ and
$S^{-1}T_1T_2T_3T_4S=\tau_{12345}=\diag[\varepsilon,\varepsilon^2,\varepsilon^4,\varepsilon^3]$ up to constant multiplication (see Lemma 3.9 for
$\tau_{12}$).
We define $c_i$ ($i\in [1,4]$) as follows, and let $\sigma=(1234)$.
\begin{eqnarray*}
\left[\begin{array}{c}
c_1\\
c_2\\
c_3\\
c_4\end{array}\right]=\left[\begin{array}{l}
                            -1\\
                            1-\varepsilon^2+\varepsilon^4+\sqrt{-1}(-1-\varepsilon+\varepsilon^3)\\
                            -\varepsilon+\varepsilon^4+\sqrt{-1}(\varepsilon^2+\varepsilon^3)\\
                            -\varepsilon^3-\varepsilon^4-\sqrt{-1}(\varepsilon+\varepsilon^2)\end{array}\right].
\end{eqnarray*}
Then we may assume $S=C\hat{\sigma}^2\diag[\varepsilon^4,\varepsilon^3,\varepsilon,\varepsilon^2]$, where, putting $D=\diag[1,\sqrt{-1},1,-1]$,
\begin{eqnarray*}
C&=&D\left[\begin{array}{cccc}
               c_1& c_2& c_3& c_4\\
               c_3& c_4& c_1& c_2\\
               c_4& c_1& c_2& c_3\\
               c_2& c_3& c_4& c_1\end{array}\right].
\end{eqnarray*}
In fact $T_1T_2T_3T_4S=S\tau_{12345}$.  In order to see $S^{-1}T_1S=\tau_{12}$ we introduce the conjugation $\bar{}$, an automorphism of the field
$\QQ(\sqrt{-1},\varepsilon)$ such that $\bar{\sqrt{-1}}=-\sqrt{-1}$ and $\bar{\varepsilon}=\varepsilon^4$, regarding $\QQ(\sqrt{-1},\varepsilon)$ as
a subfield of the complex number field $\C$, for $\QQ(\sqrt{-1},\varepsilon)=\QQ(\delta)$, where $\ord(\delta)=10$ \cite[p.278]{lan}. Note that $\bar{c}_j=(\sqrt{-1})^{j-1}c_j$ for $j\in [1,4]$. Computing $c_2c_3$, $c_3c_4$, we see that
$c_1\bar{c}_2+c_2\bar{c}_3+c_3\bar{c}_4+c_4\bar{c}_1=0$, hence the conjugate $c_1\bar{c}_4+c_2\bar{c}_1+c_3\bar{c}_2+c_4\bar{c}_3$ also vanishes.
Computing $c_j^2$ ($j\in [2,4]$),  we see
\[
 \sum_{i=1}^4c_i\bar{c}_i=10(1-\varepsilon^2-\varepsilon^3)+\sqrt{-1}(-8\varepsilon-6\varepsilon^2+6\varepsilon^3+8\varepsilon^4),\ \
 \gamma(\sum_{i=1}^4c_i\bar{c}_i)=1,
\]
where $10\gamma=2-\varepsilon-\varepsilon^4+\sqrt{-1}(\varepsilon^2-\varepsilon^3)$.
 To sum up  $\sum_{i=1}^4c_1\bar{c}_{\sigma^{j}(i)}=\gamma^{-1}\delta_{0,j}$ for $j\in [0,3]$.
Denoting the conjugate matrix $[b_{ij}]$ of $A=[a_{ij}]\in M_4(\QQ(\sqrt{-1},\varepsilon)$ by $A^*$, where $b_{ij}=\bar{a}_{ji}$,
we have $(D^{-1}C)(D^{-1}C)^*=\gamma^{-1}E_4$, i.e., $(D^{-1}C)^{-1}=\gamma (D^{-1}C)^* $. Now note that $T_1$ and $D$ commute.\\
 So it suffices to show
that $$\gamma (D^{-1}C)\hat{\sigma}^2T_{12}(D^{-1}C)^*=T_1,$$
where 
$T_{12}=\diag[\varepsilon^4,\varepsilon^3,\varepsilon,\varepsilon^2]\tau_{12}\diag[\varepsilon,\varepsilon^2,\varepsilon^4,\varepsilon^3].$
The $[1,1]$-component of $(D^{-1}C)T_{12}$ takes the form
\begin{eqnarray*}
&&\eta(c_1+2\varepsilon c_2-\varepsilon^3c_3-2\varepsilon^2 c_4)-\varepsilon c_2+\varepsilon^3c_3+\varepsilon^2c_4=c_4.
\end{eqnarray*}
In this way $\gamma (D^{-1}C)\hat{\sigma}^2T_{12}(D^{-1}C)^*$ turns out to be equal to
\begin{eqnarray*}
\gamma \left[\begin{array}{cccc}
c_4&c_1&c_2&c_3\\
c_3&c_4&c_1&c_2\\
c_1&c_2&c_3&c_4\\
c_2&c_3&c_4&c_1\end{array}\right]
\left[\begin{array}{cccc}
\bar{c}_1&\bar{c}_3&\bar{c}_4&\bar{c}_2\\
\bar{c}_2&\bar{c}_4&\bar{c}_1&\bar{c}_3\\
\bar{c}_3&\bar{c}_1&\bar{c}_2&\bar{c}_4\\
\bar{c}_4&\bar{c}_2&\bar{c}_3&\bar{c}_1\end{array}\right]=T_1
\end{eqnarray*}
by the equalities $\sum_{i=1}^4c_1\gamma\bar{c}_{\sigma^{j}(i)}=\delta_{0,j}$ ($j\in [0,3]$).
%%% theorem 7.6
\begin{theorem}
Let $S$ be as above, $M^{12}=x^4+y^4+z^4+t^4+12xyzt$, $F_0=x^3y+y^3z+z^3t+t^3x+3xyzt$, and $F_1=x^2z^2+y^2t^2+2xyzt$. Then
\[
M^{12}_{S^{-1}}=80p\{ F_0-\frac{3}{4}(1+\sqrt{-1})F_1\},
\]
where  $p=3+20\varepsilon+28\varepsilon^2+16\varepsilon^3+\sqrt{-1}(17+20\varepsilon+4\varepsilon^2-8\varepsilon^3)$.
\end{theorem}
\begin{proof}
Since $\hat{\sigma}$ and $\tau_{12345}$ leaves $F_j$ ($j\in [0,1]$) invariant, and
$\diag[\varepsilon^4,\varepsilon^3,\varepsilon,\varepsilon^2]=\tau_{12345}^4$, it suffices to show that
$M^{12}_{C^{-1}}=80p\{ F_0-\frac{3}{4}(1+\sqrt{-1})F_1\}$.
Let $F(x,y,z,t)=M_{D^{-1}}^{12}=x^4+y^4+z^4+t^4-12\sqrt{-1}xyzt$, $x_i,a_i$ ($i\in [1,4]$) be indeterminates,
$y_j=\sum_{i=1}^4a_{\sigma^{(j)}(i)}x_i$, $z_j=\sum_{i=1}^4c_{\sigma^{(j)}(i)}x_i$
($j\in \Z$), and $G(x,a)=F(y_1,y_2,y_3,y_4)=\sum_{i_1+\cdots +i_4=4}g_{i_1\cdots i_4}(a_1,...,a_4)x_1^{i_1}\cdots x_4^{i_4}$. Evidently
$y_j=\sum_{i=1}^4 a_ix_{\sigma^{-j}(i)}$, and $y_j=y_{j'}$ if and only if $j\equiv j'$ ($\modd 4$).
Since $k[x_1,...,x_4,a_1,...,a_4]=(k[a])[x]=(k[x])[a]$, $\SSS{4}$ acts on $k[x,a]$ in two ways: for $\tau\in \SSS{4}$ and $f(x,a)\in k[x,a]$
$(\tau\cdot f)(x,a)=f(x_{\tau(1)},...,x_{\tau(4)},a)$ or $(\tau\odot f)(x,a)=f(x,a_{\tau(1)},...,a_{\tau(4)})$.
Clearly $\tau F(x_1,...,,x_4)=F(x_1,...,x_4)$ for any $\tau\in \SSS{4}$. Moreover,  $\sigma\cdot y_j=y_{j-1}$ and $\sigma\odot y_j=y_{j+1}$. Since
$\{y_j:j\in [1,4]\}=\{y_{j-1}:j\in [1,4]\}=\{y_{j+1}:j\in [1,4]\}$, it holds that
$(\sigma\cdot G)(x,a)=F(y_0,...,y_3)=G(x,a)$, and $(\sigma\odot G)(x,a)=F(y_2,...,y_5)=G(x,a)$.
In particular $g_{j_1\cdots j_4}=g_{i_1\cdots i_4}$ if $[j_1,...,j_4]=[i_1,...,i_4]\sigma^\ell$ for some $\ell\in [0,3]$.
Note that $M^{12}_{C^{-1}}(x_1,...,x_4)=F(z_4,z_2,z_3,z_1)=G(x,c_1,c_2,c_3,c_4)$.

We introduce 10 polynomials in indeterminates $a,b,c,d$ as follows.
\begin{eqnarray*}
&&\left[\begin{array}{c}
      f_{4000}(a,b,c,d)\\
      f_{1111}(a,b,c,d)\\
      f_{2020}(a,b,c,d)\\
      f_{2200}(a,b,c,d)\end{array}\right]=\left[\begin{array}{l}
                                       a^4+b^4+c^4+d^4\\
                                       abcd\\
                                       a^2c^2+b^2d^2\\
                                       a^2b^2+b^2c^2+c^2d^2+d^2a^2\end{array}\right],\\
&&\left[\begin{array}{c}
      f_{3100}(a,b,c,d)\\
      f_{3010}(a,b,c,d)\\
      f_{3001}(a,b,c,d)\end{array} \right]=\left[\begin{array}{l}
                                       a^3b+b^3c+c^3d+d^3a\\
                                       a^3c+b^3d+c^3a+d^3b\\
                                       a^3d+b^3a+c^3b+d^3c\end{array}\right],\\
&&\left[\begin{array}{c}
      f_{2110}(a,b,c,d)\\
      f_{2101}(a,b,c,d)\\
      f_{2011}(a,b,c,d)\end{array}\right]=\left[\begin{array}{l}
                                       a^2bc+b^2cd+c^2da+d^2ab\\
                                       a^2bd+b^2ca+c^2db+d^2ac\\
                                       a^2cd+b^2da+c^2ab+d^2bc\end{array}\right].
\end{eqnarray*}
Then, writing $g_{i_1\cdots i_4}(a,b,c,d)$ as $g_{i_1\cdots i_4}$, we have
\begin{eqnarray*}
\left[\begin{array}{c}
             g_{4000}\\
             g_{3100}\\
             g_{3010}\\
             g_{3001}\\
             g_{2200}\\
             g_{2020}\\
             g_{2110}\\
             g_{2101}\\
             g_{2011}\\
             g_{1111}\end{array}\right]=\left[\begin{array}{l}
                                           f_{4000}-12\sqrt{-1}f_{1111}\\
                                           4f_{3100}-12\sqrt{-1}f_{2110}\\
                                           4f_{3010}-12\sqrt{-1}f_{2101}\\
                                           4f_{3001}-12\sqrt{-1}f_{2011}\\
                                           6f_{2200}-12\sqrt{-1}(f_{2020}+f_{2101})\\
                                           12f_{2020}-12\sqrt{-1}(f_{2200}+2f_{1111})\\
                                           12f_{2110}-12\sqrt{-1}(f_{3100}+2f_{2011})\\
                                           12f_{2101}-12\sqrt{-1}(f_{3010}+f_{2200}+4f_{1111})\\
                                           12f_{2011}-12\sqrt{-1}(f_{3001}+2f_{2110})\\
                                           96f_{1111}-12\sqrt{-1}(f_{4000}+2f_{2020}+4f_{2101})\end{array}\right].
\end{eqnarray*}
One can evaluate $f'_{i_1\cdots i_4}=f_{i_1\cdots i_4}(c_1,c_2,c_3,c_4)$ and $g'_{i_1\cdots i_4}=g_{i_1\cdots i_4}(c_1,c_2,c_3,c_4)$, using the
following equalities
{\scriptsize
\begin{eqnarray*}
&&\left[\begin{array}{c}
            c_2^2\\
           c_3^2\\
           c_4^2\end{array}\right]
=\left[\begin{array}{ll}
             -5\varepsilon-3\varepsilon^2+3\varepsilon^3+5\varepsilon^4 &+\sqrt{-1}(-4+6\varepsilon^2+6\varepsilon^3) \\
              -3+2\varepsilon^2+2\varepsilon^3&+\sqrt{-1}(2+4\varepsilon+4\varepsilon^2) \\
              1+2\varepsilon+2\varepsilon^2&+\sqrt{-1}(2-2\varepsilon^2-2\varepsilon^3)  \end{array}\right],\\
&&\left[\begin{array}{c}
            c_2c_3\\
            c_2c_4\\
            c_3c_4\end{array}\right]
=\left[\begin{array}{ll}
             -3-4\varepsilon+3\varepsilon^3+\varepsilon^4 &+\sqrt{-1}(-3+\varepsilon+3\varepsilon^2-4\varepsilon^4) \\
             2-3\varepsilon^2-3\varepsilon^3 &+\sqrt{-1}(-3\varepsilon-2\varepsilon^2+2\varepsilon^3+3\varepsilon^4) \\
             2-\varepsilon^2+3\varepsilon^4 &+\sqrt{-1}(-2-3\varepsilon+\varepsilon^3)   \end{array}\right].
\end{eqnarray*}
}
Namely,
{\scriptsize
\begin{eqnarray*}
&&\left[\begin{array}{c}
        f'_{4000}\\
        f'_{1111}\\
        f'_{2020}\\
        f'_{2200}                \end{array}\right]
=\left[\begin{array}{ll}
     -84+144\varepsilon^2+144\varepsilon^3 &+\sqrt{-1}(120\varepsilon+72\varepsilon-72\varepsilon^3-120\varepsilon^4)\\
      10\varepsilon+6\varepsilon^2-6\varepsilon^3-10\varepsilon^4&+\sqrt{-1}(7-12\varepsilon^2-12\varepsilon^3)\\
    28-48\varepsilon^2-48\varepsilon^3 &+\sqrt{-1}(-40\varepsilon-24\varepsilon^2+24\varepsilon^3+40\varepsilon^4)\\
    40\varepsilon+24\varepsilon^2-24\varepsilon^3-40\varepsilon^4 &+\sqrt{-1}(28-48\varepsilon^2-48\varepsilon^3)\end{array}\right],\\
&&\left[\begin{array}{c}
        f'_{3100}\\
        f'_{3010}\\
        f'_{3001}                 \end{array}\right]
=\left[\begin{array}{ll}
    -67+69\varepsilon+133\varepsilon^2+37\varepsilon^3-91\varepsilon^4&+\sqrt{-1}(87+111\varepsilon-17\varepsilon^2-113\varepsilon^3-49\varepsilon^4)\\
    -60\varepsilon-36\varepsilon^2+36\varepsilon^3+60\varepsilon^4 &+\sqrt{-1}(-42+72\varepsilon^2+72\varepsilon^3)\\
    69+87\varepsilon-9\varepsilon^2-81\varepsilon^3-33\varepsilon^4 &+\sqrt{-1}(49-53\varepsilon-101\varepsilon^2-29\varepsilon^3+67\varepsilon^4)\end{array}\right],\\
&&\left[\begin{array}{c}
        f'_{2110}\\
        f'_{2101}\\
        f'_{2011}                 \end{array}\right]
=\left[\begin{array}{ll}
-40-52\varepsilon+12\varepsilon^2+60\varepsilon^3+28\varepsilon^4 &+\sqrt{-1}(-30+38\varepsilon+70\varepsilon^2+22\varepsilon^3-42\varepsilon^4)\\
 -14+24\varepsilon^2+24\varepsilon^3  &+\sqrt{-1}(20\varepsilon+12\varepsilon^2-12\varepsilon^3-20\varepsilon^4)\\
 12-22\varepsilon-38\varepsilon^2-14\varepsilon^3+18\varepsilon^4 &+\sqrt{-1}(-22-28\varepsilon+4\varepsilon^2+28\varepsilon^3+12\varepsilon^4) \end{array}\right],
\end{eqnarray*}
}
and
{\scriptsize
\begin{eqnarray*}
\left[\begin{array}{c}
             g'_{4000}\\
             g'_{3100}\\
             g'_{3010}\\
             g'_{3001}\\
             g'_{2200}\\
             g'_{2020}\\
             g'_{2110}\\
             g'_{2101}\\
             g'_{2011}\\
             g'_{1111}\end{array}\right]=\left[\begin{array}{l}
                                           0\\
                                           80\{3+20\varepsilon+28\varepsilon^2+16\varepsilon^3+\sqrt{-1}(17+20\varepsilon+4\varepsilon^2-8\varepsilon^3)\}\\
                                           0\\
                                           0\\
                                           0\\
                                           120\{7-12\varepsilon^2-12\varepsilon^3+\sqrt{-1}(-10\varepsilon-6\varepsilon^2+6\varepsilon^3+10\varepsilon^4)\}\\
                                           0\\
                                           0\\
                                           0\\
                                           240\{10\varepsilon+6\varepsilon^2-6\varepsilon^3-10\varepsilon^4+\sqrt{-1}(7-12\varepsilon^2-12\varepsilon^3)\}\end{array}\right].
\end{eqnarray*}
}
Consequently \[
	      M^{12}_{C^{-1}}(x,y,z,t)=g'_{3100}(x^3y+y^3z+z^3t+t^3x)+g'_{2020}(x^2z^2+y^2t^2)+g'_{1111}xyzt.
	     \]
Since $g'_{3100}=80p$ and $2p^{-1}=3+4\varepsilon^3-4\varepsilon^4+\sqrt{-1}(3-4\varepsilon+4\varepsilon^2)$, $g'_{2020}/g'_{3100}$ is equal to
$-\frac{3}{4}(1+\sqrt{-1})$. In addition the equality $3g'_{3100}+2g'_{2020}=g'_{1111}$ yields $g'_{1111}/g'_{3100}=\frac{3}{2}(1-\sqrt{-1})$. Hence
$ M^{12}_{C^{-1}}(x,y,z,t)=80p\{F_0(x,y,z,t)-\frac{3}{4}(1+\sqrt{-1})F_1(x,y,z,t)\}$, as desired.
\end{proof}

\vspace{\baselineskip}

%
%
%
%Prof.ssa Fernanda Pambianco\\\frac{\sqrt{5}+2}{2}+ (\sqrt{5}+1)\eta'$,
%Dipartment di Matematica e Informatica\\
%Universita Degli Studi di Perugia\\
%Via Vanviteli 1-06123 Perugia\\
%ITALY\\}
%\ \\
%From:\\
%Hitoshi Kaneta\\
 %Palace Mozu 301, Mozu-Ume-Machi 3-34-8\\
%Kita-Ku, Sakai-Shi, Osaka Prefecture 591-8032\\
%JAPAN\\
%
%Hitoshi Kaneta\\
%Palace Mozu 301, Mozu-Ume-Machi 3-34-8\\
%Kita-Ku, Sakai-Shi, Osaka Prefecture 591-8032\\
%JAPAN\\

\end{document}